\documentclass[twoside,11pt]{article}

\usepackage[abbrvbib]{jmlr2e}

\usepackage{lastpage}
\jmlrheading{22}{2021}{1-\pageref{LastPage}}{11/20; Revised 8/21}{10/21}{20-1259}{Constantin Christof}

\ShortHeadings{On the Optimization Landscape of Training Problems with Squared Loss}{Christof}
\firstpageno{1}

\usepackage{tikz}
\usetikzlibrary{intersections}
\usetikzlibrary{arrows.meta}
\usetikzlibrary{patterns}
\usepackage{subcaption}
\usepackage{epstopdf}
\usepackage{enumitem}
\usepackage{dsfont}
\usepackage{float}
\usepackage{hyperref}
\usepackage{amsmath}
\usepackage[capitalise]{cleveref}

\newcommand{\sgn}{\operatorname{sgn}}

\newcommand\MM{\mathcal{M}}
\newcommand\NN{\mathcal{N}}

\newcommand\XX{\mathcal{X}}
\newcommand\YY{\mathcal{Y}}

\newcommand\R{\mathbb{R}}

\DeclareFontFamily{OT1}{pzc}{}
\DeclareFontShape{OT1}{pzc}{m}{it}{<-> s * [1.125] pzcmi7t}{}
\DeclareMathAlphabet{\mathpzc}{OT1}{pzc}{m}{it}

\newcommand\oo{\mathpzc{o}}

\newcommand\xx{\mathpzc{x}}
\newcommand\yy{\mathpzc{y}}
\newcommand\zz{\mathpzc{z}}
\newcommand\closure{\mathrm{cl}}
\renewcommand\span{\mathrm{span}}

\newtheorem{assumption}[theorem]{Assumption}
\crefname{assumption}{Assumption}{Assumptions} 
\Crefname{assumption}{Assumption}{Assumptions} 
\newtheorem{setting}[theorem]{Setting}
\crefname{setting}{Setting}{Settings} 
\Crefname{setting}{Setting}{Settings} 

\DeclareMathOperator*{\argmin}{arg\,min}
\DeclareMathOperator*{\dist}{dist}
\DeclareMathOperator*{\conv}{conv}

\newcommand\Item[1][]{%
  \ifx\relax#1\relax  \item \else \item[#1] \fi
  \abovedisplayskip=0pt\abovedisplayshortskip=0pt~\vspace*{-\baselineskip}}

\begin{document}

\title{On the Stability Properties and the Optimization Landscape of Training Problems with Squared Loss for
Neural Networks and General Nonlinear Conic Approximation Schemes}

\author{\name Constantin Christof \email christof@ma.tum.de \\
       \addr Technische Universit\"{a}t M\"{u}nchen\\
	   Chair of Optimal Control\\
	   Center for Mathematical Sciences, M17\\
       Boltzmannstraße 3, 85748 Garching, Germany}

\editor{Animashree Anandkumar}

\maketitle

\begin{abstract}
We study the optimization landscape and the stability properties of 
training problems with squared loss for neural networks and general nonlinear conic approximation schemes
in a deterministic setting.
It is demonstrated that,
if a nonlinear conic approximation scheme is considered that is (in an appropriately defined sense)
more expressive than a classical linear approximation approach
and if there exist unrealizable label vectors, 
then a training problem with squared loss is necessarily unstable 
in the sense that its solution set depends 
discontinuously on the label vector in the training data.
We further prove that
the same effects that are responsible for these instability properties
are also the reason for the emergence of saddle points and spurious local minima,
which may be arbitrarily far away from global solutions,
and that neither the instability of the training problem 
nor the existence of spurious local minima can, in general, be 
overcome by adding a regularization term to the objective function that penalizes the 
size of the parameters in the approximation scheme. 
The latter results are shown to be true regardless of whether the assumption of realizability 
is satisfied or not. 
It is further established that there exists a direct and quantifiable relationship between 
the analyzed instability properties and the expressiveness of the 
considered approximation instrument and that the 
set of training label vectors and, in the regularized case, 
Tikhonov regularization parameters that give rise to spurious local minima 
has a nonempty interior.
We demonstrate that our analysis  in particular applies to training problems 
for free-knot interpolation schemes 
and deep and shallow neural networks 
with variable widths that involve an arbitrary mixture of various activation functions 
(e.g., binary, sigmoid, 
tanh, arctan, soft-sign, ISRU, soft-clip, SQNL, ReLU, leaky ReLU, soft-plus, bent identity,
SILU, ISRLU, and ELU).
In summary, the findings of this paper illustrate that the improved approximation properties 
of neural networks and general nonlinear conic approximation instruments
come at a price and 
are linked in a direct and quantifiable way to undesirable properties of
the optimization problems that have to be solved in order to train them. 
\end{abstract}


~\pagebreak
\tableofcontents
~\\[-0.375cm]
\section{Introduction}
\label{sec:1}
The aim of this paper is to study the stability properties and the optimization landscape 
of training problems of the form
\begin{equation}
\label{eq:trainingpropprot}
\min_{\alpha \in D}\, \frac{1}{2n}\sum_{k=1}^n  \|\psi(\alpha, \xx_{\;d}^k) - \yy_d^k\|_\YY^2.
\end{equation}
Here, $\XX$ is supposed to be a nonempty set (the set of input elements), 
$\YY$ is supposed to be a finite-dimensional vector space over $\R$ that is
endowed with an inner product  $(\cdot, \cdot)_{\YY}$ and the associated norm $\|\cdot\|_\YY$ (the output space),
$\psi\colon D \times \XX \to \YY$, $(\alpha, \xx) \mapsto \yy$,  
is assumed to be an arbitrary but fixed approximation scheme (e.g., a neural network) that can be adjusted 
by selecting an $m$-dimensional vector $\alpha$ from a nonempty set of admissible parameters $D \subset \R^m$
(these may be weights, biases, coefficients, or something else),
and $(\xx_{\;d}^k, \yy_d^k)$, $k=1,...,n$, $n \in \mathbb{N}$, $n \geq 2$, 
is a given training set consisting of a label vector $\{ \yy_d^k\}_{k=1}^n \in \YY^n$
and an input  vector $\{ \xx_{\;d}^k\}_{k=1}^n \in \XX^n$. Note that, by introducing
the abbreviations 
\begin{equation}
\label{eq:randomeq2735355}
\begin{gathered}
X := \XX^n,
\qquad Y := \YY^n,
\qquad y_d := \{ \yy_d^k\}_{k=1}^n \in Y,
\qquad x_d := \{ \xx_{\;d}^k\}_{k=1}^n \in X,
\\
\big\| \{ \yy_k\}_{k=1}^n  \big\|_Y 
:= \left (
\frac{1}{2n} \sum_{k=1}^n   \| \yy_k  \|_{\YY}^2 \right )^{1/2}\quad \forall  \{ \yy_k\}_{k=1}^n  \in Y,
\\
\Psi\colon D  \times X \to Y,\qquad \Psi \left (\alpha, \{ \xx_{\;k}\}_{k=1}^n\right) := \left \{\psi(\alpha, \xx_{\;k}) \right \}_{k=1}^n,
\end{gathered}
\end{equation}
the problem \eqref{eq:trainingpropprot} can also be written in the more compact form 
\begin{equation}
\label{eq:trainingpropprot42}
\min_{\alpha \in D} \|\Psi( \alpha, x_d) - y_d\|_Y^2. 
\end{equation}

We prove that, if $\psi$
is more expressive than a 
linear approximation instrument, if the set $\Psi(D, x_d)$ is a cone,
and if the number of samples $n$ is so large that there exist training label vectors 
$y_d \in Y$ for which the optimal value of \eqref{eq:trainingpropprot} is positive
(i.e., $y_d$ that are not realizable),
then the problem \eqref{eq:trainingpropprot} always suffers from 
spurious local minima/spurious basins, saddle points, instability effects, and/or the nonuniqueness of solutions
for certain choices of the label vector $y_d$. This illustrates that undesirable 
properties of the minimization problem \eqref{eq:trainingpropprot} always appear
if $\psi$ is trained in a not sufficiently overparameterized regime. 
We moreover show that, in the presence of label vectors $y_d$ with a positive 
optimal loss value and under appropriate assumptions on $\psi$, there is a direct and quantifiable  relationship between 
the instability properties of \eqref{eq:trainingpropprot}, the size of the set of vectors $y_d$ for which 
spurious local minima/spurious basins exist, and the approximation power of $\psi$. This establishes
a quid-pro-quo relationship between the expressiveness of $\psi$ and undesirable properties of \eqref{eq:trainingpropprot}.
Compare also with the illustrative example in \cref{sec:4} in this context.

For problems \eqref{eq:trainingpropprot} for which the optimal value of the loss 
function is identical zero for all $y_d \in Y$, we establish that non-optimal stationary points, spurious 
local minima, and instability effects may still occur if a local linear/quadratic approximation of 
$\psi$ is not able to fit arbitrary $y_d$ with zero error. We moreover prove that the same is true 
for training problems that include an additional regularization term in the objective function. 
This shows that, although problems concerning spurious local minima, saddle points, and instability properties
may be mitigated by overparameterization and classical Tikhonov regularization, 
one cannot expect that such techniques resolve these undesirable effects entirely.
For an overview of the various theorems on the properties of squared-loss training problems
for general approximation schemes $\psi$ proved in this paper, see \cref{subsec:2.2}.

A main feature of our analysis is that it is axiomatic and discusses the 
properties of training problems of the form \eqref{eq:trainingpropprot} on a general level. 
Because of this, our results are not restricted to a certain type of approximation instrument
but can readily be applied to all functions $\psi$ that 
satisfy the required abstract assumptions. 
This offers the additional benefit of giving an insight into the
mathematical mechanisms that are behind, e.g., the emergence of spurious local minima in training problems
of the form \eqref{eq:trainingpropprot} and also allows to unify various previous results on the topic. 
For an overview of the consequences that our analysis has for neural networks,
we refer to \cref{subsec:2.3} below.
Moreover, our approach also allows to establish new results on the properties of
squared-loss training problems. 
In contrast to prior contributions, we are, for example, able to rigorously
prove that a training problem of the form \eqref{eq:trainingpropprot} satisfying mild assumptions 
always possesses spurious local minima for all label vectors $y_d$ in a nonempty 
open cone $K \subset Y$ when $\psi$ is a deep neural network that involves activation functions 
which are affine on an open nonempty interval, see \cref{cor:spurminaffineNN,cor:spurminconstantNN}.
Further, we can establish that these spurious 
minima can be arbitrarily bad in relative and absolute terms and in terms of loss. 
For a detailed discussion of this topic, see \cref{subsec:2.3,subsec:2.4} below.
 In summary, this paper thus provides an in-depth analysis of what can and---maybe more importantly---what 
 cannot be expected regarding the presence of spurious local minima and instability effects when studying squared-loss 
 training problems of the form \eqref{eq:trainingpropprot} for nonlinear approximation schemes $\psi$ in different training regimes.

We conclude this introduction with a brief summary of the content and the structure of 
the remainder of the paper:

In \cref{sec:2}, we give an overview of our main theorems and 
the consequences that our analysis has for neural networks. Here, we also discuss in more detail the 
contribution of the paper and relations to previous work.
\Cref{sec:3} is concerned with preliminaries and basic concepts 
that are needed for the rigorous analysis of the training problem \eqref{eq:trainingpropprot}.
In \cref{sec:4}, we discuss a toy example that 
illustrates the basic ideas of our approach and 
provides some intuition on how the approximation 
properties of a function $\psi$ are related to the loss landscape of training 
problems of the form \eqref{eq:trainingpropprot}.
The subsequent \cref{sec:5} contains the bulk of our analysis of the optimization landscape and
the stability properties of training problems with squared loss for 
general approximation schemes. Here, we rigorously prove the main results 
presented in \cref{sec:2}.
\Cref{sec:6} addresses the consequences that the analysis of \cref{sec:5} 
has for special instances of nonlinear conic
approximation instruments, namely, 
classical free-knot interpolation schemes and deep and shallow neural networks. 
This section in particular includes the rigorous proofs of the results collected in 
\cref{subsec:2.3}.
In \cref{sec:7}, we conclude the paper with
additional remarks on the overall role that our results play in the 
study of neural networks and the field of approximation theory.

\section{Overview of Main Theorems and Discussion of Contribution}
\label{sec:2}
In this section, we discuss the background of our work and summarize our contributions.

\subsection{Background}
\label{subsec:2.1}
Due to the widespread use of the quadratic loss function,
minimization problems of the type  \eqref{eq:trainingpropprot} 
(or \eqref{eq:trainingpropprot42}, respectively)
are encountered very frequently in machine learning and
the field of approximation theory in general. 
One of the main reasons why problems of the form \eqref{eq:trainingpropprot} are considered 
so often in the literature is that 
solving them (or, at least, solving them approximately)
by means of classical first-order methods works very well in practical applications---in particular in the context of neural networks. 
This has led some authors to speculate that training problems of the type \eqref{eq:trainingpropprot} 
are always very well behaved
when neural networks are considered, e.g., in the sense that all local minima
of \eqref{eq:trainingpropprot}  
are also globally optimal or achieve a loss that is very close to the optimum.
Compare, for instance, with the numerical results and conclusions of
\cite{LeCun2015}, \cite{Nguyen2018}, and \cite{Yu1995} in this context.
At least for neural networks with linear activation functions, 
the belief that problems of the form \eqref{eq:trainingpropprot} always possess very nice properties
turns out to be not completely unfounded. 
Indeed, \cite{Kawaguchi2016} could prove that,
for deep linear neural networks, local minima of \eqref{eq:trainingpropprot} are
always also globally optimal so that---as far as the notion of local optimality is concerned---\eqref{eq:trainingpropprot} 
effectively behaves like a convex problem. 
 This effect was later also 
discussed in more detail by \cite{Zhou2017}, \cite{Laurent2018}, and \cite{Yun2019},
and, with view on the convergence properties of 
gradient descent algorithms, by \cite{Eftekhari2020} and \cite{Zou2020}.

Unfortunately, for truly nonlinear approximation schemes, 
the picture turns out to be more bleak. 
Although there have been numerous attempts 
to establish, for instance, the ``local minima = global minima''-property 
for networks with nonlinear activation functions \citep[mostly based on the 
hope that the linear case gives a good enough impression of the nonlinear one, 
see][]{Eftekhari2020,Saxe2013ExactST}, the results that have been obtained in this context so far 
are typically only applicable in very special situations and under rather restrictive assumptions
on the network architecture, the degree of overparameterization, and/or the 
considered training data.
Compare, for instance, with the findings of 
\cite{Yu1995,Kazemipour2020,Li2018,LiDawei2021,Liang2018,Oymak2019,Soudry2016,Cooper2020} in this regard. 
For a critical discussion of this topic and further references, see also \cite{Goldblum2020Truth}
and \cite{Ding2020}. 
The reason behind these deficits of the known positive results 
on the loss surface of general neural networks is that even
slightest nonlinearities in the activation function can have a huge impact on the 
optimization landscape of training problems of the form \eqref{eq:trainingpropprot}
and may very well give rise to 
spurious (i.e., not globally optimal) local minima. Data sets illustrating this for 
two-layer ReLU neural networks have been constructed, for example, by
\cite{Swirszcz2016}, \cite{Zhou2017}, and \cite{Safran2018}. The minima documented in 
the latter of these papers  
have recently also been studied in more detail by \cite{Arjevani2020}.
Further, \cite{Yun2019} showed that for two-layer ReLU-like networks spurious local minima 
emerge for almost all choices of the training data.
This illustrates that local minima that are not globally optimal are not the exception but rather the rule
when piecewise linear activation functions are considered. 
\cite{Yun2019} also provide 
explicit examples of training problems for non-ReLU neural networks with two layers
which possess non-globally optimal local minima.
For problems involving only a single neuron, 
an example with numerous local minima can also be found in 
the early work of \cite{Auer1996}. 
Compare also with the results on spurious valleys of
\cite{Nguyen2018,Venturi2019} in this context, 
and, for an overview of papers on the existence of spurious local minima, with \cite{Sun2019,Sun2020}.
What all of the results on the existence of spurious local minima in the above contributions have in common is that 
they are only concerned with networks which are rather shallow (with depth not exceeding two).
The reason for this is that, as soon as more layers are considered, 
the explicit construction of (nontrivial) spurious local minima---or, more precisely, proving that a constructed 
local minimum is indeed not a global one---becomes 
very cumbersome. 
Two of the few contributions that address the construction of
spurious local minima for networks of arbitrary depth are 
the recent ones of \cite{Goldblum2020Truth} and \cite{Ding2020}.
In both of these papers, however, a detailed discussion of the neuralgic point of 
whether the constructed local minima are really spurious
is largely avoided. \cite{Goldblum2020Truth} address this issue merely by providing
numerical evidence, and \cite{Ding2020} resort to the assumption of realizability 
to resolve this problem\footnote[1]{While this was correct at the time of writing, 
in a revised version of their paper, 
\cite{Ding2020} were able to lift the assumption of realizability 
for a certain class of spurious local minima constructed for $C^2$-activations, 
see \cite[Theorem 1]{Ding2020}. The technique used to accomplish this, 
however, does not carry over to activations with an affine segment, see 
\cite[Theorem 2, Corollary~1]{Ding2020}.
Our \cref{cor:spurminaffineNN,cor:spurminconstantNN} are able to fill this gap, cf.\ the discussion in \cref{subsec:2.4} below.\\[-0.7cm]}. 
What is further noteworthy is that the majority of 
contributions on the existence of spurious local minima 
currently found in the literature rely on the fact that neural networks 
with piecewise linear activation functions are able to locally emulate a linear 
neural network and thus inherit the solutions of training problems of the form 
\eqref{eq:trainingpropprot42} for linear approximation schemes as spurious local minima.
Compare, for instance, with the methods of proof used by
\cite{Yun2019,Goldblum2020Truth,Ding2020} in this context.

\subsection{Overview of Main Theorems on General Approximation Schemes}
\label{subsec:2.2}
The purpose of the present paper is to demonstrate 
that the undesirable properties of the optimization landscape
of
training problems of the form \eqref{eq:trainingpropprot42}
for neural networks
with nonlinear activation functions are, in fact, 
not the result of a particular choice of network architecture
but rather a necessary consequence of the improved approximation properties that 
neural networks enjoy in comparison with linear approaches. 
More precisely, 
in what follows,
we will demonstrate that indeed \emph{every} 
nonlinear approximation scheme that is conic and---in an appropriate
sense---more expressive than a linear approximation instrument 
(regardless of whether it is a neural network or something different, e.g., 
an adaptive interpolation approach)
gives rise to squared-loss training problems that suffer
from stability and uniqueness issues and/or the existence of non-optimal 
stationary points. 

The starting point of our analysis  
is the observation that the overwhelming majority of nonlinear approximation schemes 
$\psi\colon D \times \XX \to \YY$ currently found in the literature
possess the following two properties for all $n \geq 2$ and all training data vectors $x_d \in X$
with $\xx_{\;d}^k \neq  \xx_{\;d}^j$ for all $k \neq j$: 
\begin{enumerate}[label=\Roman*)]

\item 
\label{fundass:I}
\emph{\bf (Conicity)} The set $\Psi(D, x_d)$
(with $\Psi$ etc.\ defined as in Equation \ref{eq:randomeq2735355})
is a cone, i.e., 
\begin{equation*}
y \in \Psi(D, x_d),\,\,s \in (0, \infty)\quad \Rightarrow \quad s y \in \Psi(D, x_d).
\end{equation*}

\item  
\label{fundass:II}
\emph{\bf (Improved Expressiveness)} The map $\Psi(\cdot, x_d)\colon D \to Y$ satisfies
\begin{equation*}
\forall y_d \in Y \setminus \{0\}:
\quad  \inf_{\alpha \in D}  \|\Psi(\alpha, x_d) - y_d\|_Y^2 < \|y_d\|_Y^2.
\end{equation*}
\end{enumerate}

Note that the first of the above conditions is 
rather unremarkable. If, for example, a neural network is considered, then this assumption is 
automatically satisfied since the 
topmost layer is affine, see \cref{lemma:nnconic}. Property \ref{fundass:II} is more interesting in this context. 
It expresses that, for the considered data vector $x_d$, the function $\Psi(\cdot, x_d)\colon D \to Y$ is able to provide 
an approximation of every nonzero training label vector $y_d$ that is better than the 
trivial guess $y= 0 \in Y$. 
The main point here is that
the map $\Psi(\cdot, x_d)$ can accomplish this 
regardless of the relationship between the number of parameters $m \in \mathbb{N}$
and the number of training samples $n \in \mathbb{N}$
(and in particular also in those situations with $m \ll n$). 
For further details on this topic, we refer to \cref{sec:4}.

For every training problem of the type \eqref{eq:trainingpropprot42} that involves an
approximation scheme $\psi \colon D \times \XX \to \YY$ and a training data vector $x_d \in X$
satisfying \ref{fundass:I} and \ref{fundass:II},
we are able to prove the following (see the theorems in brackets for the mathematically rigorous statements):
\begin{itemize}
\item {\bf (Nonuniqueness and Instability of Best Approximations)}
If there exist label vectors $y_d \in Y$ that are not realizable, then 
the map $\Psi(\cdot, x_d)\colon D \to Y$ is always unable to provide unique best approximations 
for all $y_d$. 
(See \cref{def:bestapproxmap} for the precise definition of what we mean with the term ``best approximation'' here.)
Further, arbitrarily small perturbations in $y_d$ can cause 
arbitrarily large changes in the set of best approximations. The degree of discontinuity of 
the best approximation map depends on the extent to which $\psi$ and $x_d$ satisfy condition \ref{fundass:II}.
(See \cref{theorem:abstractinstability}.)

\item {\bf (Choice Between Excessive Nonuniqueness and Spurious Minima/Basins)}
If there exist label vectors $y_d$ that are not realizable and if the map $\Psi(\cdot, x_d)\colon D \to Y$
is continuous, then there exist uncountably many label vectors $y_d \in Y$ for which $\Psi(\cdot, x_d)\colon D \to Y$
provides infinitely many best approximations
or there exists an open nonempty
cone $K \subset Y$ such that, for each $y_d \in K$, \eqref{eq:trainingpropprot42} possesses spurious local minima
and/or spurious basins. (See \cref{th:stevaluedspurious}.)

\item {\bf (Existence of Undesirable Stationary Points)} 
If the map $\Psi(\cdot, x_d)\colon D \to Y$ is differentiable at a point $\bar \alpha \in D$
and if the function value and partial derivatives of $\Psi(\cdot, x_d)$ at $\bar \alpha$
do not span the whole of $Y$, 
then there exist uncountably many $y_d \in Y$ such that $\bar \alpha$ is an arbitrarily bad saddle point or
spurious local minimum of \eqref{eq:trainingpropprot42}. 
In particular, in the case $m+1 < \dim(Y)$, every point of differentiability of $\Psi(\cdot, x_d)$ is 
a saddle or spurious local minimum of \eqref{eq:trainingpropprot42} for uncountably many $y_d$. 
The position of these $y_d$ depends on the extent to which $\psi$ and $x_d$
satisfy \ref{fundass:II}. 
(See \cref{theorem:existencestatpts}.)

\item {\bf (Existence of Spurious Local Minima)} 
If $\Psi(\cdot, x_d)$ is able to locally parameterize a
proper subspace $V$ of $Y$, then there exists an open nonempty cone $K \subset Y$
such that \eqref{eq:trainingpropprot42} possesses spurious local minima for all $y_d \in K$. 
These spurious minima satisfy a growth condition in $Y$ and can be arbitrarily bad in relative and absolute terms 
and in terms of loss. 
The size of $K$ depends on the extent to which $\psi$ and $x_d$ satisfy \ref{fundass:II}.
If every vector is realizable, then it holds $K = Y \setminus V$.
(See \cref{theorem:badcone}.)

\item \emph{\bf (Instability and Nonuniqueness in the Presence of Realizability)} 
If there exists an $\bar \alpha \in D$ such that $\Psi(\cdot, x_d)$ maps an open neighborhood of $\bar \alpha$ into a
proper subspace of $Y$ and if every $y_d \in Y$ is realizable, then 
the solution set of \eqref{eq:trainingpropprot42} is instable w.r.t.\ 
perturbations of the vector $y_d$
and \eqref{eq:trainingpropprot42} is not uniquely solvable 
(in the sense of minimizing sequences) for certain choices of the vector $y_d$.
(See  \cref{cor:instabilityoverpara}.)

\item \emph{\bf (Ineffectiveness of Regularization)} 
If a term of the form $\nu g(\alpha)$ 
with a $\nu > 0$ and a regularizer $g \colon D \to [0, \infty)$
is added to the objective function of \eqref{eq:trainingpropprot42}, then 
the following is true (under appropriate assumptions on $\psi$ and $g$, see \cref{subsec:5.3}):
\begin{enumerate}[label=\roman*)]
\item There exists an open nonempty set $O \subset Y \times (0, \infty)$
such that, for all label vectors $y_d$ and regularization parameters $\nu$ with $(y_d, \nu) \in O$, 
the regularized training problem possesses a spurious local minimum.
Further, these spurious minima can be arbitrarily bad
in terms of loss. (See \cref{theorem:spuriousregprob} and \cref{rem:Oneighborhood}.) 
\item There exist uncountably many combinations of $y_d$ and $\nu$
such that the resulting regularized training problem is not uniquely solvable (in the sense of minimizing sequences) and 
possesses a discontinuous solution map. (See \cref{theorem:nonuniquereg}.) 
\item Regardless of the choice of $\nu$, adding the term $\nu g(\alpha)$ to the 
objective function of \eqref{eq:trainingpropprot42}
compromises the approximation property  \ref{fundass:II}. (See \cref{theorem:approxgone}.)
\end{enumerate}
\end{itemize}

Before we comment in more detail on how the above results are related to the literature
and on the overall contribution of this paper,
 we briefly summarize the consequences that our analysis 
has for the study of neural networks. 

\subsection{Overview of Main Consequences for Deep and Shallow Neural Networks}
\label{subsec:2.3}

Our first main result on neural networks establishes that these special instances of nonlinear 
approximation schemes are indeed covered by our abstract analysis:
\begin{itemize}
\item {\bf (Conicity and Improved Expressiveness of Neural Networks)}
Consider a fully connected feedforward neural network with input space $\XX = \R^{d_\xx}$,
output space \mbox{$\YY = \R^{d_\yy}$,} $d_\xx, d_\yy \in \mathbb{N}$, depth $L \in \mathbb{N}$,
widths $w_1,...,w_L \in \mathbb{N}$, and activation functions $\sigma_i\colon \R \to \R$, $i=1,...,L$. Suppose
that an $n \in \mathbb{N}$ with $n \geq 2$ and an $x_d := \{\xx_{\;d}^k\}_{k=1}^n \in \XX^n$ 
satisfying $\smash{\xx_{\;d}^j \neq \xx_{\;d}^k}$ for all $j \neq k$ is given, and that one of the following is true:
\begin{enumerate}[label=\roman*)]
\item  The functions $\sigma_i$ are of Heaviside type for all $i =1,...,L$ and it holds $w_1 \geq 2$.
\item The set $\{1,...,L\}$ can be decomposed into two (possibly empty) disjoint index sets $I$ and $J$ such that
the function $\sigma_i$ is of ``sigmoid type''
(e.g., sigmoid, tanh, arctan, soft-sign) for all $i \in I$, such that
the function $\sigma_i$ is of ``ReLU type''
(e.g., ReLU, soft-plus, swish) for all $i \in J$,
such that  $w_i \geq 2$ holds for all $i \in J$,
and such that $w_1 \geq 2$ holds in the case $1 \in I$ and $w_1 \geq 4$ in the case $1 \in J$. 
\end{enumerate}
Then, the neural network and $x_d$ satisfy \ref{fundass:I} and \ref{fundass:II}.
(See \cref{lemma:heavisideapprox} and \cref{theorem:generalactiv}.)
\end{itemize}

We remark that the fact that neural networks
indeed possess property \ref{fundass:II}
under the above weak assumptions on the data, the network architecture, and the activation functions 
is also interesting on its own.
(See \cref{set:NNN} for our precise assumptions
and \cref{lemma:heavisideapprox,theorem:generalactiv} for an explanation of what we mean with
the terms ``sigmoid type'', ``ReLU type'',
and ``Heaviside type'' here.)
We will comment in more detail on this topic in \cref{subsec:2.4}.
For a result that shows that our analysis also covers ResNets, we refer the reader to \cref{cor:resNets}.
Given a training problem of the type \eqref{eq:trainingpropprot42} for a
neural network and a vector $x_d$ that satisfy the conditions in the last bullet point, we obtain, for instance,
the following corollaries from our abstract analysis (see again the results in brackets for the mathematically rigorous statements): 

\begin{itemize}
\item {\bf (Nonuniqueness and Instability of Best Approximations)}
If there exist unrealizable label vectors $y_d$, then the neural network is unable 
to provide unique best approximations for all $y_d \in Y$. Further, arbitrarily small perturbations of the label vector $y_d$
can affect the set of best approximations to an arbitrarily large extent. 
The degree of discontinuity of 
the best approximation map depends on the extent to which $x_d$ and
the considered network satisfy \ref{fundass:II}.
(See \cref{cor:NNbestapproxinstabil}.)

\item {\bf (Choice Between Excessive Nonuniqueness and Spurious Minima/Basins)}
If there exist unrealizable label vectors $y_d$ and if the activation functions $\sigma_i$, $i=1,...,L$, are continuous,  
then there exist uncountably many $y_d \in Y$ for which the neural network provides infinitely many best approximations 
or there exists an open nonempty cone $K \subset Y$ such that, for each $y_d \in K$, 
the training problem \eqref{eq:trainingpropprot42} 
possesses (arbitrarily bad) spurious local minima and/or spurious basins. (See \cref{cor:NNinfiniteNonUniqueness}.)

\item {\bf (Saddle Points and Spurious Minima in the Non-Overparameterized Case)}
If the number of parameters $m$ in the neural network is smaller than the product $n d_\yy$,
then every point of differentiability of the neural network is a saddle point or a spurious local minimum of \eqref{eq:trainingpropprot42}
for uncountably many $y_d$ and, as a saddle point or spurious local minimum, 
can be made arbitrarily bad in relative and absolute terms and in terms of loss by choosing $y_d$
appropriately. The position of these $y_d$ depends on the extent to which $x_d$ and
the considered network satisfy \ref{fundass:II}.  (See \cref{cor:saddlenonover}.)

\item {\bf (Saddle Points and Spurious Minima for Arbitrary Problems)}
If $d_\xx + 1 < n$ holds and if the functions $\sigma_i$ are differentiable, 
then there exists an $(m - d_\xx w_1)$-dimensional subspace of the parameter space of the network
such that each element of this subspace is a saddle point or a spurious local minimum of \eqref{eq:trainingpropprot42}
for uncountably many $y_d$. Again, these points can be made arbitrarily bad in 
relative and absolute terms and in terms of loss by choosing appropriate $y_d$.
(See \cref{cor:saddleoverpar}.)

\item {\bf (Spurious Local Minima for Activation Functions with an Affine Segment)}
If each $\sigma_i$ is affine-linear on some open nonempty interval $I_i \subset \R$ of its domain of definition
and if it holds $d_\xx + 1 < n$ and $\min(d_\xx, d_\yy) \leq \min(w_1,...,w_L)$,
then there exists an open, nonempty cone $K \subset Y$ such that
\eqref{eq:trainingpropprot42} possesses a spurious local minimum for each $y_d \in K$.
The size of this cone depends on the extent to which $x_d$ and the neural network $\psi$ satisfy \ref{fundass:II}.
If every vector is realizable, then the cone $K$ is dense in $Y$ 
and the solution map of \eqref{eq:trainingpropprot42} is discontinuous.  Further, by choosing appropriate $y_d$,
the spurious local minima can be made arbitrarily bad in relative and absolute terms and in terms of loss. 
(See \cref{cor:spurminaffineNN,cor:spurminconstantNN,cor:affineNNinstable}.)

\item \emph{\bf (Ineffectiveness of Regularization for Differentiable Activation Functions)} 
If the activation functions $\sigma_i$ are twice differentiable, 
if $\frac12(d_\xx + 2)(d_\xx + 1) < n$ holds, and if the training problem \eqref{eq:trainingpropprot42}
is regularized by adding a term of the form $\nu \|\alpha\|_p^p$, $\nu > 0$, $p \in [1,2]$,
to the objective function, where $\|\cdot\|_p$ denotes the $p$-norm on the Euclidean space, 
then there exists an open nonempty set $O \subset Y \times (0, \infty)$ 
such that the resulting regularized training problem possesses an (arbitrarily bad) spurious local minimum
for all $(y_d, \nu) \in O$
and there exist uncountably many values of $\nu > 0$
such that the regularized training problem is not uniquely solvable and 
possesses a discontinuous solution map. (See \cref{cor:regNNcrap}.)
\end{itemize}

Note that the set-valuedness and the instability of the best approximation map
in points one and two above immediately carry over to the solution operator of the problem \eqref{eq:trainingpropprot42}
w.r.t.\ $\alpha$
(just by taking preimages under the function $\alpha \mapsto \Psi(\alpha, x_d)$). 
For details on this topic,
see the comments after \cref{lemma:nonemptyproj} and \cref{rem:stability}.
We further would like to stress that the nonuniqueness of best approximations in, e.g., 
\cref{cor:NNbestapproxinstabil}
has nothing to do with symmetries in the parameterization of a neural network. On the contrary, 
it expresses that
there are different 
choices of the biases and weights (or, at least, minimizing sequences) 
which yield the same optimal loss in \eqref{eq:trainingpropprot42}
but give rise to functions $\psi(\alpha, \cdot)\colon \XX \to \YY$ that act 
differently not only on unseen data but even on the training data set
(see \cref{rem:stability} for more details).
Regarding \cref{cor:saddlenonover,cor:saddleoverpar}, 
we would like to point out that the fact that the saddle points and 
spurious local minima of \eqref{eq:trainingpropprot42} can be made arbitrarily bad is not merely a
consequence of the conicity property \ref{fundass:I}. Indeed,
it is easy to check that simply scaling the involved vectors cannot affect 
how well a non-optimal point performs in relative terms, cf.\ \cref{rem:scaling}.
Lastly, we would like to emphasize that our abstract analysis is not only applicable to deep and shallow neural networks
but also to other nonlinear conic approximation instruments. 
For an example demonstrating this, we refer the reader to \cref{subsec:6.1} where our results are applied to 
a free-knot spline interpolation scheme that has also been considered by 
\cite{Daubechies2019} and can be interpreted as a classical 
nonlinear dictionary approximation approach \citep[cf.][]{DeVore1998}. 
We only summarize the consequences of our abstract analysis for neural networks in this subsection 
because we expect that this is what the majority of readers are interested in.

\subsection{Contribution of the Paper and Comparison with Known Results}
\label{subsec:2.4}

The main contribution of this paper is that it establishes a direct and quantifiable connection between 
the improved approximation properties that nonlinear approximation schemes like neural networks 
enjoy over their linear counterparts
and undesirable properties 
of the optimization problems that have to be solved 
in order to train a nonlinear approximation instrument on a given data set. 
Compare, for instance, with the estimates \eqref{eq:thetabound61253}, \eqref{eq:scondition}, and \eqref{eq:Kspurdef}
in this context, which show that the degree of discontinuity of the 
best approximation map of a given nonlinear conic approximation scheme $\psi$,
the position of the label vectors $y_d$ that cause a given point to be a
saddle point or a spurious local minimum in \cref{theorem:existencestatpts},
and the size of the cone of ``bad'' label vectors in \cref{theorem:badcone}
depend directly on the extent to which the considered approximation instrument satisfies 
condition \ref{fundass:II}.
At least to the best of the author's knowledge, this relationship between the expressiveness of an approximation scheme 
and the loss landscape of the associated training problems  
has not been explored systematically so far in the literature 
(although it is, of course, closely related to classical topics of nonlinear approximation theory
and the study of nonlinear least-squares problems). Note that the results of 
this paper can be interpreted as an instance of the well-known fact that there is 
``no free lunch'' as they show that the improved approximation properties of, e.g., 
neural networks come at the price that the 
associated training problems are always potentially ill-posed in the sense of Hadamard and possess spurious local minima or saddle points
for certain choices of the training data. For further details on this topic
and its relationship to the curse of dimensionality and the problem of NP-hardness, see also \cref{sec:7}.

We would like to emphasize that the connections that we draw in this paper are not only 
interesting for their own sake but also allow to improve and complement 
known results on the optimization landscape of training problems with squared loss 
found in the literature. 
By exploiting the approximation property \ref{fundass:II}, for example, 
we are able to show that the assumption of realizability
used in \cite[Corollary 1]{Ding2020} to 
establish that certain local minima are not globally optimal
is unnecessary, that the conditions on the network widths in
\cite[Assumption~3]{Ding2020} 
can be relaxed, and that the observations made in the numerical experiments of \cite{Goldblum2020Truth} 
can also be backed up analytically, cf.\ \cref{cor:NNinfiniteNonUniqueness,cor:spurminaffineNN,cor:spurminconstantNN}. 
The main point in this context is that the approximation property \ref{fundass:II}
allows to prove that a point $\bar \alpha \in D$ is a spurious local minimum 
of a problem of the type \eqref{eq:trainingpropprot42} without 
the explicit construction of a parameter $\tilde \alpha \in D$
that yields a smaller loss than $\bar \alpha$. 
This makes the rather cumbersome calculations that are normally used 
to establish that a local minimum is not globally optimal unnecessary,
cf.\ the proofs of \cref{theorem:existencestatpts,prop:existencehotspurs}
and also the comments 
in the proof  of \cite[Theorem~1]{Yun2019} where it is emphasized that 
constructing points with smaller function values is precisely the hard part of showing
the existence of spurious local minima. 
We would like to point out 
that the difficulty of proving the spuriousness of a local minimum is also the reason why lifting the condition of realizability in 
\cite[Corollary~1]{Ding2020} is nontrivial. If realizability is assumed, then the optimal 
value of the loss function in  \eqref{eq:trainingpropprot42} is known to be zero. 
Accordingly, in order to construct an example of a spurious local minimum,
it suffices to construct a local minimum with a positive loss value. 
This can typically be done relatively easily by employing classical second-order sufficient conditions 
and by choosing the data of the problem appropriately. In the unrealizable setting, however, such a local construction 
is not sufficient anymore simply because it does not guarantee that the constructed local minimum 
is not a global one. To prove the latter, one needs global information about the neural network that 
cannot be obtained from derivative-based and, as a consequence, inherently local tools like second-order optimality conditions.
With the property \ref{fundass:II}, we are able to bridge this gap, see \cref{cor:spurminaffineNN,cor:spurminconstantNN}.
Since our approach does not require explicit constructions,
we are also able to rigorously prove the existence of spurious local minima in 
situations in which the classical approach of manually checking the spuriousness of a local minimum 
becomes intractable due to the presence of 
additional regularization terms or the architecture
of the considered nonlinear approximation scheme. 
Compare in particular with \cref{cor:spurminaffineNN,cor:spurminconstantNN,cor:regNNcrap} in this context,
which establish the existence of spurious local minima
for both unregularized and regularized training problems
and for neural networks with arbitrary depth and various activation functions.
At least to the best of the author's knowledge, results on the existence of spurious local minima with a comparable 
generality can currently not be found in the literature. In particular the existence of spurious local minima
in Tikhonov-regularized problems for deep networks has apparently not been considered so far. 
Note that our approach additionally offers the advantage that it allows to establish that 
saddle points and spurious local minima of training problems with squared loss 
can be arbitrarily far away from global optima in relative and absolute terms and in terms of loss,
see \cref{lemma:alldirectionsleadtosuboptimality,theorem:existencestatpts,theorem:spuriousregprob,prop:existencehotspurs}
and the associated corollaries on neural networks in \cref{sec:6}.

As already mentioned,
by exploiting the approximation property \ref{fundass:II}, we are also able to rigorously prove 
that solutions of training problems of the form \eqref{eq:trainingpropprot42} 
(or the associated best approximations, respectively)
cannot be expected to be unique or stable with respect to perturbations 
of the training label vector $y_d$, see  \cref{theorem:abstractinstability,theorem:nonuniquereg,cor:instabilityoverpara}.
This gives an analytic explanation for 
the instability effects that are commonly observed in network training, 
cf.\ \cite{Cunningham2000}
and also the comments on the nonuniqueness of 
global solutions in \cite[Section 1.1]{Cooper2020}. 
We remark that, for neural networks with one hidden layer, instability results 
similar to that in our \cref{theorem:abstractinstability} have already been proved in the $L^p$-spaces 
by \cite{Kainen1999,Kainen2001} by exploiting classical instruments of nonlinear approximation theory.
The finite-dimensionality of the training problem \eqref{eq:trainingpropprot42}
allows us to go further than the authors of these papers
and to establish the nonuniqueness and instability of solutions and best approximations for neural networks 
of arbitrary depth. By exploiting the inequality of Jung \cite[see][Theorem 11.1.1]{Burago1988}, we are further able to
establish a quantitative connection between the discontinuity properties of the best approximation map 
associated with \eqref{eq:trainingpropprot42} and the extent to which a function $\psi$ satisfies \ref{fundass:II}, cf.\ \eqref{eq:thetabound61253}.
The results that we prove in this context also seem to be new.

We would like to point out that, for deep and shallow neural networks 
whose activation functions are affine-linear on some open nonempty interval of their 
domain of definition, our results give a quite complete picture of how 
the optimization landscape of problems of the form \eqref{eq:trainingpropprot42} 
depends  on the approximation property \ref{fundass:II} or, more precisely, 
on the error bound $\Theta(\Psi, x_d) \in [0, 1)$ 
defined in \eqref{eq:defTheta} that measures the extent to which property \ref{fundass:II} is satisfied.
In the case $\Theta(\Psi, x_d) \in (0,1)$ (which corresponds to the situation where there are unrealizable label vectors),
one has to deal with both the instability of the set of best approximations of \eqref{eq:trainingpropprot42} 
and the existence of an open nonempty cone $K$ of vectors $y_d$
for which \eqref{eq:trainingpropprot42} possesses (potentially arbitrarily bad) spurious local minima
(see \cref{cor:spurminaffineNN,cor:spurminconstantNN,cor:NNbestapproxinstabil}). 
The closer $\Theta(\Psi, x_d)$ gets to zero (i.e., the more expressive the network becomes relative to $Y$, e.g., due 
to an increased number of network parameters or a smaller number of training pairs),
the less pronounced the instability properties of the best approximation map of \eqref{eq:trainingpropprot42} 
are, see \eqref{eq:thetabound61253}, and the larger the cone $K$ grows, see \eqref{eq:Kspurdef}. 
Finally, in the case $\Theta(\Psi, x_d) = 0$
(i.e., the case where every vector is realizable, cf.\ \cref{def:thetadef}),
the instability properties of the best approximation map are not present anymore 
and the cone $K$ is dense in $Y$ so that, for almost all $y_d$, \eqref{eq:trainingpropprot42}  
possesses spurious local minima. 
In summary, the above shows that, when considering problems of the type \eqref{eq:trainingpropprot42}
for a network satisfying the assumptions of \cref{cor:spurminaffineNN} or \cref{cor:spurminconstantNN} or, more generally, 
a nonlinear conic approximation instrument satisfying the conditions in \cref{theorem:badcone},
one can never get rid of both the discontinuity of the best approximation map and spurious local minima. 
The problem \eqref{eq:trainingpropprot42} always possesses at least one property that is undesirable
(cf.\ also with \cref{th:stevaluedspurious} in this context). 
Note that the fact that the instability properties of the best approximation map associated with \eqref{eq:trainingpropprot42} 
are not present when every vector $y_d \in Y$ is realizable provides a possible explanation for 
the often made observation that overparameterization benefits the training of neural networks 
in practical applications. Compare, e.g., with the results of \cite{Chen2020,Cooper2020,Li2018,Oymak2019,Zhu2019,Soudry2016}
in this context. However, it seems to be unlikely that this is the only reason 
for the advantageous properties that overparameterized training problems typically enjoy.
In fact, one can see in \cref{cor:instabilityoverpara,theorem:nonuniquereg} that, even in the case where every vector $y_d \in Y$
is realizable and the objective contains additional regularization terms,
there are still certain nonuniqueness and instability effects present in problems of the form \eqref{eq:trainingpropprot42}. 
(These, however, 
are of a different quality than those arising from the nonuniqueness of best approximations in \cref{theorem:abstractinstability}.) 
Note that the observation that neither by overparameterization nor 
by adding regularization terms to the objective function it is possible to completely remove the ill-posedness 
of training problems of the type \eqref{eq:trainingpropprot42} is also remarkable on its own. 

Regarding the application of our abstract analysis to neural networks, 
we would like to stress that the fact that these special instances of nonlinear 
approximation schemes indeed satisfy the condition \ref{fundass:II}
under the weak assumptions of \cref{lemma:heavisideapprox,theorem:generalactiv}
is also interesting independently of the study of the loss landscape
of training problems of the form \eqref{eq:trainingpropprot42}. 
As we will see in \cref{sec:4}, 
the property \ref{fundass:II} is a characteristic 
that distinguishes neural networks 
 clearly from linear approximation schemes
(e.g., polynomial approximation) and thus gives an idea of 
why these approximation instruments are able to outperform classical approaches. 
Compare also with \cref{lemma:thetaprops} in this context which establishes 
that the property \ref{fundass:II} is directly related to worst-case estimates
for the approximation error that nonlinear approximation schemes achieve for arbitrary 
training label vectors $y_d$. 
We also would like to emphasize at this point that \ref{fundass:II} 
is a global property of an approximation scheme 
and thus of a completely different flavor than, e.g., the local properties 
of activation functions (for instance, piecewise linearity) that are 
commonly worked with in the analysis of neural networks.
This also becomes apparent in the proof of \cref{theorem:generalactiv}
which, in contrast to many classical approaches, 
is not based on concepts like linearization but
on the observation that the overwhelming majority of neural networks used in practice 
are able to emulate networks with binary activation functions by saturation
and that the property \ref{fundass:II} is inherited from these binary networks obtained in the saturation limit. 
Further details on this topic can be found in \cref{sec:6}.

We finally would like to emphasize 
that the theorems proved in this paper do not contradict 
the results on the absence of spurious local minima in training problems for neural networks with linear activation functions 
established, e.g., by \cite{Kawaguchi2016,Zhou2017,Laurent2018}.
Since such networks give rise to functions $\psi(\alpha, \cdot)$ that are affine and sets 
 $\Psi(D, x_d)\subset Y$ that are subspaces, 
they only satisfy condition \ref{fundass:II} in pathological situations 
and thus do not fall under the scope of, e.g., 
\cref{th:stevaluedspurious,prop:existencehotspurs,theorem:badcone}. 
Compare again with the example in \cref{sec:4} in this context. 
Similarly, our theorems also
do not contradict the results on the absence of spurious valleys established by \cite{Nguyen2018} and \cite{Venturi2019}
(simply because we are mainly concerned with classical spurious local minima in this work, 
cf.\ \cref{def:minima}). They are, however, in good accordance with 
the observations on the role and presence of saddle points in network training made, e.g., 
by \cite{Dauphin2014}. For further details on this topic and additional remarks
on the relationship between our results and the literature, we refer the reader to the comments 
after the respective theorems in the subsequent sections.

\section{Notation, Preliminaries, and Basic Concepts Needed for the Analysis}
\label{sec:3}
Before we begin with our analysis, we fix the notation 
and introduce some basic concepts.
As already mentioned in the introduction, 
the main focus of this work will be on
training problems of the form
\begin{equation}
\label{eq:trainingpropprotalt}
\min_{\alpha \in D}\, \frac{1}{2n}\sum_{k=1}^n  \|\psi(\alpha, \xx_{\;d}^k) - \yy_d^k\|_\YY^2.
\end{equation}
For easy reference, we restate our assumptions on the quantities in \eqref{eq:trainingpropprotalt} in:
\begin{assumption}[Standing Assumptions and Notation]\label[assumption]{ass:notation}~
\begin{itemize}
\item $\XX$ is a nonempty set,
\item $\YY$ is a finite-dimensional vector space over $\R$ that is 
endowed with an inner product  $(\cdot, \cdot)_{\YY}$ and the associated norm $\|\cdot\|_\YY$
(i.e., $\|\yy\|_\YY := \smash{(\yy, \yy)_{\YY}^{1/2}}$ for all $\yy \in \YY$),
\item $m,n \in \mathbb{N}$, $n \geq 2$, and $D \subset \R^m$ is a nonempty set,
\item $\psi\colon D \times \XX \to \YY$ is a function (representing an approximation scheme),
\item $\{ \xx_{\;d}^k\}_{k=1}^n \in \XX^n$, $\{ \yy_d^k\}_{k=1}^n \in \YY^n$ is the training data. 
\end{itemize}
\end{assumption}

Note that the subscript $d$ is used in \cref{ass:notation} to highlight that  $\{ \xx_{\;d}^k\}_{k=1}^n \in \XX^n$
and $\{ \yy_d^k\}_{k=1}^n \in \YY^n$ take the role of the training data in \eqref{eq:trainingpropprotalt}
(in contrast to, e.g., the arbitrary elements of the space $\YY^n$ appearing in equation \eqref{eq:Ppsidef} below). 
We would like to point out that 
several of the results proved in the following sections can be extended (in one form or another)
to more general loss functions and to infinite dimensions. See, for instance,
\cite[Theorem 2.4]{Christof2021} for an instability and nonuniqueness result for optimization problems in Banach spaces 
with an $L^p$-loss structure that is similar in nature to \cref{theorem:abstractinstability}
and also with the general theory of Chebychev sets found in \cite[Section II-3]{Braess1986}.
We restrict the attention to the squared-loss function to be able 
to present the theory developed in this paper in a uniform manner and to make the 
proofs and discussion of results less cumbersome. 

Next, we collect the abbreviations in \eqref{eq:randomeq2735355}.
\begin{definition}[Some Abbreviations]%
\label[definition]{def:basicdefnotation}%
In the situation of \cref{ass:notation}, we define:
\begin{itemize}
\item $X$ to be the Cartesian product $X := \XX^n$,
\item $Y$ to be the Hilbert space $Y := \YY^n$ endowed with the product
\begin{equation*}
 \left ( \{ \yy_k\}_{k=1}^n , \{ \zz_k\}_{k=1}^n \right )_Y 
:= \frac{1}{2n} \sum_{k=1}^n \left ( \yy_k, \zz_k \right )_{\YY}\qquad \forall \{ \yy_k\}_{k=1}^n , \{ \zz_k\}_{k=1}^n  \in Y
\end{equation*}
and the associated norm $\|\cdot\|_Y$ (cf.\ Equation \ref{eq:randomeq2735355}),
\item $y_d$ to be the vector $y_d := \{ \yy_d^k\}_{k=1}^n \in Y$,
\item $x_d$ to be the vector $x_d := \{ \xx_{\;d}^k\}_{k=1}^n \in X$,
\item $\Psi$ to be the map
\begin{equation}
\label{eq:PsiDef}
\Psi\colon D  \times X \to Y,\qquad \left (\alpha, \{ \xx_{\;k}\}_{k=1}^n\right) \mapsto \left \{\psi(\alpha, \xx_{\;k}) \right \}_{k=1}^n . 
\end{equation}
\end{itemize}
\end{definition}

We remark that, here and in what follows, 
we always think of elements of the space $\R^m$ 
as column vectors. 
As already pointed out in \cref{sec:1}, 
the abbreviations in \cref{def:basicdefnotation} allow us to restate the problem \eqref{eq:trainingpropprotalt}
in the more compact form
\begin{equation}
\label{eq:trainingpropprot2}
\min_{\alpha \in D} \|\Psi( \alpha, x_d) - y_d\|_Y^2. 
\end{equation}
Note that, since the objective function of \eqref{eq:trainingpropprot2} 
is not necessarily coercive w.r.t.\ $\alpha$, it can, in general, not be expected that \eqref{eq:trainingpropprot2} 
possesses a
global minimizer $\bar \alpha \in D$. One can only guarantee that there exists 
a minimizing sequence $\{\alpha_i\} \subset D$, i.e., a sequence satisfying 
\begin{equation}
\label{eq:infsequence}
\|\Psi( \alpha_i, x_d) - y_d\|_Y^2 \to \inf_{\alpha \in D} \|\Psi( \alpha, x_d) - y_d\|_Y^2
\end{equation}
as $i$ tends to infinity. (This is, for example, the case when some of the activation functions in a neural 
network have to saturate to fit a training vector $y_d \in Y$ precisely.) 
To get a grip on these effects, it makes sense to not only study 
local and global minimizers $\bar \alpha \in D$ of \eqref{eq:trainingpropprot2}  
but also the set of all elements of $Y$ that can be approximated by the function $\Psi(\cdot, x_d)$ for a given $x_d$
and fit a training label vector $y_d$ in an optimal manner. 
This gives rise to: 

\begin{definition}[Best Approximation Map]
\label[definition]{def:bestapproxmap}
Let $x_d \in X$ be arbitrary but fixed and let $\Psi\colon D \times X \to Y$ etc.\ be as before. Then, we define 
$P_{\Psi}^{x_d}$ to be the map 
\begin{equation}
\label{eq:Ppsidef}
P_{\Psi}^{x_d}\colon Y \rightrightarrows Y,\qquad y_d \mapsto \argmin_{y \in  \closure_Y\left (\Psi(D, x_d)\right )} \|y - y_d\|_Y^2. 
\end{equation}
Here, the symbol $\rightrightarrows$ expresses that the function $P_{\Psi}^{x_d}$ may be set-valued
and with $\closure_Y(\cdot)$ we denote the topological closure of a set in $Y$. 
\end{definition}

Note that the map $P_{\Psi}^{x_d}$ is precisely the set-valued metric projection in $Y$ onto the 
closure of the image  $\Psi(D, x_d)$ of $D$ under the function $\Psi(\cdot, x_d)\colon D \to Y$
for the given vector $x_d$. 
Because of this, we in particular have:
\begin{lemma}[Properties of the Set of Best Approximations]
\label[lemma]{lemma:nonemptyproj}~
Suppose that $x_d \in X$ and $\psi \colon D \times \XX \to \YY$ are arbitrary but fixed. 
Then, the set $P_\Psi^{x_d}(y_d)$ is nonempty and compact for every training label vector $y_d \in Y$.
\end{lemma}

\begin{proof}
The nonemptyness and compactness of 
$P_\Psi^{x_d}(y_d)$ for all $y_d \in Y$ follow immediately from 
the fact that the minimization problem in the variable $y$
associated with the right-hand side of \eqref{eq:Ppsidef} possesses 
a nonempty, closed, and bounded set of solutions for all $y_d \in Y$
due to the continuity and coercivity of the norm, the closedness 
and nonemptyness of the set  $\closure_Y\left (\Psi(D, x_d)\right )$, the finite-dimensionality of $Y$,
and the theorem of Weierstrass. 
\end{proof}

We would like to point out that, by taking preimages and images under the function $\Psi(\cdot, x_d)\colon D \to Y$,
properties of the map $P_\Psi^{x_d}$ directly translate into properties of the optimization landscape 
of \eqref{eq:trainingpropprot2}
and vice versa. 
If, for example, $x_d$ and $y_d$ are vectors such that $P_\Psi^{x_d}(y_d) = \{\bar y_1, \bar y_2\}$ 
holds for some $\bar y_1 \neq \bar y_2$
and if we denote the closed balls in $Y$ of radius $\varepsilon$ around $\bar y_i$, $i=1,2$,
with  $B_{\varepsilon}^Y(\bar y_i)$, then, 
for every arbitrary but fixed $\varepsilon > 0$ with $B_\varepsilon^Y(\bar y_1) \cap B_\varepsilon^Y(\bar y_2) = \emptyset$,
we trivially have that the preimages
\begin{equation*}
D_1 := \Psi(\cdot, x_d)^{-1}\left ( B_\varepsilon^Y(\bar y_1) \right ) \subset D
\quad
\text{and}
\quad
D_2 := \Psi(\cdot, x_d)^{-1}\left ( B_\varepsilon^Y(\bar y_2) \right ) \subset D
\end{equation*}
satisfy $D_1 \neq \emptyset$, $D_2 \neq \emptyset$, $D_1 \cap D_2 = \emptyset$, and 
\begin{equation*}
\inf_{\alpha \in D_1} \|\Psi(\alpha, x_d) - y_d\|_Y^2
=
\inf_{\alpha \in D_2} \|\Psi(\alpha, x_d) - y_d\|_Y^2
=
\inf_{\alpha \in D} \|\Psi(\alpha, x_d) - y_d\|_Y^2.
\end{equation*}
The above implies that 
each of the two disjoint subsets $D_1$ and $D_2$
of the parameter space $D$
has to contain a global solution of the minimization problem \eqref{eq:trainingpropprot2}
or a sequence $\{\alpha_i\}$ satisfying \eqref{eq:infsequence}.
Note that the main advantage of considering the projection $P_\Psi^{x_d}$
instead of the objective $D \ni \alpha \mapsto \|\Psi(\alpha, x_d) - y_d\|_Y^2 \in \R$ of \eqref{eq:trainingpropprot2}
is that the former function allows  to also detect those cases where 
\eqref{eq:trainingpropprot2} possesses spurious local minima ``at infinity''
in the sense that the optimization landscape of \eqref{eq:trainingpropprot2} possesses 
basins which stretch to the boundary of $D$ and do not contain 
a local minimum in the classical sense (so-called spurious basins). Compare, e.g., with the behavior of 
the function $\R \ni \alpha \mapsto \min(\mathrm{e}^\alpha, \mathrm{e}^{-\alpha} - 1) \in \R$
in this context and also with \cref{th:stevaluedspurious}. Such cases should, of course, not be neglected as descent methods may very 
well get trapped in a non-optimal basin of this type and subsequently drive the parameter $\alpha$
to the boundary of the set $D$
without approximating the optimal value of the loss function on the 
right-hand side of \eqref{eq:infsequence} in the limit. 
Completely analogously to the above, 
stability and instability properties of $P_\Psi^{x_d}$ carry over to \eqref{eq:trainingpropprot2}, too.
For further details on this topic, we refer the reader to \cref{rem:stability}. 

For the sake of clarity, let us finally make precise what we mean with the terms 
``global minimum'', ``local minimum'', ``spurious local minimum'', etc.\
appearing in our analysis: 
\begin{definition}[Notions of Optimality]
\label[definition]{def:minima}
Given a function $f\colon U \to \R$ that is defined 
on a subset $U$ of a normed space $(V, \|\cdot \|_V)$, we call a point $\bar v \in U$ a:
\begin{itemize}
\item  global minimum (or, more precisely, global minimizer) of the function $f$ if
$f(v) \geq f(\bar v)$ holds for all $v \in U$.
\item  local minimum (or, more precisely, local minimizer) of the function $f$ if there 
exists a closed ball $B^V_\varepsilon(\bar v)$ of radius $\varepsilon > 0$ in $V$ centered at $\bar v$ such that
$f(v) \geq f(\bar v)$ holds for all $v \in U \cap B^V_\varepsilon(\bar v)$.
\item spurious local minimum of $f$ if $\bar v$ is a local minimum but not a global minimum of $f$.
\item global (respectively, local, respectively, spurious local)
maximum of $f$ if $\bar v$ is a global (respectively, local, respectively, spurious local) minimum of the function $-f$.
\item saddle point of $f$ if $V = \R^l$ holds for some $l \in \mathbb{N}$, 
$\bar v$ is an element of the interior of $U$, $f$ is differentiable 
at $\bar v$, it holds $\nabla f(\bar v) = 0$, and $\bar v$ is neither 
a local minimum nor a local maximum of $f$.
\end{itemize}
\end{definition}

We remark that some authors apparently go so far as to call every point with a vanishing 
gradient and a vanishing Hessian a spurious local minimum.
We believe that the term ``spurious local minimum''
should be reserved for points that are local minima.
Finally, we would like to emphasize that, throughout this work, the symbols $\min$, $\argmin$, etc.\
always refer to the global notion of optimality (e.g., in the definition of the map $P_\Psi^{x_d}$).

\section{A Toy Problem Illustrating the Basic Ideas}
\label{sec:4}

Having introduced the necessary notation, 
we now turn our attention to the optimization landscape and the stability properties 
of training problems of the form \eqref{eq:trainingpropprot2}.
We begin with a simple example that illustrates the main ideas of our analysis and
gives some intuition on why nonlinear approximation schemes may possess better 
approximation properties than their linear counterparts and on how these properties are related 
to the behavior of the function $P_\Psi^{x_d}$
and the loss landscape of \eqref{eq:trainingpropprot2}.
To construct our example, let us suppose that 
\begin{equation}
\label{eq:examplesituation323}
\XX = \YY = \R,\quad 
n \in \mathbb{N},\quad 
n \geq 3,
\quad 
m = 2,
\quad
X = Y = \R^n,
\quad
\text{and}
\quad
D = \R^2,
\end{equation}
and that 
$x_d = \{\xx_{\;d}^k\}_{k=1}^n \in X$ 
is an arbitrary but fixed training data vector which satisfies
$\xx_{\;d}^1 < \xx_{\;d}^2  < ... < \xx_{\;d}^n $.
Let us further assume, for a start, that 
we are given an approximation scheme $\psi\colon \R^2 \times \R \to \R$, $(\alpha, \xx) \mapsto \yy$,
that is linear in the sense that the function $\psi$
is linear in the parameter vector $\alpha$. Then, 
we trivially have 
\[
\Psi(\alpha, x_d) = \alpha_1 \Psi(e_1, x_d) + \alpha_2 \Psi(e_2, x_d),
\]
where $e_1, e_2$ denote the standard basis vectors of $\R^2$,
and the training problem \eqref{eq:trainingpropprot2}
can also be written as 
\begin{equation}
\label{eq:toyprob}
\min_{(\alpha_1, \alpha_2) \in \R^2}\, 
\left \|\alpha_1 \Phi_1 + \alpha_2 \Phi_2 - y_d\right \|_Y^2
\end{equation}
for every arbitrary but fixed $y_d = \{\yy_d^k\}_{k=1}^n \in Y$, 
where $\|\cdot\|_Y$ is the Euclidean norm on $Y = \R^n$ scaled with the factor $1/\sqrt{2n}$
and where $\Phi_j := \Psi(e_j, x_d) \in \R^n$, $j = 1,2$.
For a linear scheme $\psi$,
\eqref{eq:trainingpropprot2} thus boils down to a standard approximation problem 
which aims to find a function $\phi\colon \R \to \R$ in the linear subspace 
spanned by the set $\{\xx \mapsto \psi(e_1, \xx),\, \xx \mapsto \psi(e_2, \xx)\}$ 
that fits the $n$ given function values $\yy_d^k \in \R$
at the locations $\xx_{\;d}^k\in \R$ optimally in the least-squares sense.  
Note that the structure of \eqref{eq:toyprob} in particular implies that,
regardless of which linear scheme $\psi\colon \R^2 \times \R \to \R$ we consider here, there are 
always nontrivial choices of $y_d$ for which the problem \eqref{eq:toyprob} 
possesses the optimal solution $\bar \alpha = (0,0)^T$ so that $\psi$
does not provide an approximation of $y_d$ that is better than the trivial guess $\bar y = 0$. 
Indeed, for all $y_d$ in the orthogonal complement of the space $\span\{\Phi_1, \Phi_2\} \subset \R^n$
w.r.t.\ the Euclidean scalar product,
we clearly have 
\begin{equation*}
\argmin_{(\alpha_1, \alpha_2) \in \R^2}\, 
\left \|\alpha_1 \Phi_1 + \alpha_2 \Phi_2 - y_d\right \|_Y^2
=
\argmin_{(\alpha_1, \alpha_2) \in \R^2}\, 
\left \|\alpha_1 \Phi_1 + \alpha_2 \Phi_2\right \|_Y^2
+
\left \|y_d\right \|_Y^2
\supset 
\{0\}. 
\end{equation*}
Using the notation in \cref{def:basicdefnotation},
this observation can also be expressed in the more compact form
\begin{equation}
\label{eq:linearproperty263535}
\exists y_d \in Y \setminus \{0\}:
\qquad  \inf_{\alpha \in D}  \|\Psi(\alpha, x_d) - y_d\|_Y^2 = \|y_d\|_Y^2.
\end{equation}
For comparison, let us now consider
the nonlinear approximation scheme 
\begin{equation}
\label{eq:nonlinearapproxexample}
\psi \colon \R^2 \times \R \to \R,
\qquad
(\alpha, \xx) \mapsto \sigma(\alpha_1 \xx + \alpha_2),
\end{equation}
with $\sigma\colon \R \to \R$ given by 
\begin{equation*}
\sigma(s) := \min(0, |s +1| - 1) + \max(0, 1 - |s - 1|)
\end{equation*}
and define $\bar \alpha_{l, \delta} \in \R^2$, $l \in \{1,...,n\}$, $|\delta| \leq 1$,  by
\begin{equation*}
\bar \alpha_{l, \delta}^1 
:= 3\left (\min_{k=2,...,n} \xx_{\;d}^k - \xx_{\;d}^{k-1}\right )^{-1},
\qquad
\bar \alpha_{l, \delta}^2  := - \bar \alpha_{l, \delta}^1 \xx_{\;d}^l + \delta. 
\end{equation*}
Then, from the properties of $\psi$, it follows straightforwardly that 
\begin{equation*}
\psi(\bar \alpha_{l, \delta}, \xx_{\;d}^k)
=
\begin{cases}
\delta  & \text{ for } k = l
\\
0 & \text{ for all } k \neq l
\end{cases}
\end{equation*}
holds for all $\delta$ with $|\delta| \leq 1$. This implies in particular that, for every 
arbitrary but fixed label vector $y_d \in \R^n \setminus \{0\}$,
there exists a parameter $\bar \alpha \in D$
such that the nonlinear approximation scheme \eqref{eq:nonlinearapproxexample} 
satisfies 
\begin{equation*}
\|\Psi(\bar \alpha, x_d) - y_d\|_Y^2 < \|y_d\|_Y^2,
\end{equation*}
namely, in the case $\yy_d^l \neq 0$, the vector $\bar \alpha_{l, \delta}$
with $\delta :=  \sgn(\yy_d^l)\min(|\yy_d^l|, 1)$.
In short, 
\begin{equation}
\label{eq:improvedapproxprop32}
\forall y_d \in Y \setminus \{0\}:
\qquad  \inf_{\alpha \in D}  \|\Psi(\alpha, x_d) - y_d\|_Y^2 < \|y_d\|_Y^2.
\end{equation}
The above result shows that the nonlinearity of the function $\psi$ in \eqref{eq:nonlinearapproxexample}
indeed allows this map to possess better approximation properties than the 
linear schemes considered at the beginning of this section 
in the sense that, for every arbitrary but fixed nonzero $y_d \in \R^n$, 
we can find an $\alpha \in D$ such that $\Psi(\alpha, x_d)$
provides a loss that is smaller than that of the trivial 
guess $\alpha = 0$ and the associated vector $\Psi(0, x_d) = 0$.
The map $\psi$ in \eqref{eq:nonlinearapproxexample} is thus able to approximate 
every given label vector $y_d$ at least to a small extent 
even in those situations where the problem \eqref{eq:trainingpropprot2} is grossly underparameterized,
i.e., satisfies $m \ll n$---a feature that is not obtainable 
with a scheme that is linear in $\alpha$ and possesses the parameter space $D = \R^2$
as we have seen in \eqref{eq:linearproperty263535}.
Note  that this property can also be interpreted as a ``relaxed''
version of realizability that holds for all label vectors $y_d \in Y$ regardless of the choice of $n$ and $m$, 
cf.\ the analysis in \cref{sec:5}.

However, the example  \eqref{eq:nonlinearapproxexample} also immediately shows that 
the improved expressiveness in \eqref{eq:improvedapproxprop32} does not come for free. 
If we consider, for instance, the image of the parameter space $D = \R^2$ under the function $\Psi(\cdot, x_d)$
associated with the nonlinear approximation scheme in \eqref{eq:nonlinearapproxexample} 
in the case $n=3$ 
for the training data vector $x_d = (-1/2, 1/2, 1)^T$, then it is readily seen 
that this set is a nontrivial union of numerous segments of two-dimensional subspaces, cf.\ \cref{fig:graphplots}. 
This implies in particular that the projection $P_\Psi^{x_d}$ onto $\closure_{Y}(\Psi(D, x_d))$
is not single-valued at all points and, as a consequence, 
that the best approximating element provided by $\closure_{Y}(\Psi(D, x_d))$ 
is not uniquely determined for all possible choices of the training label vector $y_d$. 
It is moreover easy to check that the locally affine-linear structure of
the set $\closure_{Y}(\Psi(D, x_d))$ entails
that the optimization landscape of the training problem \eqref{eq:trainingpropprot2} 
for the approximation scheme $\psi$ in \eqref{eq:nonlinearapproxexample}
possesses spurious local minima and saddle points for various choices of $y_d$, 
cf.\ \cref{prop:existencehotspurs,theorem:badcone} below. 
The intuitive reason behind all these effects is that the same geometric properties of the 
image $ \Psi(D, x_d)$, that allow $\psi$ to satisfy \eqref{eq:improvedapproxprop32},
also imply that this set is folded in a way that causes the normal cones of various points
on $\closure_{Y}(\Psi(D, x_d))$ to intersect. 
\vspace{-0.1cm}
\begin{figure}[H]
\centering
\begin{subfigure}{.5\textwidth}
  \centering
  \includegraphics[width=.85\linewidth]{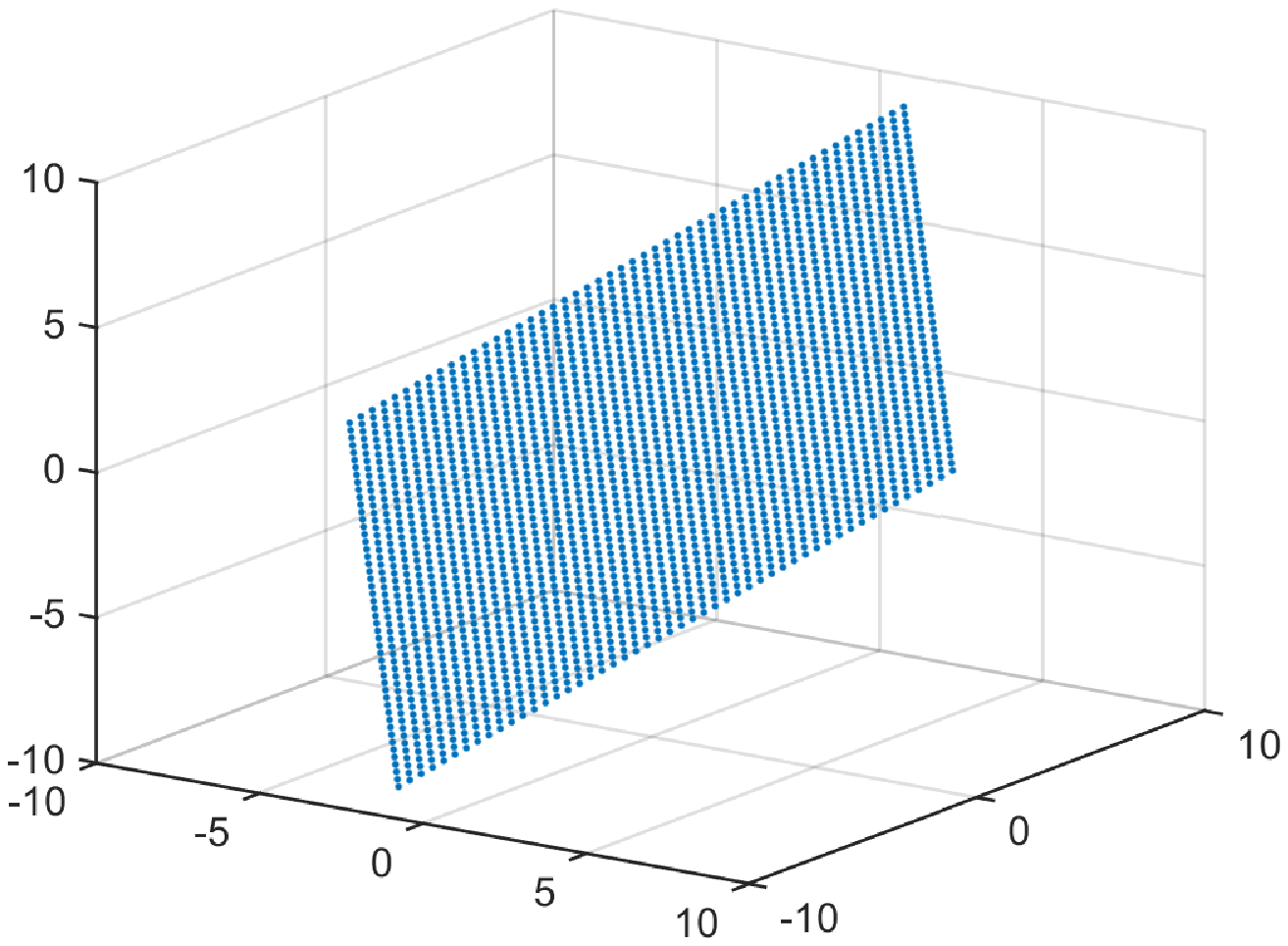}
  \caption{ $\Psi(D, x_d)$ for a linear scheme $\psi$}
  \label{fig:sub1}
\end{subfigure}%
\begin{subfigure}{.5\textwidth}
  \centering
  \includegraphics[width=.85\linewidth]{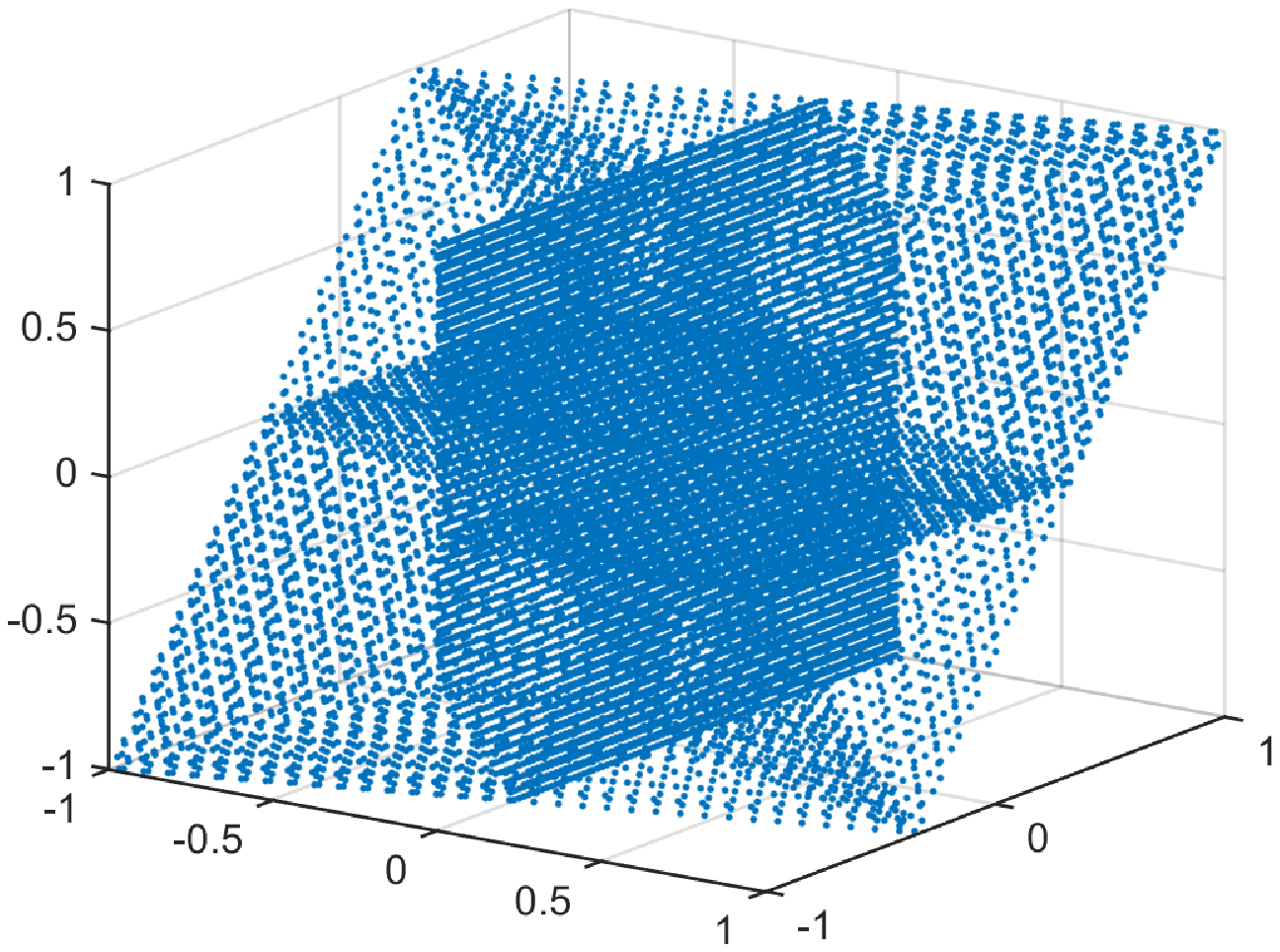}
  \caption{$\Psi(D, x_d)$ for the scheme $\psi$ in \eqref{eq:nonlinearapproxexample}}
  \label{fig:sub2}
\end{subfigure}
\caption{Scatter plot of the image $\Psi(D, x_d)$
for the linear, polynomial approximation scheme $\psi\colon \R^2 \times \R \to \R$, $\psi(\alpha, \xx) = \alpha_1 \xx + \alpha_2$,
(left) and the nonlinear function $\psi$ in \eqref{eq:nonlinearapproxexample} (right)
for $n = 3$ and the training data vector $x_d = (-1/2, 1/2, 1)^T$.}
\label{fig:graphplots}
\end{figure}

In the remainder of this paper, we will prove that the above undesirable properties of the function $P_\Psi^{x_d}$ 
and the optimization problem  \eqref{eq:trainingpropprot2} indeed inevitably appear
when the considered approximation scheme $\psi$ satisfies \eqref{eq:improvedapproxprop32} and is conic in the sense 
that the set $\Psi(D, x_d)$ is a cone. We will moreover demonstrate that 
nearly all commonly used nonlinear approximation instruments
(and in particular neural networks) are covered by this setting
and are thus subject to the above effects. 
Note that this also shows that \eqref{eq:improvedapproxprop32} is, in fact, a quite fundamental property. 

Before we demonstrate  that the above observations 
indeed carry over to a far more general setting, we would like to point out that 
the example that we have studied in this section is a rather academic one.
It is easy to check that the approximation scheme \eqref{eq:nonlinearapproxexample} possesses 
various properties that are highly undesirable and thus would never be 
a sensible choice for a practical application.
Moreover, the scheme $\psi$ in \eqref{eq:nonlinearapproxexample} 
is clearly not conic and thus violates one of the main assumptions of the subsequent analysis.
We remark that this second deficit can be fixed easily by adding 
a further parameter $\alpha_3 \in \R$ to $\psi$, i.e., 
by considering the modified function $\tilde \psi\colon \R^3 \times \R \to \R$,
$(\alpha, \xx) \mapsto \alpha_3\sigma(\alpha_1 \xx + \alpha_2)$. 
In fact, after this modification, the resulting approximation instrument $\tilde \psi$
is nothing else than a simple neural network with a single neuron and 
a lightning-shaped activation function, cf.\ \cref{set:NNN}.
We have considered the function $\psi$ in \eqref{eq:nonlinearapproxexample} in this section since, on the one hand, 
it possesses the property \eqref{eq:improvedapproxprop32} and, on the other hand, 
satisfies $\closure_{Y}(\Psi(D, x_d)) \neq Y$ for $n=3$---thus enabling the visualization in \cref{fig:graphplots}. 
For the function $\tilde \psi$, one has to consider at least $n = 4$ samples to 
achieve that  $\closure_{Y}(\Psi(D, x_d)) \neq Y$ holds and for this dimension
an illustration as in \cref{fig:graphplots} is not possible anymore.
(It seems to be difficult to construct an example of a function $\psi$
that satisfies both \ref{fundass:I} and \ref{fundass:II} and simultaneously $\closure_{Y}(\Psi(D, x_d)) \neq Y$
for $n=3$.)
As already mentioned, we will see in \cref{sec:6} that various
commonly used approximation schemes (and in particular general neural networks) exhibit a behavior that is very 
similar to that of the function $\psi$ in \eqref{eq:nonlinearapproxexample}.
In \cref{subsec:5.2}, we will moreover see that, 
at least as far as the existence of saddle points 
and spurious local minima is concerned, it is not essential that it holds $n > m$ as in the situation of \eqref{eq:examplesituation323}.

\section{Analysis and Rigorous Proofs in the Abstract Setting}
\label{sec:5}
The aim of this section is to study the behavior of the function 
$P_\Psi^{x_d}$ and the loss landscape of the training problem \eqref{eq:trainingpropprot2}
for a general, nonlinear, conic approximation scheme satisfying
\eqref{eq:improvedapproxprop32}. 
Motivated by the observations made in \cref{sec:4} and by what is encountered in practical applications, 
we will consider the following setting: 

\begin{assumption}
{\hspace{-0.05cm}\bf(Standing Assumptions for the Analysis of \cref{sec:5})}\label[assumption]{ass:standingassumpssec4}
Let $\XX$, $\YY$, etc.\ be defined as in \cref{sec:3}. 
We assume that an approximation scheme $\psi\colon D \times \XX \to \YY$ and 
an arbitrary but fixed training data vector $x_d \in X$ are given such that the following two conditions are satisfied:
\begin{enumerate}[label=\textup{\Roman*)}]
\item  
\label{fa:I}
\emph{(Conicity)} The set $\Psi(D, x_d)$ is a cone in the sense that
\begin{equation*}
y \in \Psi(D, x_d),\,\,s \in (0, \infty)\quad \Rightarrow \quad s y \in \Psi(D, x_d).
\end{equation*}

\item  
\label{fa:II}
\emph{(Improved Expressiveness)} The map $\Psi(\cdot, x_d)$ satisfies
\begin{equation*}
\forall y_d \in Y \setminus \{0\}:
\quad  \inf_{\alpha \in D}  \|\Psi(\alpha, x_d) - y_d\|_Y^2 < \|y_d\|_Y^2.
\end{equation*}
\end{enumerate}
\end{assumption}

As already mentioned, various examples of approximation schemes $\psi\colon D \times \XX \to \YY$ satisfying the conditions 
in \cref{ass:standingassumpssec4} will be presented in \cref{sec:6}. 
Henceforth, the basic idea of our analysis will be to prove that the properties \ref{fa:I} and \ref{fa:II}---although very desirable from the approximation point of view---also automatically 
imply that training problems of the form \eqref{eq:trainingpropprot2} possess various 
disadvantageous properties.  
We begin with some basic observations:

\begin{lemma}[Reformulation of the Improved Expressiveness Property]
\label[lemma]{lemma:reformulatedcondition}
In the situation of \cref{ass:standingassumpssec4}, 
the property in \ref{fa:II} is equivalent to the condition
\begin{equation}
\label{eq:randomeq172636}
\min_{y \in  \closure_Y\left (\Psi(D, x_d)\right )} \|y - y_d\|_Y^2 < \|y_d\|_Y^2\qquad \forall y_d \in Y\setminus \{0\}.
\end{equation}
\end{lemma}

\begin{proof}
The implication ``\ref{fa:II} $\Rightarrow$ \eqref{eq:randomeq172636}'' is trivial. To prove ``\eqref{eq:randomeq172636} $\Rightarrow$ \ref{fa:II}'', it suffices to note that,
for every arbitrary but fixed $y_d \in Y \setminus \{0\}$, there exists a $\bar y \in \closure_Y\left (\Psi(D, x_d)\right )$
with 
\begin{equation*}
\|\bar y - y_d\|_Y^2   = \min_{y \in  \closure_Y\left (\Psi(D, x_d)\right )} \|y - y_d\|_Y^2  
\end{equation*}
by \cref{lemma:nonemptyproj} and to subsequently exploit the definition of the closure and 
the continuity of the norm $\|\cdot\|_Y$. This also shows that it indeed makes sense 
to write ``$\min$'' on the left-hand side of \eqref{eq:randomeq172636} instead of ``$\inf$''.
\end{proof}

Note that
\cref{lemma:reformulatedcondition} implies that \ref{fa:II}
is a property of the closure of the image $\Psi(D, x_d)$ of the map $\Psi(\cdot, x_d)\colon D \to Y$ 
and completely independent of how this image is parameterized by the variable $\alpha \in D$. 
To measure the extent to which condition \ref{fa:II} is satisfied by 
a given approximation scheme, we introduce: 

\begin{definition}[Error Bound $\boldsymbol{\Theta(\Psi, x_d)}$]
\label[definition]{def:thetadef}
In the situation of \cref{ass:standingassumpssec4}, we define
\begin{equation}
\label{eq:defTheta}
\Theta(\Psi, x_d) :=
\sup_{y_d \in Y,\, \|y_d\|_Y = 1} \left ( \inf_{y \in  \closure_Y\left (\Psi(D, x_d)\right )} \|y - y_d\|_Y^2 \right). 
\end{equation}
\end{definition}

\begin{remark}
The number $\Theta(\Psi, x_d)$ is precisely the square of the deviation 
of the unit sphere in $(Y, \|\cdot\|_Y)$ from the closure of the 
image of the map $\Psi(\cdot, x_d)\colon D \to Y$
in the sense of nonlinear approximation theory, 
see \cite[Section II]{Kurkova2002}. It corresponds to the squared
worst-case approximation error achieved by the map $\Psi(\cdot, x_d)$
w.r.t.\ the norm $\|\cdot\|_Y$
for label vectors $y_d$ chosen from the unit sphere in $Y$.
\end{remark}

Using our assumptions \ref{fa:I} and \ref{fa:II} and the closedness of the set 
$\closure_Y\left (\Psi(D, x_d)\right )$, it is easy to establish the following:

\begin{lemma}[Properties of the Bound $\boldsymbol{\Theta(\Psi, x_d)}$]\label[lemma]{lemma:thetaprops}%
In the situation of \cref{ass:standingassumpssec4},
it holds $\Theta(\Psi, x_d) \in [0, 1)$. 
Further, for all label vectors $y_d \in Y$, the optimal value of the loss function in \eqref{eq:trainingpropprot2} satisfies 
\begin{equation}
\label{eq:randomeq2736352}
\inf_{\alpha \in D}  \|\Psi(\alpha, x_d) - y_d\|_Y^2 \leq  \Theta(\Psi, x_d) \|y_d\|_Y^2,
\end{equation}
and there exists at least one ``worst-case'' unit label vector $\bar y_d \in Y$ with the properties
\begin{equation}
\label{eq:randomeq273635}
\|\bar y_d\|_Y = 1
\qquad
\text{and}
\qquad
\inf_{\alpha \in D}  \|\Psi(\alpha, x_d) - \bar y_d\|_Y^2 =  \Theta(\Psi, x_d). 
\end{equation}
\end{lemma}

\begin{proof}
Using the distance function 
$\dist(\cdot, \closure_Y\left (\Psi(D, x_d)\right))$ to the set $\closure_Y\left (\Psi(D, x_d)\right )$ 
w.r.t.\ the norm $\|\cdot\|_Y$,
the identity in \eqref{eq:defTheta} can also be written as
\begin{equation*}
\begin{aligned}
\Theta(\Psi, x_d) &= 
\sup_{y_d \in Y,\, \|y_d\|_Y = 1}   \dist(y_d, \closure_Y\left (\Psi(D, x_d)\right ))^2.
\end{aligned}
\end{equation*}
Since the map $Y \ni y \mapsto  \dist(y, \closure_Y\left (\Psi(D, x_d)\right )) \in \R$ is continuous,
since the unit sphere 
$\{ y \in Y \mid \|y\|_Y = 1 \}$ is compact due to the finite-dimensionality of $Y$,
and since
\begin{equation}
\label{eq:randomeq123635}
 \dist(y_d, \closure_Y\left (\Psi(D, x_d)\right ))^2 = \min_{y \in  \closure_Y\left (\Psi(D, x_d)\right )} \|y - y_d\|_Y^2
=  \inf_{\alpha \in D}  \|\Psi(\alpha, x_d) - y_d\|_Y^2
\end{equation}
holds for all $y_d \in Y$ by exactly the same arguments as in the proof of \cref{lemma:reformulatedcondition}, 
 it now follows immediately that there 
exists at least one $\bar y_d \in Y$ with the properties in \eqref{eq:randomeq273635}. 
Note that, in combination with \ref{fa:II}, this also yields
\begin{equation*}
0 \leq \Theta(\Psi, x_d)  
=
\inf_{\alpha \in D}  \|\Psi(\alpha, x_d) - \bar y_d\|_Y^2 
 < \|\bar y_d\|_Y^2 = 1
\end{equation*}
so that $\Theta(\Psi, x_d)$ is an element of the interval $[0, 1)$ as claimed. 
It remains to prove \eqref{eq:randomeq2736352}. To this end, we note that
\eqref{eq:randomeq123635} and the cone property of the set 
$\closure_Y\left (\Psi(D, x_d)\right )$,
which follows immediately from \ref{fa:I}, imply that 
$y_d = 0$ is an element of $\closure_Y\left (\Psi(D, x_d)\right )$ and that
\begin{equation*}
\begin{aligned}
\inf_{\alpha \in D}  \|\Psi(\alpha, x_d) - y_d\|_Y^2 
&= 
\min_{y \in  \closure_Y\left (\Psi(D, x_d)\right )} \|y - y_d\|_Y^2
=
\min_{\tilde y \in  \closure_Y\left (\Psi(D, x_d)\right )} \|\|y_d\|_Y \tilde y - y_d\|_Y^2
\\
&=
\|y_d\|_Y^2 \min_{\tilde y \in  \closure_Y\left (\Psi(D, x_d)\right )} \left \| \tilde y - \frac{y_d}{\|y_d\|_Y} \right \|_Y^2
\leq 
\Theta(\Psi, x_d) \|y_d\|_Y^2 
\end{aligned}
\end{equation*}
holds for all $y_d \in Y \setminus \{0\}$. Combining the last two observations gives the desired 
estimate \eqref{eq:randomeq2736352}. This completes the proof. 
\end{proof}

As \cref{lemma:thetaprops} shows, the smaller the number $\Theta(\Psi, x_d)$, 
the better the ability of the function $\Psi(\cdot, x_d)$ to fit arbitrarily chosen label vectors 
$y_d \in Y$. However, since $\Theta(\Psi, x_d)$ is also a measure for the nonlinearity of 
the considered approximation scheme (at least in the case $\closure_Y\left (\Psi(D, x_d)\right ) \neq Y$),
one also has to expect that the optimization landscape 
of the problem \eqref{eq:trainingpropprot2} worsens as $\Theta(\Psi, x_d)$ tends to zero,
cf.\ the observations made in \cref{sec:4}. 
In \cref{subsec:5.2}, we will see that such an effect is indeed present and that 
the value of $\Theta(\Psi, x_d)$ also gives an estimate 
on how likely it is to encounter vectors $y_d$ for which the problem \eqref{eq:trainingpropprot2}
possesses spurious local minima and saddle points, cf.\ \cref{theorem:existencestatpts,theorem:badcone}. 

Before we turn our attention to this topic, we study the:

\subsection{Set-Valuedness and Discontinuity of the Best Approximation Map}
\label{subsec:5.1}

Recall that we have defined $P_\Psi^{x_d}$ to be the function that 
maps a label vector $y_d \in Y$ to the set of elements of the closure $\closure_Y\left (\Psi(D, x_d)\right )$
that attain the minimal loss in \eqref{eq:trainingpropprot2}, i.e., 
\begin{equation*}
P_{\Psi}^{x_d}\colon Y \rightrightarrows Y,\qquad y_d \mapsto \argmin_{y \in  \closure_Y\left (\Psi(D, x_d)\right )} \|y - y_d\|_Y^2. 
\end{equation*}
The purpose of this subsection is to analyze which consequences the properties 
\ref{fa:I} and \ref{fa:II} in \cref{ass:standingassumpssec4} have for this metric projection 
onto the set $\closure_Y\left (\Psi(D, x_d)\right )$ 
and the stability properties  of the training problem \eqref{eq:trainingpropprot2}. 
As talking about the map $P_\Psi^{x_d}$ is only
sensible when the closure of the image of $\Psi(\cdot, x_d)\colon D \to Y$ is not the whole of $Y$
(otherwise $P_\Psi^{x_d}$ is just the identity map),
throughout this subsection, we always assume the following:

\begin{assumption}
{\hspace{-0.05cm}\bf(Existence of Unrealizable Vectors)}\label[assumption]{ass:unrealizable_ass}
It holds $\closure_Y(\Psi(D, x_d)) \neq Y$.
\end{assumption}

Note that \cref{ass:unrealizable_ass} expresses that 
there exist label vectors $y_d \in Y$ that are unrealizable 
in the sense that they cannot be approximated by the function $\Psi(\cdot, x_d)\colon D \to Y$
up to an arbitrary tolerance and thus yield a positive optimal value of the loss in \eqref{eq:trainingpropprot2}.
Such situations occur when optimization problems \eqref{eq:trainingpropprot2} are considered 
that are (roughly speaking) not sufficiently overparameterized, i.e., problems in which 
the number of training samples is too high relative to the approximation 
capabilities of the considered approximation instrument $\psi$. Compare also with the comments after 
\cref{cor:NNinfiniteNonUniqueness} in this context.
We would like to emphasize that \cref{ass:unrealizable_ass} is only needed for the analysis of this subsection
and \cref{th:stevaluedspurious} in \cref{subsec:5.2}. 
For the derivation of our other results on stationary points and spurious local minima,
it is sufficient to assume that a local approximation of the function $\Psi(\cdot, x_d)\colon D \to Y$ is unable to 
fit arbitrary label vectors $y_d$ precisely, cf.\ \cref{prop:existencehotspurs,theorem:badcone,cor:instabilityoverpara}. 
The starting point for our study of the properties of the function $P_\Psi^{x_d}$ is the following observation: 

\begin{lemma}
\label[lemma]{lemma:propertiesydnonoverpar}
Suppose that \cref{ass:standingassumpssec4,ass:unrealizable_ass} hold. Then, 
we have  $\Theta(\Psi, x_d) \in (0, 1)$ and, 
for every 
$\bar y_d \in Y$  with the properties in \eqref{eq:randomeq273635}, the following is true:
\begin{enumerate}[label=\roman*)]
\item\label{item:propyd:i}
The set $P_\Psi^{x_d}(\bar y_d)$ contains more than one element.
\item\label{item:propyd:ii}
The set $P_\Psi^{x_d}(\bar y_d)$ is a subset of the affine-linear space 
\begin{equation}
\label{eq:Hdef}
H := (1 - \Theta(\Psi, x_d) ) \bar y_d + \bar y_d^\perp.
\end{equation}
Here,
$\bar y_d^\perp$
denotes the orthogonal complement
$\bar y_d^\perp := \{z \in Y \mid (\bar y_d, z)_Y = 0\}$.
\item\label{item:propyd:iii}
 It holds 
$\left (1 - \Theta(\Psi, x_d)\right) \bar y_d \in  \conv(P_\Psi^{x_d}(\bar y_d))$,
where $\conv(\cdot)$ denotes the convex hull.
\end{enumerate}
\end{lemma}

\begin{proof}
Since $\closure_Y(\Psi(D, x_d)) \neq Y$ holds and since the set $\closure_Y(\Psi(D, x_d))$ is a closed cone,
there exists at least one $y  \in Y$ with $\|y\|_Y = 1$ and $\dist(y, \closure_Y\left (\Psi(D, x_d)\right )) > 0$.
This shows that the error bound $\Theta(\Psi, x_d)$ has to be positive in the situation of the lemma
and, in combination with \cref{lemma:thetaprops}, that $\Theta(\Psi, x_d) \in (0, 1)$ holds as claimed. 
To prove the remaining assertions \ref{item:propyd:i}, \ref{item:propyd:ii}, and \ref{item:propyd:iii}, 
let us assume that 
an arbitrary but fixed worst-case unit label vector
$\bar y_d \in Y$ as in \eqref{eq:randomeq273635}
is given. (Recall that the existence of such a $\bar y_d$ is guaranteed by \cref{lemma:thetaprops}.)
Then, we obtain from \cref{lemma:nonemptyproj} that the set $P_\Psi^{x_d}(\bar y_d)$
contains at least one element $\bar y \in Y$ and it follows from 
the second equality in \eqref{eq:randomeq273635},
the fact that $\Theta(\Psi, x_d)$ is smaller than one, 
and the definition of $P_\Psi^{x_d}(\bar y_d)$
that $\bar y \neq 0$ has to hold. 
Since $\bar y$ is a solution of the problem 
\begin{equation*}
\min_{y \in  \closure_Y\left (\Psi(D, x_d)\right )} \|y - \bar y_d\|_Y^2
\end{equation*}
and again due to the cone property of the set $\closure_Y(\Psi(D, x_d))$, we moreover have
\begin{equation}
\label{eq:randomeq283645}
\|\bar y - \bar y_d\|_Y^2 \leq \|s \bar y - \bar y_d\|_Y^2\qquad \forall s >0.
\end{equation}
Choosing parameters of the form $s = 1 - \delta$ with $0 < |\delta| < 1$ in \eqref{eq:randomeq283645},
using the binomial identities, dividing by $\delta$, and passing to the limit $\delta \to 0$
yields that $(\bar y - \bar y_d, \bar y)_Y = 0$ has to hold
and, as a consequence, 
\begin{equation*}
\left (\bar y - (1 - \Theta(\Psi, x_d) ) \bar y_d, \bar y_d \right )_Y
=
\left (\bar y -  \bar y_d, \bar y_d - \bar y \right )_Y + 
\Theta(\Psi, x_d) \left \| \bar y_d \right \|_Y^2
= 
0.
\end{equation*}
The above shows that $\bar y$ is contained in the affine subspace $H$ in \eqref{eq:Hdef} and,
since $\bar y$ was an arbitrary element of $P_\Psi^{x_d}(\bar y_d)$, that 
$P_\Psi^{x_d}(\bar y_d) \subset H$.
This establishes \ref{item:propyd:ii}.
To prove \ref{item:propyd:iii}, we use a contradiction argument:
Suppose that the vector $(1 - \Theta(\Psi, x_d) ) \bar y_d \in H$
is not an element of the convex hull $\conv(P_\Psi^{x_d}(\bar y_d)) \subset H$. Then, 
by noting that the set $\conv(P_\Psi^{x_d}(\bar y_d))$ is
compact due to the finite-dimensionality of $Y$ and the compactness of $P_\Psi^{x_d}(\bar y_d)$, see \cref{lemma:nonemptyproj}, 
and by applying the strong hyperplane separation theorem to the sets 
$(1 - \Theta(\Psi, x_d) ) \bar y_d + \R \bar y_d $ and 
$\conv(P_\Psi^{x_d}(\bar y_d)) + \R \bar y_d $, see  \cite[Theorem 2.39]{RockafellarWets1998}, we obtain that  
there exist a nonzero $z \in Y$ and constants $c \in \R$ and $\varepsilon > 0$ such that 
\begin{equation*}
\left ( z, v_1\right )_Y \leq c - \varepsilon  < c \leq \left ( z, v_2\right )_Y\qquad 
\end{equation*}
holds for all $v_1 \in (1 - \Theta(\Psi, x_d) ) \bar y_d + \R \bar y_d $ and $v_2 \in  \conv(P_\Psi^{x_d}(\bar y_d)) + \R \bar y_d$.
Note that the above is only possible if $z \in \bar y_d^\perp$ and $c \geq \varepsilon$.
We may thus conclude that $z$, $c$, and $\varepsilon$
satisfy 
\begin{equation}
\label{eq:z}
\varepsilon \leq c  \leq \left ( z, \bar y\right )_Y\qquad \forall \bar y \in \conv(P_\Psi^{x_d}(\bar y_d)). 
\end{equation}
Consider now for 
all sufficiently small $\tau > 0$
the vectors
$y_d^\tau := (\bar y_d - \tau z)/ \|\bar y_d - \tau z\|_Y$ 
and select arbitrary but fixed $w_\tau \in  P_\Psi^{x_d}( y_d^\tau)$.
(Recall that the sets $P_\Psi^{x_d}( y_d^\tau)$ are nonempty by \cref{lemma:nonemptyproj}.)
Then, it follows from \eqref{eq:randomeq2736352},
the definition of $P_\Psi^{x_d}$, and the second property in \eqref{eq:randomeq273635}
that $w_\tau$, $y_d^\tau$, and $\bar y_d$ satisfy 
\begin{equation}
\label{eq:randomeq276335}
\|y_d^\tau - w_\tau\|_Y^2 \leq \Theta(\Psi, x_d) \leq \|\bar y_d- w_\tau\|_Y^2.
\end{equation}
Using the binomial identities, the properties 
$\|y_d^\tau\|_Y = \|\bar y_d\|_Y = 1$ and $z \in \bar y_d^\perp$, 
and the definition of $y_d^\tau$ in \eqref{eq:randomeq276335} yields
\begin{equation*}
\begin{aligned}
0 &\geq \left (y_d^\tau - \bar y_d , - w_\tau \right )_Y 
\\
&= \left (\frac{\bar y_d - \tau z}{\|\bar y_d - \tau z\|_Y} - \bar y_d , - w_\tau \right )_Y 
\\
&= 
\frac{\tau }{\|\bar y_d - \tau z\|_Y}  \left (z  , w_\tau \right )_Y 
+
\left (\frac{\|\bar y_d - \tau z\|_Y - 1}{\|\bar y_d - \tau z\|_Y}\right )
\left ( \bar y_d ,  w_\tau \right )_Y 
\\
&= 
\frac{\tau }{\|\bar y_d - \tau z\|_Y}  
\left (\left (z  , w_\tau \right )_Y 
+
 \frac{\tau \|z\|_Y^2}{ 1 +  \|\bar y_d - \tau z\|_Y} 
\left ( \bar y_d ,  w_\tau \right )_Y \right ). 
\end{aligned}
\end{equation*}
Since the family $\{w_\tau\}$ is necessarily bounded (see the first inequality in Equation \ref{eq:randomeq276335}),
the above implies that there exists a $\tau_0 > 0$ such that
$\left (z  , w_\tau \right )_Y  \leq \varepsilon/2$
holds for all $0 < \tau < \tau_0$, where $\varepsilon$ is the constant in \eqref{eq:z},
and, again by \eqref{eq:z}, that there exists an $\tilde \varepsilon > 0$ with 
$\dist(w_\tau, P_\Psi^{x_d}(\bar y_d)) \geq \tilde \varepsilon$ for all $0 < \tau < \tau_0$.
However, from the boundedness of $\{w_\tau\}$
and the closedness of the set $\closure_Y\left (\Psi(D, x_d)\right )$,
we also obtain that we can find a sequence $\{w_{\tau_i}\}$ 
with $(0, \tau_0) \ni \tau_i \to 0$ and 
$w_{\tau_i} \to w$ for some $w \in \closure_Y\left (\Psi(D, x_d)\right ) $, and,
due to the convergence $y_d^\tau \to \bar y_d$ for $\tau \to 0$ and \eqref{eq:randomeq276335}, such a $w$ 
clearly has to satisfy
\begin{equation*}
\|\bar y_d- w\|_Y^2 =   \Theta(\Psi, x_d) = \inf_{\alpha \in D}  \|\Psi(\alpha, x_d) - \bar y_d\|_Y^2.
\end{equation*}
The above yields $w \in P_\Psi^{x_d}(\bar y_d)$ and, as a consequence, 
\begin{equation*}
0 < \tilde \varepsilon \leq \dist(w , P_\Psi^{x_d}(\bar y_d)) = 0
\end{equation*}
which is  not possible. 
The vector $(1 - \Theta(\Psi, x_d) ) \bar y_d$ thus has to be an element of the 
set $\conv(P_\Psi^{x_d}(\bar y_d))$ and the proof of \ref{item:propyd:iii} is complete. 
Since the assertion in \ref{item:propyd:i} is a trivial consequence of  \ref{item:propyd:ii},  \ref{item:propyd:iii},
and the fact that $ \Theta(\Psi, x_d)$ is smaller than one by \cref{lemma:thetaprops}, 
this concludes the proof of the lemma. 
\end{proof}

Note that \cref{lemma:propertiesydnonoverpar} implies that the
cone $\closure_Y(\Psi(D, x_d))$ can only be convex if it is equal to the whole space $Y$.
By exploiting the properties \ref{fa:I} and \ref{fa:II} directly, 
we can also establish the following, stronger result
on the geometry of this set:
 
\begin{proposition}[Nonexistence of Solar Points]
\label[proposition]{prop:solar}~
Let \cref{ass:standingassumpssec4,ass:unrealizable_ass} hold.
Then, for every $y_d \in Y \setminus \closure_Y(\Psi(D, x_d))$ and every 
arbitrary but fixed $\bar y \in  P_\Psi^{x_d}(y_d)$, it is true that
\begin{equation}
\label{eq:randomeq17363eb36wgs}
\bar y \not \in P_\Psi^{x_d}\left (\bar y + s\frac{(y_d- \bar y)}{\|y_d- \bar y\|_Y}\right )
\qquad \forall s \in \R \text{ with } |s| > \left (\frac{\Theta(\Psi, x_d)}{1 - \Theta(\Psi, x_d)}\right)^{1/2} \|\bar y\|_Y .
\end{equation}
In particular, the set $\closure_Y(\Psi(D, x_d))$ does not admit any solar points, i.e., 
there do not exist any $y_d \in Y \setminus \closure_Y(\Psi(D, x_d))$ such that 
there is a $\bar y \in P_\Psi^{x_d}(y_d)$ with
\begin{equation*}
\bar y \in P_\Psi^{x_d}(\bar y + s (y_d- \bar y))
\qquad \forall s \in (0, \infty). 
\end{equation*}
\end{proposition}

\begin{proof}
Consider an arbitrary but fixed $y_d \in Y \setminus \closure_Y(\Psi(D, x_d))$ and some $\bar y \in  P_\Psi^{x_d}(y_d)$.
Then, it necessarily holds $y_d \neq \bar y$, 
and we obtain from the same arguments as in the proof of \cref{lemma:propertiesydnonoverpar}
that $(\bar y - y_d, \bar y)_Y = 0$ has to hold. 
Define $v := (y_d- \bar y)/\|y_d- \bar y\|_Y$
and $y_d^s := \bar y + s v \in Y$ for all $s \in \R$,
and let $\bar y_s \in Y$ be arbitrary but fixed elements of the sets $P_\Psi^{x_d}(y_d^s)$ for all $s \in \R$.
Then, from \eqref{eq:randomeq2736352}, the definition of $y_d^s$, the orthogonality 
between $v$ and $\bar y$, and the equation $\|v\|_Y = 1$, we obtain that
\begin{equation}
\label{eq:randomeq26353672}
\begin{aligned}
\|\bar y_s - y_d^s\|_Y^2  - \|\bar y - y_d^s \|_Y^2 
&\leq
\Theta(\Psi, x_d) \|y_d^s\|_Y^2   - \|s v\|_Y^2
\\
&=
\Theta(\Psi, x_d) \|\bar y\|_Y^2 + (\Theta(\Psi, x_d) - 1)s^2
\end{aligned}
\end{equation}
holds for all $s \in \R$. 
Since the set $\Psi(D, x_d)$ is dense in $\closure_Y(\Psi(D, x_d))$ and since 
$\Theta(\Psi, x_d)$ is an element of the interval $(0, 1)$ by \cref{lemma:propertiesydnonoverpar},
the above shows that, for all $s \in \R$ with
\begin{equation*}
|s| > \left (\frac{\Theta(\Psi, x_d)}{1 - \Theta(\Psi, x_d)}\right)^{1/2} \|\bar y\|_Y,
\end{equation*}
we have 
\begin{equation*}
\inf_{\alpha \in D} \|\Psi( \alpha, x_d) - y_d^s\|_Y^2 < \|\bar y - y_d^s\|_Y^2
\end{equation*}
and, as a consequence, $\bar y \not \in P_\Psi^{x_d}(\bar y + s v)$. This establishes the first claim
of the proposition. The second one is an immediate consequence. 
\end{proof}

It is easy to check that the property \eqref{eq:randomeq17363eb36wgs} 
implies that the set $\closure_Y(\Psi(D, x_d))$ does not admit
\emph{any} supporting hyperplanes in the situation of \cref{prop:solar}.
This shows that the cone $\closure_Y(\Psi(D, x_d))$ indeed has to be highly nonconvex 
if the map $\Psi(\cdot, x_d)\colon D \to Y$ satisfies \ref{fa:I} and \ref{fa:II}
and there exist unrealizable vectors. 
Compare also with the geometry of the set in \cref{fig:sub2})
in this context. 
For further details on solar points and their role in the field of nonlinear approximation theory, 
we refer the reader to \cite{Braess1986}. We remark that arguments very similar to those in the  proof of  \cref{prop:solar}
will also be used in \cref{subsec:5.2} for the derivation of our results on saddle points and spurious minima. 

To study which consequences the inclusion 
in point \ref{item:propyd:iii} of \cref{lemma:propertiesydnonoverpar} has for 
the continuity properties of the map $P_\Psi^{x_d}$, we need:
\begin{lemma}[A Variant of Jung's Inequality]
\label[lemma]{lemma:jung}
Suppose that $H$ is an affine-linear subspace of $Y$
with dimension $d \in \{1,2,..., \dim(Y)\}$. Assume further that a point $\bar y \in H$, a
compact set $E \subset H$, and a number $r > 0$ satisfying 
$\bar y \in \conv(E)$ and $\|\bar y - z\|_Y = r$ for all $z \in E$
are given. Then, it is true that\\[-0.3cm]
\begin{equation}
\label{eq:jung}
\sup_{z_1, z_2 \in E} \|z_1 - z_2\|_Y \geq \left ( \frac{2 d + 2}{d}\right )^{1/2}r.
\end{equation}
\end{lemma}

\begin{proof}
Note that, 
by introducing a suitably defined orthonormal basis and by
restricting the attention to the 
space of directions of the affine subspace $H$, 
we can always transform the situation considered in the lemma into that
with $Y = H = \R^d$ and $\|\cdot\|_Y = \|\cdot\|_2$, where $\|\cdot\|_2$ denotes the Euclidean norm. 
It thus suffices to prove \eqref{eq:jung} in the space $(\R^d, \|\cdot\|_2)$ for
all $\bar y \in \R^d$, compact sets $E \subset \R^d$, and constants $r>0$ 
that satisfy $\bar y \in \conv(E)$ and $\|\bar y - z\|_2 = r$ for all $z \in E$.
So let us assume that such $\bar y$, $E$, and $r$ are given, 
and suppose further that $B_R(v)$ is a closed ball in $(\R^d, \|\cdot\|_2)$
with center $v$ and radius $R$ that covers the set $E$.
Then, in the case $v=\bar y$, our assumption $\|\bar y - z\|_2 = r$ for all $z \in E$
immediately yields that $R \geq r$ has to hold. In what follows, we will show that
this inequality is also true for $v \neq \bar y$. To this end, 
we note that the inclusion $\bar y \in \conv(E)$ and
Carathéodory's theorem, see \cite[Theorem 1.2.5]{Borwein2010},
imply that there exist $\lambda_1,..., \lambda_{d+1} \in [0, 1]$ 
and $z_1,...,z_{d+1} \in E$ satisfying  $\sum_{i=1}^{d+1}\lambda_i = 1$
and $\sum_{i=1}^{d+1}\lambda_i z_i = \bar y$, and, as a consequence, 
\[
0 = \left (\bar y - v, \bar y - \bar y  \right )_2 = \sum_{i=1}^{d+1} \lambda_i \left (\bar y - v, z_i - \bar y  \right )_2.
\]
Here, $(\cdot, \cdot)_2$ denotes the Euclidean scalar product. The above implies in particular that 
there has to be at least one $j \in \{1,..., d+1\}$ with $\left (\bar y - v, z_j - \bar y  \right )_2 \geq 0$, 
and from this inequality and the inclusion $E \subset B_R(v)$, it follows straightforwardly that 
\begin{equation*}
R^2 \geq \|z_j - v\|_2^2 
= \|z_j - \bar y + \bar y- v\|_2^2 
= \|z_j - \bar y\|_2^2 + 2 \left (\bar y - v, z_j - \bar y  \right )_2 +  \|\bar y- v\|_2^2
\geq r^2. 
\end{equation*}
Thus, $R \geq r$ as claimed.
In summary, we have now proved that every closed ball $B \subset \R^d$ with $E \subset B$
has to have radius at least $r$. In combination with the classical 
inequality of Jung, see \cite[Theorem 11.1.1]{Burago1988}, this yields
\begin{equation*}
\sup_{z_1, z_2 \in E} \|z_1 - z_2\|_2  \left ( \frac{d}{2d + 2}\right)^{1/2} \geq r. 
\end{equation*}
Rearranging the above establishes \eqref{eq:jung} and completes the proof. 
\end{proof}\\[-1cm]\pagebreak[1]

By combining \cref{lemma:propertiesydnonoverpar,lemma:jung} and by using elementary properties of the map $P_\Psi^{x_d}$,
we now arrive at the following main result of this subsection:

\begin{theorem}{\hspace{-0.05cm}\bf(Nonuniqueness and Instability of Best Approximations)}\label{theorem:abstractinstability}
Suppose that \cref{ass:standingassumpssec4,ass:unrealizable_ass} hold.
Then, the best approximation map 
\begin{equation*}
P_{\Psi}^{x_d}\colon Y \rightrightarrows Y,\qquad y_d \mapsto \argmin_{y \in  \closure_Y\left (\Psi(D, x_d)\right )} \|y - y_d\|_Y^2,
\end{equation*}
associated with the training problem \eqref{eq:trainingpropprot2} has the following properties:
\begin{enumerate}[label=\roman*)]

\item\label{item:stabth:i}
There are uncountably many $y_d$ such that $P_{\Psi}^{x_d}(y_d)$ contains more than 
one element.

\item\label{item:stabth:ii}
The function $P_{\Psi}^{x_d}$  is discontinuous in the following sense:
For every arbitrary but fixed $C > 0$, there exists an uncountable set $\MM_C \subset Y$
such that, for every label vector $y_d \in \MM_C$, there exist sequences $\{y_d^l\}, \{\tilde y_d^l\} \subset Y$ with
\begin{equation}
\label{eq:randomeq2837}
\begin{gathered}
y_d^l \to y_d  \text{ for }l \to \infty,
\qquad
\tilde y_d^l \to y_d \text{ for }l \to \infty,
\\
|P_\Psi^{x_d}(y_d^l)| =  |P_\Psi^{x_d}(\tilde y_d^l)|  = 1 \quad \forall l,
\quad \text{and}\quad
\|P_\Psi^{x_d}(y_d^l) - P_\Psi^{x_d}(\tilde y_d^l)\|_Y \geq C\quad \forall l.
\end{gathered}
\end{equation}
Here, $|\cdot|$ denotes the cardinality of a set and with $\|P_\Psi^{x_d}(y_d^l) - P_\Psi^{x_d}(\tilde y_d^l)\|_Y$
we mean the distance between the elements of the singletons $P_\Psi^{x_d}(y_d^l)$ and $P_\Psi^{x_d}(\tilde y_d^l)$. 
Further, for every $C>0$, there exists at least one $y_d \in Y$ with the properties 
\begin{equation}
\label{eq:thetabound61253} 
[1, \infty)y_d \subset\MM_C\quad\text{and}\quad
\|y_d\|_Y = C \left ( \frac{\dim(Y) - 1}{2\dim(Y) (\Theta(\Psi, x_d) - \Theta(\Psi, x_d)^2)} \right )^{1/2}.
\end{equation}
\end{enumerate}
\end{theorem}

\begin{proof}
Let $\bar y_d \in Y$ be an arbitrary but fixed 
worst-case unit label vector as in \eqref{eq:randomeq273635}.
Then, from \cref{lemma:propertiesydnonoverpar}, it follows that 
$| P_\Psi^{x_d}(\bar y_d)| > 1$ holds, 
and we obtain from the conicity of the set $\closure_Y(\Psi(D, x_d))$ 
that $P_\Psi^{x_d}$ satisfies $P_\Psi^{x_d}(s y_d) = s P_\Psi^{x_d}(y_d)$ for all 
$y_d \in Y$ and all $s > 0$. 
Combining these two observations yields $|P_\Psi^{x_d}(s \bar y_d)|> 1$ for all $s > 0$
which proves the assertion of \ref{item:stabth:i}.
To establish \ref{item:stabth:ii}, we recall that, by \cref{lemma:propertiesydnonoverpar},
the compact set $E := P_\Psi^{x_d}(\bar y_d)$ has to satisfy 
$\left (1 - \Theta(\Psi, x_d)\right) \bar y_d \in  \conv(E) \subset H$,
where $H$ again denotes the affine subspace in \eqref{eq:Hdef},
and that the definition of $P_\Psi^{x_d}(\bar y_d)$ yields 
$(z - \bar y_d, z)_Y = 0$ for all $z \in E$ (see the first part of the proof of \cref{lemma:propertiesydnonoverpar}).
The latter implies, in combination with the properties of $\bar y_d$, that 
\begin{equation*}
\begin{aligned}
\left \|z -  \left (1 - \Theta(\Psi, x_d)\right) \bar y_d\right \|_Y^2
&=
\|z - \bar y_d\|_Y^2 
+ 2 \Theta(\Psi, x_d)\left (z - \bar y_d, \bar y_d \right )_Y
+ \Theta(\Psi, x_d)^2 
\\
&= \Theta(\Psi, x_d) - \Theta(\Psi, x_d)^2\qquad \forall z \in E.
\end{aligned}
\end{equation*}
The vector $\bar y:= \left (1 - \Theta(\Psi, x_d)\right) \bar y_d $ and the 
number $r := (\Theta(\Psi, x_d) - \Theta(\Psi, x_d)^2)^{1/2} > 0$
thus satisfy 
\begin{equation*}
\bar y \in  \conv(E) \subset H
\quad \text{and}\quad
\|\bar y - z\|_Y = r\quad \forall z \in E,
\end{equation*}
and we may invoke \cref{lemma:jung} to deduce that 
there exist $z_1, z_2 \in P_\Psi^{x_d}(\bar y_d)$ with 
\begin{equation*}
\|z_1 - z_2\|_Y \geq  \left ( \frac{2\dim(Y) (\Theta(\Psi, x_d) - \Theta(\Psi, x_d)^2)}{\dim(Y) - 1}\right )^{1/2}.
\end{equation*}
Consider now the sequence $y_d^l := (1 - 1/l) \bar y_d + (1/l) z_1$, $l \in \mathbb{N}$. 
Then, we clearly have $y_d^l  \to \bar y_d$ for $l \to \infty$ and 
it holds
\begin{equation}
\label{eq:randomeq23636}
\begin{aligned}
\|y_d^l - z\|_Y 
&=
\|(1 - 1/l) \bar y_d + (1/l) z_1  - z\|_Y
\\
&\geq \| \bar y_d  - z\|_Y - (1/l) \| \bar y_d - z_1 \|_Y
\\
&\geq (1 - 1/l) \Theta(\Psi, x_d)^{1/2}\qquad \forall z \in \closure_Y(\Psi(D, x_d))
\end{aligned}
\end{equation}
with equality everywhere if and only if $z = z_1$. 
In combination with the definition of $P_\Psi^{x_d}$,
this implies in particular that $P_\Psi^{x_d}(y_d^l) = \{z_1\}$ holds for all $l \in \mathbb{N}$.
Completely analogously, 
we also obtain that the vectors $\tilde y_d^l := (1 - 1/l) \bar y_d + (1/l) z_2$, $l \in \mathbb{N}$,
satisfy  $\tilde y_d^l  \to \bar y_d$ for $l \to \infty$ and $P_\Psi^{x_d}(\tilde y_d^l) = \{z_2\}$ for all $l \in \mathbb{N}$.
By again exploiting the positive homogeneity of 
the map $P_{\Psi}^{x_d}\colon Y \rightrightarrows Y$ and by combining all of the above, it now follows immediately 
that, for every arbitrary but fixed $C>0$ and all 
\begin{equation*}
s \geq C \left ( \frac{\dim(Y) - 1}{2\dim(Y) (\Theta(\Psi, x_d) - \Theta(\Psi, x_d)^2)}\right )^{1/2},
\end{equation*}
we have 
\begin{equation*}
s y_d^l \to s \bar y_d  \text{ for }l \to \infty,
\quad
s \tilde y_d^l \to s \bar y_d \text{ for }l \to \infty,
\end{equation*}
and
\begin{equation*}
P_\Psi^{x_d}(s y_d^l) = \{s z_1\},\qquad P_\Psi^{x_d}(s \tilde y_d^l) = \{s z_2\},
\qquad 
\|s z_1 - s z_2\|_Y \geq  C\quad \forall l \in \mathbb{N}. 
\end{equation*}
Since $\|\bar y_d\|_Y = 1$ holds by \eqref{eq:randomeq273635}, this establishes 
\ref{item:stabth:ii} and completes the proof. 
\end{proof}

Several remarks are in order regarding the last result: 

\begin{remark}~\label[remark]{rem:stability}
\begin{itemize}

\item \Cref{theorem:abstractinstability} shows that, if there exist label vectors $y_d$ that cannot be approximated 
up to arbitrary tolerances and if \ref{fa:I} and \ref{fa:II} hold, 
then the approximation scheme $\psi$ is always unable to provide 
unique best approximations for all possible choices of $y_d$ (see point one)
and arbitrarily small perturbations in $y_d$ can change the set of best approximations to an
arbitrarily large extent (see point two). This implies in particular that,
in the situation of \cref{theorem:abstractinstability}, the problem of finding best approximations for 
a given $y_d$ is always ill-posed in the sense of Hadamard for certain choices of $y_d$.

\item 
As already mentioned in the introduction, for neural networks with one hidden layer,
instability results similar to those in \cref{theorem:abstractinstability} 
have already been proved in the $L^p$-spaces by \cite{Kainen1999,Kainen2001}
by exploiting classical instruments from nonlinear approximation theory. 
The finite-dimensionality of the training problem \eqref{eq:trainingpropprot2} 
allows us to show---not only for one-hidden-layer networks but 
for all approximation schemes satisfying the conditions \ref{fa:I} and \ref{fa:II} and $\closure_Y(\Psi(D, x_d)) \neq Y$---that 
the instability of the best approximation map $P_\Psi^{x_d}$ associated with \eqref{eq:trainingpropprot2}  is directly linked to the 
number $\Theta(\Psi, x_d)$ in \eqref{eq:defTheta} which also measures the worst-case 
approximation error achievable with the function $\Psi(\cdot, x_d)\colon D \to Y$, 
see \eqref{eq:randomeq2736352}  and \eqref{eq:thetabound61253}. 
(Note that the arguments that we have used 
to establish \eqref{eq:thetabound61253} indeed only work in the finite-dimensional setting, cf.\ 
the proofs of \cref{lemma:propertiesydnonoverpar,lemma:jung}.)

\item The instability 
properties in \cref{theorem:abstractinstability} are of a different type than those 
arising, e.g., in a least-squares problem of the form 
\begin{equation*}
\min_{\alpha \in \R^m} \left \| A \alpha - y_d \right \|_2^2
\end{equation*}
with given $y_d \in \R^n$, $A \in \R^{n \times m}$, and $n \geq m$,
when the matrix $A^TA \in \R^{m \times m}$ (i.e., the matrix in the normal equation)
is ill-conditioned or singular. Indeed, as we have seen in \cref{sec:4}, 
for approximation schemes that depend linearly on 
$\alpha$, the map $P_\Psi^{x_d}$ is always a metric projection 
onto a linear subspace of $Y$ and thus necessarily single-valued and 
globally one-Lipschitz. The set-valuedness and the discontinuity of the function 
$P_\Psi^{x_d}$ in \cref{theorem:abstractinstability} are effects that 
can only be encountered in the nonlinear setting as they 
stem from curvature properties of the set $\closure_Y(\Psi(D, x_d))$. 
Instability properties arising from a particular choice of 
the parameterization of the set $\closure_Y(\Psi(D, x_d))$ 
via the parameter $\alpha$ come on top of the effects documented in \cref{theorem:abstractinstability}. 

\item
It is easy to check (e.g., by means of the examples 
$\closure_Y(\Psi(D, x_d)) = \R\bar y$, $\bar y \in Y$ arbitrary but fixed, and $\closure_Y(\Psi(D, x_d)) = B_1^Y(0)$, 
and by observing that $P_\Psi^{x_d}$ is the identity map when $\closure_Y\left (\Psi(D, x_d)\right ) = Y$ holds)
that neither the conditions in \cref{ass:standingassumpssec4}
nor the assumption $\closure_Y\left (\Psi(D, x_d)\right ) \neq Y$ can 
be dropped for \cref{theorem:abstractinstability} to be true. 

\item Note that the right-hand side of the identity in \eqref{eq:thetabound61253} tends to infinity when
$\Theta(\Psi, x_d)$ goes to  zero or one, respectively. This makes sense as
the function $\psi$ behaves more and more like a linear approximation scheme
when $\Theta(\Psi, x_d)$ converges to one (at least as far as the 
worst-case approximation error is concerned, cf.\ \cref{sec:4}), and since, in the limit $\Theta(\Psi, x_d) \to 0$, 
one recovers the case with $\closure_Y(\Psi(D, x_d)) = Y$, so that, for both 
$\Theta(\Psi, x_d) \to 0$ and $\Theta(\Psi, x_d) \to 1$, the
setting considered in \cref{theorem:abstractinstability} approximates 
a situation in which the map $P_\Psi^{x_d}$ is single-valued and continuous. 

\item 
The nonuniqueness in point \ref{item:stabth:i} of \cref{theorem:abstractinstability}
has nothing to do with, e.g., a non-injective parameterization of the set 
$\Psi(D, x_d)$ via the variable $\alpha$ as present, for instance, in neural networks due to symmetries. 
On the contrary, as $P_{\Psi}^{x_d}(y_d)$ is defined as the set of best approximations for a given $y_d$ in 
the space $Y$,
the set-valuedness of $P_{\Psi}^{x_d}$ implies (just by taking preimages) that for some choices of $y_d$ 
there are different 
parameters $\alpha$ (or, at least, minimizing sequences) which yield the same optimal loss in \eqref{eq:trainingpropprot2}
but give rise to maps $\psi(\alpha, \cdot)\colon \XX \to \YY$ that behave 
differently not only on unseen data but even on the data in $x_d$
that the approximation scheme is trained on. Compare with the 
remarks after \cref{lemma:nonemptyproj} in this context and also with the next comment.

\item The discontinuity properties of the map $P_\Psi^{x_d}$  in point \ref{item:stabth:ii} 
of \cref{theorem:abstractinstability} imply that,
if we solve the training problem \eqref{eq:trainingpropprot2} with a descent method and, 
by doing so, obtain a sequence of parameters $\{\alpha_i\}$ satisfying \eqref{eq:infsequence}
and $\Psi(\alpha_i, x_d) \to \bar y$ for some $\bar y \in Y$, then an arbitrarily small perturbation 
of the training label vector $y_d$ can cause 
the solution algorithm to produce a different sequence  $\{\tilde \alpha_i\}$,
which again satisfies \eqref{eq:infsequence} and, in the limit $i \to \infty$,
yields a loss that is arbitrarily close to that obtained with $\{\alpha_i\}$, 
but satisfies $\Psi(\tilde \alpha_i, x_d) \to \tilde y$ with a vector $\tilde y \in Y$ that is 
arbitrarily far away from $\bar y$. Note that the latter again implies that 
the functions $\psi(\alpha_i, \cdot)\colon \XX \to \YY$ and $\psi(\tilde \alpha_i, \cdot)\colon  \XX \to \YY$ 
behave differently on the training data as $i$ tends to infinity. 
\end{itemize}
\end{remark}

\subsection{Existence of Spurious Local Minima and Saddle Points}
\label{subsec:5.2}
Having discussed the properties of the map $ P_\Psi^{x_d}$,
we now turn our attention to the question of whether the 
problem \eqref{eq:trainingpropprot2} possesses saddle points and spurious local minima. 
We begin with a result that
builds upon the findings of \cref{theorem:abstractinstability} 
and shows that, in the presence of unrealizable vectors, 
the training problem \eqref{eq:trainingpropprot2} can only lack spurious 
local minima and non-optimal basins stretching to 
the boundary of the parameter set $D$ 
for all $y_d \in Y$ if the image of the function $Y\ni y_d \mapsto | P_\Psi^{x_d}(y_d)| \in \mathbb{N} \cup \{\infty\}$
is equal to $\{1, \infty\}$, i.e., if the space $Y$ can be decomposed into 
two nonempty disjoint sets $Y_1$ and $Y_2$ such that every $y_d \in Y_1$
possesses exactly one best approximation in $\closure_Y(\Psi(D, x_d))$
and such that, for every $y_d \in Y_2$, there are infinitely many best approximations in $\closure_Y(\Psi(D, x_d))$. 

\begin{theorem}{\hspace{-0.05cm}\bf(Relation Between Set-Valuedness and Spurious Minima/Basins)}\label{th:stevaluedspurious}
Let \cref{ass:standingassumpssec4,ass:unrealizable_ass} hold.
Assume further that the function $\Psi(\cdot, x_d)\colon D \to Y$ is continuous
and that the image of the map 
$Y\ni y_d \mapsto | P_\Psi^{x_d}(y_d)| \in \mathbb{N} \cup \{\infty\}$ is not equal to $\{1, \infty\}$
(where $|\cdot|$ again denotes the cardinality of a set).
Then, there exist an open nonempty cone $K \subset Y$ and a number $M \in \mathbb{N}$ with $M \geq 2$
such that, for every $y_d \in K$, there are nonempty, disjoint, relatively closed subsets $D_1,..., D_M$ 
of the set $D \subset \R^m$ with 
\begin{equation}
\label{eq:spurpro1}
\inf_{\alpha \in D_1} \|\Psi( \alpha, x_d) - y_d\|_Y^2
< \inf_{\alpha \in D_i} \|\Psi( \alpha, x_d) - y_d\|_Y^2\quad \forall i=2,...,M
\end{equation}
and
\begin{equation}
\label{eq:spurpro2}
\sup_{\alpha \in D_1 \cup ... \cup D_M} \|\Psi( \alpha, x_d) - y_d\|_Y^2 < \|\Psi( \tilde \alpha, x_d) - y_d\|_Y^2 
\quad \forall \tilde \alpha \in D \setminus (D_1 \cup ... \cup D_M).
\end{equation}
\end{theorem}

\begin{figure}[H]
\centering
\begin{tikzpicture}[scale=0.9]

\draw[arrows={-Triangle[angle=60:4pt, black,fill=black]},  line width=1pt]  (0, -3) -- (0, 3);
\draw[arrows={-Triangle[angle=60:4pt, black,fill=black]},  line width=1pt]  (-4, 0) -- (4, 0);

\draw [line width=1pt, dashed] plot [smooth] coordinates {
(0.2,2.75)
(0.7, 0.75)
(4, 1.2)
};
\draw [line width=1pt, dashed] plot [smooth] coordinates {
(-0.5,2.75)
(-0.1, -0.4)
(4, -2.8)
};
\draw [line width=1pt, dashed] plot [smooth] coordinates {
(-1.5,2.75)
(-1.1,2)
(-1.1, 0)
(2,  -3)
};
\draw [line width=1pt, dashed] plot [smooth] coordinates {
(-4, 0.1)
(-3, -0.2)
(-1.2, -1)
(-0.1,  -3)
};

\draw [line width=1pt, fill=gray!45] plot [smooth] coordinates {
(1,2.75)
(0.9, 1)
(3.75, 2.75)
};
\draw [line width=1pt, dashed] plot [smooth] coordinates {
(1.3,2.75)
(1.2, 1.5)
(3.4, 2.75)
};
\draw [line width=1pt, dashed] plot [smooth] coordinates {
(1.5,2.75)
(1.5, 2)
(3.15, 2.75)
};
\draw [line width=1pt, dashed] plot [smooth] coordinates {
(1.7,2.75)
(1.8, 2.4)
(2.7, 2.75)
};
\draw[arrows={-Triangle[angle=60:4pt, black,fill=black]},  line width=1pt] (0.55,0.7) --  (1, 1.2);
\node at (3.1, 0.4) {$D_1$ containing $\{\alpha_i\}$ as in \eqref{eq:infsequence}};

\draw [line width=1pt, fill=gray!45] plot [smooth cycle] coordinates {
(-2,2.75)
(-1.5, 1.5)
(-3.75, 2)
};
\draw [line width=1pt, dashed] plot [smooth cycle] coordinates {
(-2.2,2.4)
(-1.8, 1.7)
(-3.25, 2)
};
\draw [line width=1pt, dashed] plot [smooth cycle] coordinates {
(-2.3,2.2)
(-2.2, 1.9)
(-2.7, 2)
};

\draw [line width=1pt, fill=gray!45] plot [smooth] coordinates {
(-4, 1.7)
(-2, 1)
(-4, 0.9)
};
\draw [line width=1pt, dashed] plot [smooth] coordinates {
(-4, 1.5)
(-2.8, 1.15)
(-4, 1.1)
};
\node at (-4.2, 2.65) {$D_2$};
\draw[arrows={-Triangle[angle=60:4pt, black,fill=black]},  line width=1pt] (-4.2, 2.45) --  (-3.5, 2.0);
\draw[arrows={-Triangle[angle=60:4pt, black,fill=black]},  line width=1pt] (-4.2, 2.45) --  (-3.9, 1.5);
\node at (-2.8, 0.4) {spurious basin};
\draw[arrows={-Triangle[angle=60:4pt, black,fill=black]},  line width=1pt] (-2.8, 0.6) --  (-3.2, 1.2);

\draw [line width=1pt, fill=gray!45] plot [smooth cycle] coordinates {
(-2,-1)
(-1, -2)
(-2, -2.5)
(-3, -1)
};
\draw [line width=1pt, dashed] plot [smooth cycle] coordinates {
(-1.8,-1.4)
(-1.3, -2)
(-2, -2.2)
(-2.65, -1.15)
};
\draw [line width=1pt, dashed] plot [smooth cycle] coordinates {
(-1.7, -1.7)
(-1.9, -2)
(-2.25, -1.4)
};
\draw[arrows={-Triangle[angle=60:4pt, black,fill=black]},  line width=1pt] (0.3, -0.45) --  (-2.0, -1.7);
\draw[arrows={-Triangle[angle=60:4pt, black,fill=black]},  line width=1pt] (0.3, -0.45) --  (-2.4, 2.1);
\node at (2.45, -0.45) {spurious local minima};
\draw[arrows={-Triangle[angle=60:4pt, black,fill=black]},  line width=1pt] (2.05, -1.5)  --  (1.75, -1.6) ;
\node at (2.85, -1.5) {level set};
\node at (-2.95, -2.5) {$D_3$};
\draw[arrows={-Triangle[angle=60:4pt, black,fill=black]},  line width=1pt] (-2.8, -2.3) --  (-2.3, -1.9);

\end{tikzpicture}
\caption{Geometric meaning of \eqref{eq:spurpro1} and \eqref{eq:spurpro2} for $D = \R^2$ and $M=3$.
Condition \eqref{eq:spurpro2} implies that the sublevel set 
$\Omega_L(c) := \{\alpha \in D \mid \| \Psi( \alpha, x_d) - y_d\|_Y^2 \leq c\}$ (gray)
of the loss function associated with the number 
${c := \sup_{\alpha \in D_1 \cup ... \cup D_M} \|\Psi( \alpha, x_d) - y_d\|_Y^2}$ is equal to the 
union of the nonempty, disjoint, and closed sets $D_1, ..., D_M$. 
This entails that each $D_i$ contains (at least) one nonempty connected 
component of the sublevel set $\Omega_L(c)$. Condition \eqref{eq:spurpro1} yields 
 that only the connected components of $\Omega_L(c)$ that are subsets of $D_1$
 can contain a minimizing sequence $\{\alpha_i\}$ as in \eqref{eq:infsequence}. 
Thus, all connected components of the sets $D_2,..., D_M$ contain spurious local minima or spurious basins.  }
\label{fig:spuriousvalleyillustration}
\end{figure}
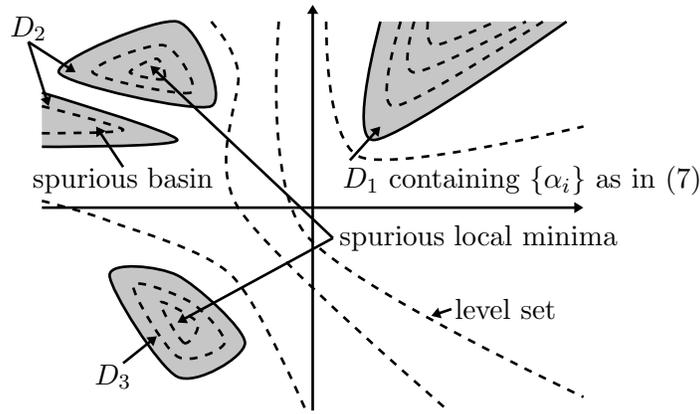

\begin{proof}
If the image of the map $Y\ni y_d \mapsto | P_\Psi^{x_d}(y_d)| \in \mathbb{N} \cup \{\infty\}$ is not equal to $\{1, \infty\}$,
then 
it follows from \cref{lemma:nonemptyproj} and \cref{theorem:abstractinstability}\ref{item:stabth:i} that there has to exist 
at least one $\bar y_d \in Y \setminus \{0\}$ with $1 < |P_{\Psi}^{x_d}(\bar y_d)| < \infty$.
Define $M := |P_{\Psi}^{x_d}(\bar y_d)|$ and $r := \dist(\bar y_d, P_{\Psi}^{x_d}(\bar y_d)) > 0$,
and let us denote the 
$M$ distinct elements of the set $P_{\Psi}^{x_d}(\bar y_d)$
with $\bar y_i$, $i=1,...,M$. Consider further an $\varepsilon > 0$
such that the closed balls $B_\varepsilon^Y(\bar y_i)$, $i=1,...,M$, 
satisfy $\dist(B_\varepsilon^Y(\bar y_i), B^Y_\varepsilon(\bar y_j)) > 2\varepsilon$
for all $i \neq j$.
Then, it follows from the definition of 
$P_{\Psi}^{x_d}(\bar y_d)$ that
there exists a number $\delta \in (0, \varepsilon]$ with 
$\dist(B_r^Y(\bar y_d) \setminus (B_\varepsilon^Y(\bar y_1) \cup ... \cup B_\varepsilon^Y(\bar y_M)),\closure_Y(\Psi(D, x_d))) \geq \delta$.
Using this $\delta$, we define 
$E_i := B_{\varepsilon + \delta}^Y(\bar y_i) \cap B^Y_{r + \delta/2}(\bar y_d) \cap \closure_Y(\Psi(D, x_d))$.
Note that this construction ensures that the sets $E_i$, $i=1,...,M$, 
are nonempty, compact, and disjoint.
From the choice of $\delta$ and $\varepsilon$, we further obtain that the sets $E_i$, $i=1,...,M$, satisfy
\begin{equation}
\label{eq:randomeq2736365}
B_{r + \delta/2}^Y(\bar y_d) \cap \closure_Y(\Psi(D, x_d)) = \bigcup_{i=1}^M E_i. 
\end{equation}
Indeed, 
the inclusion ``$\supset$'' in the equality \eqref{eq:randomeq2736365} follows immediately from the definition of the sets $E_i$,
and if there was a $\smash{\bar y \in B_{r + \delta/2}^Y(\bar y_d) \cap \closure_Y(\Psi(D, x_d)) \setminus \bigcup_{i=1}^M E_i}$,
then this vector $\bar y$ would satisfy $\dist(\bar y, B_r^Y(\bar y_d)) \leq \delta/2$ 
and $\|\bar y - \bar y_i\|_Y > \varepsilon + \delta$ for all $i=1,...,M$ which would 
imply the existence of a $\tilde y \in B_r^Y(\bar y_d) \setminus (B^Y_\varepsilon(\bar y_1) \cup ... \cup B^Y_\varepsilon(\bar y_M))$
with $\|\bar y - \tilde y\|_Y \leq \delta/2$ and thus contradict the definition of $\delta$. 

To prove the claim of the theorem, we now consider the vector 
\begin{equation*}
\tilde y_d := \bar y_d + \frac{\delta}{8}\frac{\bar y_1 - \bar y_d}{\|\bar y_1 - \bar y_d\|_Y}.
\end{equation*}
Note that, by exactly the same arguments as in \eqref{eq:randomeq23636}, we obtain that this $\tilde y_d$
satisfies $P_{\Psi}^{x_d}(\tilde y_d) = \{\bar y_1\}$ and, as a consequence, 
\begin{equation*}
\begin{aligned}
\|\bar y_1 - \tilde y_d\|_Y 
&= \min_{y \in  E_1} \|y - \tilde y_d\|_Y 
\\
&= 
\min_{y \in  \closure_Y\left (\Psi(D, x_d)\right )} \|y - \tilde y_d\|_Y =
r - \frac{\delta}{8}.
\end{aligned}
\end{equation*}
The above implies in particular that 
\begin{equation}
\label{eq:randomeq36354}
\begin{gathered}
\dist(\tilde y_d, E_1) = \dist(\tilde y_d, \closure_Y\left (\Psi(D, x_d)\right )) < r,
\\
\dist(\tilde y_d, E_1) < \dist(\tilde y_d, E_i)\quad \forall i=2,...,M.
\end{gathered}
\end{equation}
Due to the Lipschitz continuity of the 
distance functions in \eqref{eq:randomeq36354}, the definitions of $\tilde y_d$
and $E_i$, and \eqref{eq:randomeq2736365}, 
the estimates in \eqref{eq:randomeq36354} remain 
valid for all $y_d$ that are sufficiently close to $\tilde y_d$.
We can thus find a $\tau  \in (0, \delta/8)$ such that, 
for every $y_d \in Y$ with $\|y_d - \tilde y_d\|_Y < \tau$, we have
\begin{equation}
\label{eq:randomeq36354-2}
\begin{gathered}
\dist(y_d, E_1) = \dist(y_d, \closure_Y\left (\Psi(D, x_d)\right )) < r,
\\
\dist(y_d, E_1) < \dist(y_d, E_i)\quad \forall i=2,...,M.
\end{gathered}
\end{equation}
Note that the choice $\tau \in (0, \delta/8)$ ensures that the closed ball $\smash{B^Y_{r + \delta/4}(y_d)}$
is contained in the interior of $\smash{B^Y_{r + \delta/2}(\bar y_d)}$ for every $y_d$ with $\|y_d - \tilde y_d\|_Y < \tau$,
and that the intersection of the interior of the ball $B^Y_{r + \delta/4}(y_d)$ with each $E_i$
is nonempty. 

The latter property implies, in combination with the fact that every vector in $E_i$ can be approximated by elements of
the image $\Psi(D, x_d)$, the compactness and disjointness of the sets $E_i$, and 
\eqref{eq:randomeq2736365}, that the estimates in \eqref{eq:randomeq36354-2} remain true when we intersect the 
sets $E_i$ with $B^Y_{r + \delta/4}(y_d) \cap \Psi(D, x_d)$,
i.e., it holds
\begin{equation}
\label{eq:randomeq18363292}
\dist(y_d, E_1 \cap B^Y_{r + \delta/4}(y_d) \cap \Psi(D, x_d)) = \dist(y_d, \closure_Y\left (\Psi(D, x_d)\right )) < r
\end{equation}
and
\begin{equation}
\label{eq:randomeq18363292-3535}
\dist(y_d, E_1 \cap B^Y_{r + \delta/4}(y_d) \cap \Psi(D, x_d)) 
< \dist(y_d, E_i \cap B^Y_{r + \delta/4}(y_d) \cap \Psi(D, x_d))
\end{equation}
for all $i=2,...,M$. Consider now an arbitrary but fixed $y_d$ with $\|y_d - \tilde y_d\|_Y < \tau$
and define
$D_i := \Psi(\cdot, x_d)^{-1}( E_i \cap B^Y_{r + \delta/4}(y_d))$.
Then, the continuity of the map $\Psi(\cdot, x_d)$ and 
the properties discussed above imply 
that the sets $D_i$, $i=1,...,M$, are relatively closed, disjoint, and nonempty subsets of $D$ which satisfy 
\begin{equation*}
\inf_{\alpha \in D_1} \|\Psi( \alpha, x_d) - y_d\|_Y^2
< \inf_{\alpha \in D_i} \|\Psi( \alpha, x_d) - y_d\|_Y^2\quad \forall i=2,...,M.
\end{equation*}
From the definition of the sets $D_i$, we further obtain that,
for every arbitrary but fixed $\tilde \alpha \in D \setminus (D_1 \cup ... \cup D_M)$,
we have $\Psi(\tilde \alpha, x_d) \not \in (E_1 \cup ... \cup E_M) \cap B^Y_{r + \delta/4}(y_d)$.
Since \eqref{eq:randomeq2736365} and the inclusion $\smash{B^Y_{r + \delta/4}(y_d) \subset B^Y_{r + \delta/2}(\bar y_d)}$
yield
\begin{equation*}
B^Y_{r + \delta/4}( y_d)  \cap \closure_Y(\Psi(D, x_d)) = 
\bigcup_{i=1}^M B^Y_{r + \delta/4}( y_d)  \cap  E_i,
\end{equation*}
this implies in particular that $\Psi(\tilde \alpha, x_d) \not \in B^Y_{r + \delta/4}(y_d)$
holds for all $\tilde \alpha \in D \setminus (D_1 \cup ... \cup D_M)$
and, as a consequence, that
\begin{equation*}
\sup_{\alpha \in D_1 \cup ... \cup D_M} \|\Psi( \alpha, x_d) -  y_d\|_Y^2
\leq \left (r + \frac{\delta}{4}\right)^2 < \|\Psi( \tilde \alpha, x_d) -  y_d\|_Y^2\quad 
\end{equation*}
for all $\tilde \alpha \in D \setminus (D_1 \cup ... \cup D_M)$.
The vector $y_d$ and the sets $D_i$
thus indeed satisfy \eqref{eq:spurpro1} and \eqref{eq:spurpro2}. 
As $y_d$ was an arbitrary vector with $\|y_d - \tilde y_d\|_Y < \tau$,
the existence of an open set $K$ with the properties in \cref{th:stevaluedspurious} 
now follows immediately. To see that the set $K$ can be chosen to be an open cone, 
it suffices to note that,
since all of the above arguments up to the estimates \eqref{eq:randomeq18363292} and 
\eqref{eq:randomeq18363292-3535} only rely on geometric properties of the 
set $ \closure_Y\left (\Psi(D, x_d)\right )$
and since the set 
$\closure_Y(\Psi(D, x_d))$ is a cone by \ref{fa:I}, 
by rescaling, 
we also obtain the claim for all $y_d \in Y$ 
which satisfy $\|y_d - s \tilde y_d\|_Y < s \tau$ for some $s > 0$. 
This completes the proof. 
\end{proof}

\begin{figure}[H]
~\\[-0.8cm]
\centering
\begin{tikzpicture}[scale=0.95]

\draw [line width=1pt, fill=gray!90] plot [smooth] coordinates {
(-0.34202014332*3.5 -1, -0.93969262078*3.5-2)
(-0.34202014332*3.5, -0.93969262078*3.5)
(-0.34202014332*3.5 + 0.5, -0.93969262078*3.5-2.2)
};

\draw [line width=1pt, fill=gray!90] plot [smooth] coordinates {
 (-0.64278760968 * 3.5 +0.1, 0.76604444311*3.5 + 2.2)
 (-0.64278760968 * 3.5, 0.76604444311*3.5)
 (-0.64278760968 * 3.5 - 1.5, 0.76604444311*3.5 +1.9)
};

\draw [line width=1pt, fill=gray!90] plot [smooth] coordinates {
(0.64278760968*3.5 + 2.0, 0.76604444311*3.5 + 0.6) 
(0.64278760968*3.5, 0.76604444311*3.5) 
(0.64278760968*3.5 + 2.1, 0.76604444311*3.5 + 0.1) 
};

\def\bigrectangle{(-5,-5.5) rectangle (5,4.9)}
\def\firstcircle{(0,0) circle (3.5 + 0.4)}
\scope
\clip \firstcircle \bigrectangle;

\draw [line width=1pt, fill=gray!45] plot [smooth] coordinates {
(-0.34202014332*3.5 -1, -0.93969262078*3.5-2)
(-0.34202014332*3.5, -0.93969262078*3.5)
(-0.34202014332*3.5 + 0.5, -0.93969262078*3.5-2.2)
};

\draw [line width=1pt, fill=gray!45] plot [smooth] coordinates {
 (-0.64278760968 * 3.5 +0.1, 0.76604444311*3.5 + 2.2)
 (-0.64278760968 * 3.5, 0.76604444311*3.5)
 (-0.64278760968 * 3.5 - 1.5, 0.76604444311*3.5 +1.9)
};

\draw [line width=1pt, fill=gray!45] plot [smooth] coordinates {
(0.64278760968*3.5 + 2.0, 0.76604444311*3.5 + 0.6) 
(0.64278760968*3.5, 0.76604444311*3.5) 
(0.64278760968*3.5 + 2.1, 0.76604444311*3.5 + 0.1) 
};

\endscope

\draw[line width=1pt] (-0.34202014332*3.5, -0.93969262078*3.5) circle (1cm); 
\draw[line width=1pt] (-0.64278760968 * 3.5, 0.76604444311*3.5)  circle (1cm); 
\draw[line width=1pt] (0.64278760968*3.5, 0.76604444311*3.5) circle (1cm); 

\draw[line width=1pt, dashed] (-0.34202014332*3.5, -0.93969262078*3.5) circle (1.8cm); 
\draw[line width=1pt, dashed] (-0.64278760968 * 3.5, 0.76604444311*3.5)  circle (1.8cm); 
\draw[line width=1pt, dashed] (0.64278760968*3.5, 0.76604444311*3.5) circle (1.8cm); 

\draw[line width=1pt, red] (0,0) circle (3.5cm);

\draw[line width=1pt, dashed, red] (0,0) circle (3.5cm + 0.4cm);

\draw[line width=1pt, dashed] (0, 0) -- (-0.34202014332*3.5, -0.93969262078*3.5);
\node at (-0.7, 0.1) {$\tilde y_d$};
\draw[fill=black] (-0.34202014332*0.125, -0.93969262078*0.125) circle (0.05cm);
\draw[arrows={-Triangle[angle=60:4pt, black,fill=black]},  line width=1pt]  
(-0.5, 0.05) -- (-0.34202014332*0.125 - 0.08, -0.93969262078*0.125 + 0.01);

\draw[fill=red, red] (0,0) circle (0.05cm);
\node at (0.1,0.7) {\textcolor{red}{$\bar y_d$}};
\draw[arrows={-Triangle[angle=60:4pt,red,fill=red]},  line width=1pt, red]  (0.08,0.5) -- (0.02,0.06);

\draw[fill=blue, blue] (-0.34202014332*0.125 + 0.1, -0.93969262078*0.125 - 0.05) circle (0.05cm);
\draw[line width=1pt, blue] (-0.34202014332*0.125 + 0.1, -0.93969262078*0.125 - 0.05)  circle (3.7cm);
\node at (0.3, -0.75) {\textcolor{blue}{$y_d$}};
\draw[arrows={-Triangle[angle=60:4pt,blue,fill=blue]},  line width=1pt, blue]  
(0.2,-0.55) -- (-0.34202014332*0.125 + 0.1 + 0.03, -0.93969262078*0.125 - 0.05 - 0.08);

\draw[fill=black] (-0.34202014332*3.5, -0.93969262078*3.5) circle (0.05cm); 
\node at (-0.34202014332*3.5 + 0.35, -0.93969262078*3.5 + 0.15) {$\bar y_1$};
\draw[fill=black] (-0.64278760968 * 3.5, 0.76604444311*3.5) circle (0.05cm); 
\node at (-0.64278760968 * 3.5 + 0.35, 0.76604444311*3.5 -0.1) {$\bar y_2$};
\draw[fill=black] (0.64278760968*3.5, 0.76604444311*3.5) circle (0.05cm); 
\node at (0.64278760968*3.5-0.25, 0.76604444311*3.5-0.1) {$\bar y_3$};

\draw[arrows={-Triangle[angle=60:4pt,red,fill=red]},  line width=1pt, red] (0.6, 0) -- (3.45, 0);
\draw[arrows={-Triangle[angle=60:4pt,red,fill=red]},  line width=1pt, red] (3, 0) -- (0.1, 0);
\node at (1.75, 0.2) {\textcolor{red}{$r$}};

\draw[arrows={-Triangle[angle=60:4pt,red,fill=red]},  line width=1pt, red] (3.7, 0) -- (3.85, 0);
\draw[arrows={-Triangle[angle=60:4pt,red,fill=red]},  line width=1pt, red] (3.7, 0) -- (3.55, 0);
\node at (4.35, 0.7) {\textcolor{red}{$\delta/2$}};
\draw[arrows={-Triangle[angle=60:4pt,red,fill=red]},  line width=1pt, red]  (4.0, 0.7) -- (3.7, 0.1);

\draw[arrows={-Triangle[angle=60:4pt,blue,fill=blue]},  line width=1pt, blue] 
(-0.34202014332*0.125 + 0.1 + 0.2*1, -0.93969262078*0.125 - 0.05 + 0.2*-0.3) 
-- 
(-0.34202014332*0.125 + 0.1 + 3.5*1, -0.93969262078*0.125 - 0.05 + 3.5*-0.3);
\draw[arrows={-Triangle[angle=60:4pt,blue,fill=blue]},  line width=1pt, blue] 
(-0.34202014332*0.125 + 0.1 + 0.5*1, -0.93969262078*0.125 - 0.05 + 0.5*-0.3) 
-- 
(-0.34202014332*0.125 + 0.1 + 0.08*1, -0.93969262078*0.125 - 0.05 + 0.08*-0.3);
\node at (2.65, -0.55) {\textcolor{blue}{$r + \delta/4$}};

\draw[arrows={-Triangle[angle=60:4pt, black,fill=black]},  line width=1pt]  
(-0.34202014332*3.5 -0.875 , -0.93969262078*3.5 - 0.2 -0.7) -- (-0.34202014332*3.5, -0.93969262078*3.5 - 0.15);
\node at (-0.34202014332*3.5 -1.1, -0.93969262078*3.5 - 0.2 -0.7) {$E_1$};

\draw[arrows={-Triangle[angle=60:4pt, black,fill=black]},  line width=1pt]  
(-0.64278760968 * 3.5 - 1.1, 0.76604444311*3.5 + 0.4) --(-0.64278760968 * 3.5 - 0.05, 0.76604444311*3.5 + 0.2);
\node at (-0.64278760968 * 3.5 - 1.35, 0.76604444311*3.5 + 0.4) {$E_2$};

\draw[arrows={-Triangle[angle=60:4pt, black,fill=black]},  line width=1pt]  
(0.64278760968*3.5 + 0.2, 0.76604444311*3.5 + 1.2)  -- (0.64278760968*3.5 + 0.3, 0.76604444311*3.5 + 0.03);
\node at(0.64278760968*3.5 + 0.2, 0.76604444311*3.5 + 1.4) {$E_3$};

\draw[arrows={-Triangle[angle=60:4pt, fill=black]},  line width=1pt] 
(-0.64278760968 * 3.5 + 0.3*-0.2, 0.76604444311*3.5 + 0.3*-0.9) -- (-0.64278760968 * 3.5 + 1.05*-0.2, 0.76604444311*3.5 + 1.05*-0.9);
\draw[arrows={-Triangle[angle=60:4pt, fill=black]},  line width=1pt] 
(-0.64278760968 * 3.5 + 0.7*-0.2, 0.76604444311*3.5 + 0.7*-0.9) -- (-0.64278760968 * 3.5 + 0.1*-0.2, 0.76604444311*3.5 + 0.1*-0.9);
\node at(-0.64278760968 * 3.5 + 0.6*-0.2 + 0.185, 0.76604444311*3.5 + 0.6*-0.9) {$\varepsilon$};

\draw[arrows={-Triangle[angle=60:4pt, fill=black]},  line width=1pt] 
(-0.64278760968 * 3.5 + 1.4*-0.2, 0.76604444311*3.5 + 1.4*-0.9) -- (-0.64278760968 * 3.5 + 1.9*-0.2, 0.76604444311*3.5 + 1.9*-0.9);
\draw[arrows={-Triangle[angle=60:4pt, fill=black]},  line width=1pt] 
(-0.64278760968 * 3.5 + 1.7*-0.2, 0.76604444311*3.5 + 1.7*-0.9) -- (-0.64278760968 * 3.5 + 1.15*-0.2, 0.76604444311*3.5 + 1.15*-0.9);
\node at(-0.64278760968 * 3.5 + 1.5*-0.2 + 0.185, 0.76604444311*3.5 + 1.5*-0.9) {$\delta$};
\end{tikzpicture}
\caption{Geometric situation in the proof of \cref{th:stevaluedspurious} in the case $M = 3$.
The set $\closure_Y(\Psi(D, x_d))$ is depicted in gray and the sets 
$E_1$, $E_2$, and $E_3$ in dark gray. 
The vector $\bar y_d$ and the circles $\smash{B^Y_{r}(\bar y_d)}$ and  $\smash{B^Y_{r + \delta/2}(\bar y_d)}$ 
centered at $\bar y_d$ are shown in red and the vector $y_d$ and the circle $\smash{B^Y_{r + \delta/4}(y_d)}$
centered at $y_d$ are shown in blue. The essential idea of the proof is that, 
if $\bar y_d \in Y$ satisfies $P_{\Psi}^{x_d}(\bar y_d) = \{\bar y_1, ..., \bar y_M\}$,
then by perturbing $\bar y_d$ slightly in the direction of $\bar y_1$, one 
obtains a vector $y_d$ for which the projection problem in the variable $y$ associated 
with the right-hand side of \eqref{eq:Ppsidef} possesses spurious local minima in each of the sets $E_i$, $i=2,...,M$.
These minima translate into spurious minima/basins of the optimization landscape 
of \eqref{eq:trainingpropprot2} by taking preimages 
under the function $\Psi(\cdot, x_d)\colon D \to Y$. }
\label{fig:1}
\end{figure}
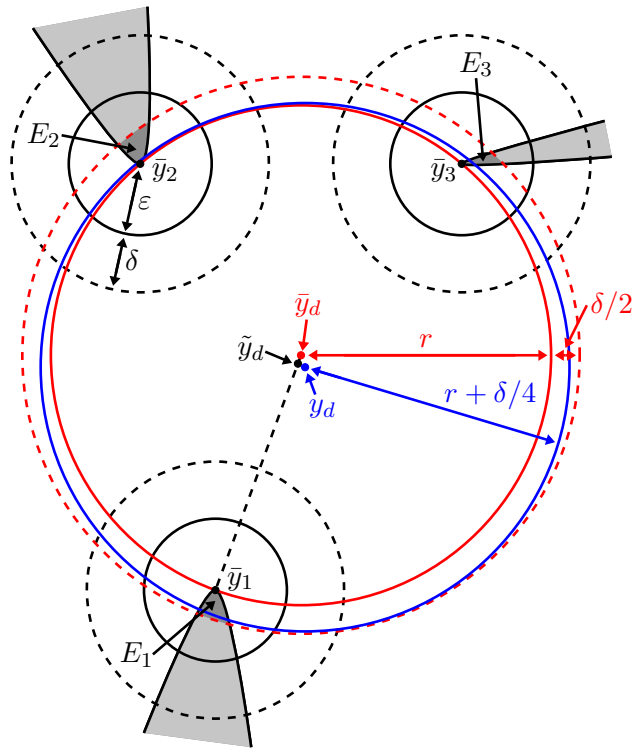~\\[-1.15cm]

Note that the result in \cref{th:stevaluedspurious} 
is, in fact, slightly stronger than that stated in point two of \cref{subsec:2.2} as 
it not only expresses that, in the presence of unrealizability, one 
cannot simultaneously get rid of both vectors $y_d$ with infinitely many best approximations 
and vectors $y_d$ for which \eqref{eq:trainingpropprot2} possesses
spurious minima/basins, but even that 
the only situation, in which spurious minima/basins 
can be completely absent in problem \eqref{eq:trainingpropprot2} for all $y_d \in Y$ in the presence 
of unrealizable vectors,
is that where the image of the function
$ y_d \mapsto | P_\Psi^{x_d}(y_d)|$ is equal to $\{1, \infty\}$.
(Recall that, if there exists a $y_d \in Y$ with $| P_\Psi^{x_d}(y_d)| = \infty$,
then there are automatically uncountably many such vectors
by the conicity in \ref{fa:I}.)
We remark that, as the cone $\closure_Y(\Psi(D, x_d))$
has to have very special geometric properties 
for the map $y_d \mapsto | P_\Psi^{x_d}(y_d)|$
to only take the values one and infinity, 
cases without spurious local minima and/or basins seem to be very rare in the above context. 
(An example of a cone $\closure_Y(\Psi(D, x_d)) \neq Y$ that 
satisfies $|P_\Psi^{x_d}(y_d)| \in \{1, \infty\}$
for all $y_d \in Y$ is the complement of a Lorentz cone in~$\R^3$.) 
This impression is also confirmed by the results 
on the existence of spurious valleys in
one-hidden-layer neural networks with 
non-polynomial non-negative activation functions proved by \cite{Venturi2019}.
Our analysis complements the findings of these authors, cf.\ the comments after \cref{cor:NNinfiniteNonUniqueness}.
Moreover, \cref{th:stevaluedspurious} is also in good accordance 
with the results on the absence of spurious valleys in overparameterized 
neural networks of \cite{Nguyen2018} and \cite{LiDawei2021}.
Indeed, as the condition $\closure_Y(\Psi(D, x_d)) \neq Y$ in \cref{ass:unrealizable_ass} 
can be expected to hold in the non-overparameterized regime, our result indicates
that some sort of overparameterization assumption is necessarily needed
to establish that spurious valleys cannot occur. 
For related work on the existence and role of bad basins in the loss 
landscape of training problems, see also \cite{Cooper2020}.

Regarding the comparison with the results of \cite{Venturi2019}, 
we would like to point out that 
the properties \eqref{eq:spurpro1} and \eqref{eq:spurpro2}
immediately imply that the problem \eqref{eq:trainingpropprot2}
possesses spurious valleys in the sense of \cite[Definition 1]{Venturi2019}.
To see this, define
\begin{equation}
\label{eq:cdef2635}
c := \sup_{\alpha \in D_1 \cup ... \cup D_M} \|\Psi( \alpha, x_d) - y_d\|_Y^2.
\end{equation}
From \eqref{eq:spurpro2}, we obtain that the sublevel set
$\Omega_L(c) := \{\alpha \in D \mid \| \Psi( \alpha, x_d) - y_d\|_Y^2 \leq c\}$ of the 
loss function in  \eqref{eq:trainingpropprot2} associated with the number $c$ in \eqref{eq:cdef2635}
is identical to the union of the nonempty, disjoint, and relatively closed sets 
$D_1,..., D_M$. 
Consider now a continuous path  $\gamma\colon [0, 1] \to \Omega_L(c) = D_1 \dot \cup ... \dot \cup D_M$.
Then, by taking preimages, we obtain that the interval $[0, 1]$
can be written as the disjoint union of the closed sets $\gamma^{-1}(D_i)$, $i=1,...,M$.
Since the interval $[0, 1]$ is connected, this is only possible if all of the sets $\gamma^{-1}(D_i)$ 
but one are empty. In particular, there cannot be a continuous path 
$\gamma\colon [0, 1] \to \Omega_L(c)$ satisfying $\gamma(0) \in D_i$ and $\gamma(1) \in D_j$ for some $i \neq j$,
and we may conclude that points $\alpha \in D_i$ and $\tilde \alpha  \in D_j$ with $i \neq j$
cannot be in the same path-connected component of the sublevel set $\Omega_L(c)$.
Since the sets $D_1,..., D_M$ are nonempty, this implies that each $D_i$ contains at least 
one path-connected component of the sublevel set $\Omega_L(c)$. 
However, 
from \eqref{eq:spurpro1}, we also obtain that only path-connected components 
of $\Omega_L(c)$ contained in the set $D_1$ can contain a global minimizer of \eqref{eq:trainingpropprot2}.
This shows that there exist path-connected components of the sublevel set $\Omega_L(c)$
(namely all those contained in the sets $D_2,...,D_M$),
on which the global optimum of \eqref{eq:trainingpropprot2} is not attained, 
and that
 the two conditions \eqref{eq:spurpro1} and \eqref{eq:spurpro2} indeed imply that 
there exist spurious valleys in the sense of \cite[Definition 1]{Venturi2019}.

Checking whether \cref{th:stevaluedspurious} is applicable in a 
certain situation or not is, of course, typically far from trivial.
Because of this and since \cref{th:stevaluedspurious} does not yield 
any information about how far away spurious local minima can be from
global solutions of \eqref{eq:trainingpropprot2} (should they exist),
in what follows, we will prove criteria for the existence 
of non-optimal stationary points that do not rely on the 
geometric properties of the set $\closure_Y(\Psi(D, x_d))$ but rather exploit the condition 
\ref{fa:II} directly.
As we will see below, this approach 
has the additional advantage that it does not require the condition 
$\closure_Y(\Psi(D, x_d)) \neq Y$ and is thus also applicable in the overparameterized regime.
The starting point of our analysis is 
the following lemma whose proof follows the lines of that of \cref{prop:solar}:
\begin{lemma}
\label[lemma]{lemma:alldirectionsleadtosuboptimality}
Suppose that \cref{ass:standingassumpssec4} holds
and that $\bar \alpha \in D$ is arbitrary but fixed. 
Assume further that a vector $v \in Y$ satisfying $\|v\|_Y = 1$ and $(v, \Psi(\bar \alpha, x_d))_Y = 0$ is given,
and define $y_d^s := \Psi(\bar \alpha, x_d) + s v$ for all $s \in \R$. 
Then, it holds
\begin{equation}
\label{eq:randomeq1273535dd3}
\bar \alpha \not \in \argmin_{\alpha \in D} \|\Psi( \alpha, x_d) - y_d^s\|_Y^2
\quad \forall s \in \R \text{ with } |s| > \left (\frac{\Theta(\Psi, x_d)}{1 - \Theta(\Psi, x_d)}\right)^{1/2} \|\Psi(\bar \alpha, x_d)\|_Y,
\end{equation}
i.e., $\bar \alpha$ is not a global minimum of \eqref{eq:trainingpropprot2} for all $s \in \R$
that satisfy the condition in \eqref{eq:randomeq1273535dd3}.
\end{lemma}
\begin{proof}
Let $\bar y_s$ denote an arbitrary but fixed element of $P_\Psi^{x_d}(y_d^s)$ for all $s \in \R$. 
Then, we may use \eqref{eq:randomeq2736352} and the properties of $v$ to compute
completely analogously to \eqref{eq:randomeq26353672}  that
\begin{equation*}
\begin{aligned}
\|\bar y_s - y_d^s\|_Y^2  - \|  \Psi(\bar \alpha, x_d) - y_d^s \|_Y^2 
\leq
\Theta(\Psi, x_d) \|\Psi(\bar \alpha, x_d) \|_Y^2 + (\Theta(\Psi, x_d) - 1)s^2
\end{aligned}
\end{equation*}
holds for all $s \in \R$. The above implies 
\begin{equation*}
\inf_{\alpha \in D} \|\Psi( \alpha, x_d) - y_d^s\|_Y^2 < \|  \Psi(\bar \alpha, x_d)  - y_d^s\|_Y^2
\end{equation*}
for all $s \in \R$ that satisfy the condition in \eqref{eq:randomeq1273535dd3}. This proves the claim. 
\end{proof}

By exploiting the observation in \cref{lemma:alldirectionsleadtosuboptimality}, 
we readily obtain: 
\begin{theorem}{\hspace{-0.05cm}\bf(Criterion for the Existence of Non-Optimal Stationary Points)}\label{theorem:existencestatpts}
Suppose that \cref{ass:standingassumpssec4} holds,
that $\bar \alpha \in D$ is an arbitrary but fixed 
element of the interior of the set $D$, and that the function $\Psi(\cdot, x_d)\colon D \to Y$ 
is differentiable at $\bar \alpha$. Assume further that the 
linear hull 
$V := \span(\Psi(\bar \alpha, x_d), \partial_1 \Psi(\bar \alpha, x_d),...,\partial_m \Psi(\bar \alpha, x_d)) \subset Y$
of the vector $\Psi(\bar \alpha, x_d)$ and the partial derivatives $\partial_i \Psi(\bar \alpha, x_d)$, $i=1,...,m$, of the map
$\Psi(\cdot, x_d)\colon D \to Y$  at $\bar \alpha$
is not equal to $Y$.
Then, for every arbitrary but fixed element $v$
of the $(\cdot, \cdot)_Y$-orthogonal complement of $V$ with $\|v\|_Y = 1$
and every $s \in \R$ satisfying 
\begin{equation}
\label{eq:scondition}
|s| > \left (\frac{\Theta(\Psi, x_d)}{1 - \Theta(\Psi, x_d)}\right)^{1/2} \|\Psi(\bar \alpha, x_d)\|_Y,
\end{equation}
there exists a $\tau \in \{-1,1\}$ such that $\bar \alpha$ is a 
spurious local minimum or a saddle point of the 
training problem \eqref{eq:trainingpropprot2} with label vector $y_d^{\tau s} := \Psi(\bar \alpha, x_d) + \tau s v$. 
Moreover, for every arbitrary but fixed $C> 0$, there 
exist uncountably many label vectors $y_d \in Y$ such that 
$\bar \alpha$ is a saddle point or a spurious local minimum 
of  \eqref{eq:trainingpropprot2},
 such that
\begin{equation}
\label{eq:errorestimates}
\inf_{\bar y \in P_\Psi^{x_d}(y_d)}
\|\Psi(\bar \alpha, x_d) - \bar y\|_Y \geq C
\qquad\text{and}\qquad 
\inf_{\bar y \in P_\Psi^{x_d}(y_d)}
\frac{\|\Psi(\bar \alpha, x_d) - \bar y\|_Y}{\|\bar y\|_Y} \geq 1 - \frac{1}{C},
\end{equation}
and such that 
\begin{equation}
\label{eq:lossestimate42}
\inf_{\alpha \in D} \|\Psi( \alpha, x_d) - y_d\|_Y^2 + C \leq \|\Psi( \bar \alpha, x_d) - y_d\|_Y^2
\end{equation}
holds. 
The absolute error between $\Psi(\bar \alpha, x_d)$ 
and every true best approximation of  $y_d$ can thus be made arbitrarily large,
the relative error between $\Psi(\bar \alpha, x_d)$ 
and every true best approximation of $y_d$ can be made larger than $1 - \varepsilon$ for all $\varepsilon > 0$, 
and the difference between the value of the loss function at $\bar \alpha$
and the optimal loss can be made arbitrarily large. 
\end{theorem}

\begin{proof}
Suppose that an $\bar \alpha \in D$ with the properties in the theorem is given
and that $v \in Y$ is an arbitrary but fixed element of the orthogonal 
complement of $V$ satisfying $\|v\|_Y = 1$.
Define  $y_d^s := \Psi(\bar \alpha, x_d) + s v$ for all $s \in \R$.
Then, the differentiability of the map  $\Psi(\cdot, x_d)\colon D \to Y$
at $\bar \alpha$, the properties of $v$, and the definition of $y_d^s$ imply that 
\begin{equation}
\label{eq:randomeq236363}
\begin{aligned}
& \|\Psi(\bar \alpha + h, x_d) - y_d^s\|_Y^2 - \|\Psi( \bar \alpha, x_d) - y_d^s \|_Y^2
\\
&\quad 
= 
2\left ( 
\Psi(\bar \alpha + h, x_d) - \Psi( \bar \alpha, x_d), \frac{\Psi(\bar \alpha + h, x_d)  - \Psi(\bar \alpha  , x_d)}{2}- s v
\right )_Y
\\
&\quad 
= 
2\left ( 
\sum_{i=1}^m h_i \partial_i \Psi(\bar \alpha, x_d) , \frac{\Psi(\bar \alpha + h, x_d)  - \Psi(\bar \alpha  , x_d)}{2}- s v
\right )_Y + \oo(\|h\|_2)
=\oo(\|h\|_2)
\end{aligned}
\end{equation}
holds for all $s \in \R$ and all sufficiently small $h \in \R^m$, 
where the Landau symbol refers to the limit $\|h\|_2 \to 0$.
Dividing by $\|h\|_2$ and passing to the limit in \eqref{eq:randomeq236363} 
yields that $\bar \alpha$ is a stationary point of \eqref{eq:trainingpropprot2}
for every $y_d^s$, $s \in \R$, i.e., 
the gradient of 
the loss function 
of \eqref{eq:trainingpropprot2} vanishes at $\bar \alpha$. 
In combination with \cref{lemma:alldirectionsleadtosuboptimality},
it now follows immediately that $\bar \alpha$
has to be a spurious local minimum, a local maximum,
or a saddle point of  the training problem \eqref{eq:trainingpropprot2} with label vector $y_d^s$
for all $s \in \R$ satisfying \eqref{eq:scondition}. 
Next, we show that,
for every $s \in \R$ with \eqref{eq:scondition},
the point $\bar \alpha$ is a saddle point or a spurious local minimum of \eqref{eq:trainingpropprot2}
for one of the vectors $y_d^s$ and $y_d^{-s}$. To see this, 
let us assume that there exists an $s \in \R$ with \eqref{eq:scondition} such that 
the latter is not the case. Then, $\bar \alpha$ has to be a local maximum 
of  \eqref{eq:trainingpropprot2} for both $y_d^s$ and $y_d^{-s}$
and we obtain from the same 
calculation as in \eqref{eq:randomeq236363} that
\begin{equation*}
\begin{aligned}
&\|\Psi(\bar \alpha + h, x_d) - y_d^s\|_Y^2 - \|\Psi( \bar \alpha, x_d) - y_d^s \|_Y^2
\\
&\qquad= 
\|\Psi(\bar \alpha + h, x_d)  - \Psi( \bar \alpha, x_d) \|_Y^2 +
2\left ( 
\Psi(\bar \alpha + h, x_d) - \Psi( \bar \alpha, x_d), - s v
\right )_Y \leq 0
\end{aligned}
\end{equation*}
and
\begin{equation*}
\begin{aligned}
&\|\Psi(\bar \alpha + h, x_d) - y_d^{-s}\|_Y^2 - \|\Psi( \bar \alpha, x_d) - y_d^{-s} \|_Y^2
\\
&\qquad=
\|\Psi(\bar \alpha + h, x_d)  - \Psi( \bar \alpha, x_d) \|_Y^2 +
2\left ( 
\Psi(\bar \alpha + h, x_d) - \Psi( \bar \alpha, x_d), s v
\right )_Y \leq 0
\end{aligned}
\end{equation*}
for all $h \in \R^m$ in a sufficiently small open ball around zero. Adding the above yields 
\begin{equation*}
2 \|\Psi(\bar \alpha + h, x_d)  - \Psi( \bar \alpha, x_d) \|_Y^2  \leq 0
\end{equation*}
for all small $h$ which can only be true if the function $\alpha \mapsto \Psi(\alpha, x_d)$
is constant in a small open neighborhood of $\bar \alpha$. But if this is the case, 
then $\bar \alpha$ is trivially also a local minimum of \eqref{eq:trainingpropprot2}
(for both $y_d^s$ and $y_d^{-s}$). The point $\bar \alpha$ is thus indeed always 
a spurious local minimum or a saddle point of \eqref{eq:trainingpropprot2} for at least one of the label vectors 
$y_d^s$ and $y_d^{-s}$ for all $s \in \R$ with \eqref{eq:scondition}.
This proves the first claim of the theorem. To see
that we can also achieve \eqref{eq:errorestimates} and \eqref{eq:lossestimate42}
for all arbitrary but fixed $C>0$,
we note that \eqref{eq:randomeq2736352} implies that 
\begin{equation*}
\|y_d^s\|_Y - \|\bar y_s\|_Y  \leq \|\bar y_s - y_d^s\|_Y \leq \Theta(\Psi, x_d)^{1/2}\|y_d^s\|_Y
\end{equation*}
holds for all $\bar y_s \in P_\Psi^{x_d}(y_d^s)$ and that, as a consequence,
\begin{equation*}
\inf_{\bar y \in P_\Psi^{x_d}(y_d^s)}\|\bar y\|_Y 
\geq \left (1 - \Theta(\Psi, x_d)^{1/2} \right ) \left (\| \Psi( \bar \alpha, x_d)\|_Y^2 + s^2 \right )^{1/2}\qquad \forall s \in \R.
\end{equation*}
We thus have 
\begin{equation*}
\begin{aligned}
\inf_{\bar y \in P_\Psi^{x_d}(y_d^s)} \|\Psi(\bar \alpha, x_d) - \bar y\|_Y 
\geq 
\left (1 - \Theta(\Psi, x_d)^{1/2} \right ) \left (\| \Psi( \bar \alpha, x_d)\|_Y^2 + s^2 \right )^{1/2} - \|\Psi( \bar \alpha, x_d)\|_Y
\to \infty
\end{aligned}
\end{equation*}
as well as 
\begin{equation*}
\inf_{\bar y \in P_\Psi^{x_d}(y_d^s)} \frac{\|\Psi(\bar \alpha, x_d) - \bar y\|_Y}{\|\bar y\|_Y} 
\geq 
1 - 
\frac{\| \Psi( \bar \alpha, x_d)\|_Y}{\left (1 - \Theta(\Psi, x_d)^{1/2} \right ) \left (\| \Psi( \bar \alpha, x_d)\|_Y^2 + s^2 \right )^{1/2}}
\to 1
\end{equation*}
and, again by \eqref{eq:randomeq2736352},
\begin{equation*}
\begin{aligned}
&\inf_{\alpha \in D} \|\Psi( \alpha, x_d) - y_d^s\|_Y^2  - \|\Psi( \bar \alpha, x_d) - y_d^s\|_Y^2
\\
&\qquad\leq  
 \Theta(\Psi, x_d) \|y_d^s\|_Y^2    - \|s v\|_Y^2
=  \Theta(\Psi, x_d) \|\Psi(\bar \alpha, x_d) \|_Y^2 + (\Theta(\Psi, x_d) - 1)s^2  \to - \infty
\end{aligned}
\end{equation*}
for $|s| \to \infty$. This shows that  \eqref{eq:errorestimates} and \eqref{eq:lossestimate42}
hold for every arbitrary but fixed constant $C>0$ provided $|s|$ is large enough. 
In combination with what we already know about the vectors $y_d^s$, 
this establishes the second assertion of the theorem and completes the proof.\end{proof}
~\\[-1cm]
\begin{remark}\label[remark]{rem:scaling}
As already pointed out in the introduction, 
the fact that the saddle points and spurious local minima
in \cref{theorem:existencestatpts} can be made arbitrarily bad is not a
mere consequence of the conicity condition \ref{fa:I}. 
Indeed, if we naively scale the vectors appearing in \cref{theorem:existencestatpts} by a factor $\gamma > 0$,
then this factor clearly cancels out in the second estimate of \eqref{eq:errorestimates}
and it is not possible to ensure that the relative error in \eqref{eq:errorestimates} becomes 
larger than $1 - \varepsilon$ for every $\varepsilon > 0$
by passing to the limit $\gamma \to \infty$.
\end{remark}

Note that the assumptions on the linear hull
$\span(\Psi(\bar \alpha, x_d), \partial_1 \Psi(\bar \alpha, x_d),...,\partial_m \Psi(\bar \alpha, x_d))$ in \cref{theorem:existencestatpts}
are trivially satisfied if $m + 1 < \dim(Y) = n \dim(\YY)$ holds, i.e., if
the product of the number of training pairs in \eqref{eq:trainingpropprot} 
and the dimension of the output space $\YY$ 
exceeds the number of parameters in the considered approximation scheme
by more than one. In this non-overparameterized case, \cref{theorem:existencestatpts} yields that 
\emph{every} point of differentiability of the function $\Psi(\cdot, x_d)\colon D \to Y$
is an arbitrarily bad saddle point or spurious local minimum of  \eqref{eq:trainingpropprot2} 
for uncountably many choices of the label vector $y_d$ in the situation of \cref{ass:standingassumpssec4}. 
Compare also with \cref{cor:saddlenonover} in \cref{subsec:6.2} in this context. However, 
as already mentioned, such a non-overparameterization is not necessary to be able to apply the last theorem. 
To do so, it suffices to show that a local affine-linear approximation of the function $\Psi(\cdot, x_d)\colon D \to Y$ is not surjective
(in contrast to, e.g., \cref{th:stevaluedspurious} which requires that the image of the map $\Psi(\cdot, x_d)\colon D \to Y$ 
itself is not dense in $Y$).  
In \cref{lemma:polystruct}, we will prove that, for approximation schemes
on an Euclidean space $\XX = \R^{d_\xx}$,
that, after reordering the entries of the vector $\alpha \in \R^m$
as a tuple $(\beta, A)  \in \R^p \times \R^{q \times d_\xx}$
with some $p, q \in \mathbb{N}$ satisfying $ m = p + q d_\xx$,
can be written in the form $\psi(\alpha, \xx) = \phi(\beta, A\xx)$ with a 
differentiable function $\phi$
(and thus in particular for neural networks with differentiable activations), 
points with the property ${\span(\Psi(\bar \alpha, x_d), \partial_1 \Psi(\bar \alpha, x_d),...,\partial_m \Psi(\bar \alpha, x_d)) \neq Y}$
always exist if the number $d_\xx + 1$ is smaller than $n$. 
This shows that \cref{theorem:existencestatpts} can be used to establish the existence of spurious local minima or 
saddle points that are arbitrarily far away
from global optima in many situations arising in practice. 

If we not only know that the linear hull
$\span(\Psi(\bar \alpha, x_d), \partial_1 \Psi(\bar \alpha, x_d),...,\partial_m \Psi(\bar \alpha, x_d))$
is not equal to $Y$, but even that $\bar \alpha$ possesses an open neighborhood $U \subset D$ such that
the image of $U$ under $\Psi(\cdot, x_d)\colon D \to Y$ is contained in a 
proper subspace $V$ of $Y$, then we can prove that the construction in \cref{theorem:existencestatpts}
always produces spurious local minima:
\begin{proposition}{\hspace{-0.05cm}\bf(Criterion for the Existence of Spurious Local Minima)}\label[proposition]{prop:existencehotspurs}
Consider the situation in \cref{ass:standingassumpssec4} 
and suppose that a point $\bar \alpha \in D$ is given
such that there exist an open set $U \subset D$ with $\bar \alpha \in U$
and a subspace $V$ of $Y$ satisfying  $\Psi(U, x_d) \subset V \neq Y$. 
Then, for every arbitrary but fixed element $v$
of the $(\cdot, \cdot)_Y$-orthogonal complement of $V$ with $\|v\|_Y = 1$
and every $s \in \R$ satisfying \eqref{eq:scondition},
the point $\bar \alpha$ is a spurious local minimum of the 
training problem \eqref{eq:trainingpropprot2} with label vector $y_d^{s} := \Psi(\bar \alpha, x_d) + s v$
that satisfies a quadratic growth condition of the form 
\begin{equation}
\label{eq:quadgrowthstates}
 \|\Psi(\alpha, x_d) - y_d^s\|_Y^2 - \|\Psi( \bar \alpha, x_d) - y_d^s \|_Y^2
\geq 
 \|\Psi(\alpha, x_d) - \Psi( \bar \alpha, x_d)\|_Y^2\qquad \forall \alpha \in U.
\end{equation}
Further, for every arbitrary but fixed $C> 0$, there 
exist uncountably many $y_d \in Y$ such that 
$\bar \alpha$ is a spurious local minimum 
of  \eqref{eq:trainingpropprot2} and such that 
\eqref{eq:errorestimates}, \eqref{eq:lossestimate42}, and \eqref{eq:quadgrowthstates} hold. 
\end{proposition}\pagebreak

\begin{proof}
Suppose that an $\bar \alpha \in D$ satisfying the assumptions of the proposition
is given and that $v \in Y$ is an arbitrary but fixed vector with $\|v\|_Y = 1$ that 
is $(\cdot, \cdot)_Y$-orthogonal to $V$.
Then, it follows from the inclusion $\Psi(\bar \alpha, x_d) \in \Psi(U, x_d) \subset V$ 
that $v$ is orthogonal to $\Psi(\bar \alpha, x_d)$ and we may invoke \cref{lemma:alldirectionsleadtosuboptimality} 
to deduce that $\bar \alpha$ is not a global minimum of \eqref{eq:trainingpropprot2}  
when the training label vector is chosen as $y_d^s = \Psi(\bar \alpha, x_d) + s v$ 
with an $s \in \R$ satisfying \eqref{eq:scondition}. From
exactly the same calculation as in \eqref{eq:randomeq236363}, we further obtain that 
\begin{equation*}
\begin{aligned}
\|\Psi(\alpha, x_d) - y_d^s\|_Y^2 - \|\Psi( \bar \alpha, x_d) - y_d^s \|_Y^2
&
= 
\|\Psi(\alpha, x_d) - \Psi( \bar \alpha, x_d) - s v \|_Y^2 - \|sv\|_Y^2
\\
&= 
\|\Psi(\alpha, x_d) - \Psi( \bar \alpha, x_d) \|_Y^2 
\end{aligned}
\end{equation*}
holds for all $\alpha \in U$. This shows that $\bar \alpha$ satisfies the growth condition \eqref{eq:quadgrowthstates} for all $y_d^s$
and, in combination with our first observation, that
$\bar \alpha$ is indeed a spurious local minimum of \eqref{eq:trainingpropprot2} 
for all $y_d^s$ with an $s \in \R$ satisfying \eqref{eq:scondition}. 
To complete the proof, it remains to show that, for every $C>0$, there 
exist uncountably many vectors $y_d$ such that $\bar \alpha$ is a spurious local minimum 
of \eqref{eq:trainingpropprot2} and such that the estimates
\eqref{eq:errorestimates}, \eqref{eq:lossestimate42}, and \eqref{eq:quadgrowthstates} hold. 
This, however, follows completely analogously to the proof of \cref{theorem:existencestatpts}. 
\end{proof}

We would like to emphasize that the last result does not require any form of differentiability. 
Under slightly stronger assumptions on the function $\Psi(\cdot, x_d)\colon D \to Y$,
we can also analyze the size of the set of label vectors $y_d$ that give 
rise to training problems \eqref{eq:trainingpropprot2} with spurious local minima
in the situation of \cref{prop:existencehotspurs}: 

\begin{theorem}{\hspace{-0.05cm}\bf (An Open Cone of Label Vectors with Spurious Local Minima)}\label{theorem:badcone}
Consider the situation in \cref{ass:standingassumpssec4} and suppose that 
there exists a subspace $V$ of $Y$ with $V \neq Y$ such that, for 
every $z \in V$, there exist an $\bar \alpha \in D$ and an open set $U \subset D$
satisfying $\bar \alpha \in U$, $z = \Psi(\bar \alpha, x_d)$, and $\Psi(U, x_d) \subset V$.
Denote the orthogonal complement of $V$ in $Y$ with $V^\perp$
(so that $Y = V \oplus V^\perp$). 
Then, the training problem \eqref{eq:trainingpropprot2} possesses at least one spurious local minimum
satisfying a growth condition of the form \eqref{eq:quadgrowthstates}
for all label vectors $y_d \in Y$ that are elements of the open cone
\begin{equation}
\label{eq:Kspurdef}
K :=
\left \{
y_d^1 + y_d^2 \in Y
\,
\Bigg |
\,
y_d^1 \in V,\, y_d^2 \in V^\perp,\, \|y_d^2\|_Y > \left (\frac{\Theta(\Psi, x_d)}{1 - \Theta(\Psi, x_d)}\right)^{1/2} \|y_d^1\|_Y
\right \}.
\end{equation}
Further, for every arbitrary but fixed $C>0$, there exist uncountably many $y_d \in K$
such that at least one of the spurious local minima of \eqref{eq:trainingpropprot2} satisfies 
\eqref{eq:errorestimates}, \eqref{eq:lossestimate42}, and \eqref{eq:quadgrowthstates},
and if $\closure_Y(\Psi(D, x_d)) = Y$ holds, then the cone $K$ in \eqref{eq:Kspurdef} 
is equal to $Y \setminus V$ and \eqref{eq:trainingpropprot2} 
possesses spurious local minima for all $y_d$ that are not elements of $V$.
\end{theorem}

\begin{proof}
Consider an arbitrary but fixed element $y_d = y_d^1 + y_d^2 \in V \oplus V^\perp$ of the cone $K$.
Then, our assumptions on the map 
$\Psi(\cdot, x_d)\colon D \to Y$ imply that we can find an $\bar \alpha \in D$ and 
an open set $U \subset D$ satisfying $\bar \alpha \in U$, $y_d^1 = \Psi(\bar \alpha, x_d)$, and $\Psi(U, x_d) \subset V \neq Y$.
Define $v := y_d^2 / \|y_d^2\|_Y$ and $s := \|y_d^2\|_Y$. (Note that $y_d^2$ cannot be zero by the definition of $K$.) 
Then, it clearly holds $v \in V^\perp$, $\|v\|_Y = 1$, $y_d = \Psi(\bar \alpha, x_d)+ sv$, and 
\begin{equation*}
\left (\frac{\Theta(\Psi, x_d)}{1 - \Theta(\Psi, x_d)}\right)^{1/2} \|\Psi(\bar \alpha, x_d)\|_Y
=
\left (\frac{\Theta(\Psi, x_d)}{1 - \Theta(\Psi, x_d)}\right)^{1/2} \|y_d^1\|_Y
<
\|y_d^2\|_Y = s.
\end{equation*}
By applying \cref{prop:existencehotspurs}, it now follows immediately that $\bar \alpha$ is a spurious 
local minimum of the problem \eqref{eq:trainingpropprot2} with label vector $y_d$ 
that satisfies a growth condition of the form \eqref{eq:quadgrowthstates}. 
This proves the first part of the theorem. 
To establish that there exist uncountably many vectors $y_d \in K$
with arbitrarily bad spurious local minima and 
that the identity $K = Y \setminus V$ holds in the case $\closure_Y(\Psi(D, x_d)) = Y$,
it suffices to invoke \cref{prop:existencehotspurs} and
the definition of the number $\Theta(\Psi, x_d)$ in \eqref{eq:defTheta}.
This completes the proof.  
\end{proof}\\[-1cm]

Note that the inequalities in \eqref{eq:scondition} and \eqref{eq:Kspurdef}
again link the approximation properties of the map $\Psi(\cdot, x_d)\colon D \to Y$
to  properties of the loss landscape of the training problem \eqref{eq:trainingpropprot2}
(cf.\ \cref{def:thetadef}). 
If $\Theta(\Psi, x_d) \to 1$ holds, i.e., if the behavior of ${\Psi(\cdot, x_d)\colon D \to Y}$ approximates that 
of a linear approximation scheme, then the distance between the point $\Psi(\bar \alpha, x_d)$
and the label vectors $y_d$
that cause the parameter $\bar \alpha$ to be
a saddle point or a spurious local minimum of \eqref{eq:trainingpropprot2} in the situation of 
\cref{theorem:existencestatpts} tends to infinity
and the cone $K$ in \cref{theorem:badcone} degenerates.
If, on the other hand,  $\Theta(\Psi, x_d)$ tends to zero, i.e., if the expressiveness of 
the considered nonlinear approximation scheme relative to $Y$ increases and we approach 
the case $\closure_Y(\Psi(D, x_d)) = Y$, then the vectors $y_d$ in \cref{theorem:existencestatpts} can be chosen
arbitrarily close to the point $\Psi(\bar \alpha, x_d)$ and 
the cone $K$ in \cref{theorem:badcone} exhausts the set $Y \setminus V$. 
We remark that, in combination with the results 
on the stability properties of the best approximation map $P_{\Psi}^{x_d}$
in \cref{theorem:abstractinstability},
the above observations give a quite good impression of the issues that one has to deal with 
when considering training problems with squared loss for nonlinear 
approximation schemes satisfying the conditions in \cref{ass:standingassumpssec4}
and of how these issues are related to the error bound $\Theta(\Psi, x_d)$ in \cref{def:thetadef}.
We will get back to this topic in \cref{subsec:6.2}, where we 
demonstrate that \cref{theorem:abstractinstability,theorem:badcone}
in particular apply to neural networks 
that involve activation functions with an affine segment. 

We conclude this subsection with a result that demonstrates that 
the subspace property in \cref{prop:existencehotspurs} 
is not only relevant for the existence of spurious local minima but 
also for the stability and uniqueness of global solutions 
of  \eqref{eq:trainingpropprot2} in the case $\closure_Y(\Psi(D, x_d)) = Y$, i.e., 
in the situation where every vector $y_d \in Y$ is realizable. 
\begin{corollary}{\hspace{-0.05cm}\bf(Instability and Nonuniqueness in the Realizable Case)}\label[corollary]{cor:instabilityoverpara}
Consider the situation in \cref{ass:standingassumpssec4} and
suppose that a point $\bar \alpha \in D$ is given
such that there exist an open set $U \subset D$ with $\bar \alpha \in U$
and a subspace $V$ of $Y$ satisfying  $\Psi(U, x_d) \subset V \neq Y$.
Assume further that $\closure_Y(\Psi(D, x_d)) = Y$ holds
and that $v$ is an arbitrary but fixed element 
of the orthogonal complement of $V$ in $Y$ with $\|v\|_Y = 1$. Define 
$\bar y_d :=  \Psi(\bar \alpha, x_d)$ and $y_d^{s} := \bar y_d + s v$ for all $s \in \R \setminus \{0\}$. 
Then, the solution map
\begin{equation*}
 Y \ni
y_d \mapsto \argmin_{\alpha \in D} \|\Psi( \alpha, x_d) - y_d\|_Y^2 \subset D
\end{equation*}
of the problem \eqref{eq:trainingpropprot2} is discontinuous at $\bar y_d$ in the sense that the following is true:
\begin{equation}
\label{eq:randomeq182636535}
\bar \alpha \in  \argmin_{\alpha \in D} \|\Psi( \alpha, x_d) - \bar y_d\|_Y^2,
~~~
y_d^s \xrightarrow{s \to 0} \bar y_d,
~~~
U \, \cap\, \argmin_{\alpha \in D} \|\Psi( \alpha, x_d) - y_d^s\|_Y^2  = \emptyset~\,\forall s \in \R \setminus \{0\}.
\end{equation}
Further, in this situation, there exists a family of parameters $\{\alpha_s\}_{s \in \R \setminus \{0\}}$ satisfying
\begin{equation*}
\{\alpha_s\}_{s \in \R \setminus \{0\}} \subset D \setminus U
\qquad
\text{and}
\qquad
\lim_{s \to 0}  \|\Psi(\alpha_s, x_d) - \bar y_d\|_Y^2  = 0 = \|\Psi(\bar \alpha, x_d) - \bar y_d\|_Y^2.
\end{equation*}
The problem \eqref{eq:trainingpropprot2} with label vector $\bar y_d$ is thus 
not uniquely solvable in the generalized sense that it 
possesses at least one minimizing sequence that does not converge to $\bar \alpha$. 
\end{corollary}\pagebreak

\begin{proof}
In the situation of the corollary, 
it follows from \cref{def:thetadef} and the assumption $\closure_Y(\Psi(D, x_d)) = Y$ that $\Theta(\Psi, x_d) = 0$ holds, 
and we obtain from \cref{prop:existencehotspurs} 
that $\bar \alpha$ is a local minimum of the problem \eqref{eq:trainingpropprot2} with 
label vector $y_d^s$ for all $s \in \R \setminus \{0\}$ that satisfies
\begin{equation}
\label{eq:randomeq173635}
\inf_{\alpha \in D} \|\Psi( \alpha, x_d) - y_d^s\|_Y^2 = 0
< s^2 =  \|\Psi(\bar \alpha, x_d) - y_d^s\|_Y^2 \leq  \|\Psi(\tilde \alpha, x_d) - y_d^s\|_Y^2
\end{equation}
for all $\tilde \alpha \in U$ and all $s \in \R\setminus \{0\}$.
The above implies in particular that none of the points in $U$ can be a global minimizer of 
\eqref{eq:trainingpropprot2} with label vector $y_d^s$ for all $s \in \R\setminus \{0\}$.
 This establishes \eqref{eq:randomeq182636535}.
On the other hand, \eqref{eq:randomeq173635} also yields that, for every $s \in \R \setminus \{0\}$, 
we can find an $\alpha_s \in D \setminus U$ with 
\begin{equation*}
|s|
> \|\Psi( \alpha_s, x_d) - y_d^s\|_Y
\geq \|\Psi( \alpha_s, x_d) - \bar y_d\|_Y  - |s|.
\end{equation*}
The existence of a family $\{\alpha_s\}_{s \in \R \setminus \{0\}}$ with the properties in the
second part of the corollary now follows immediately. 
This completes the proof. 
\end{proof}

\subsection{Spurious Local Minima in the Presence of a Regularization Term}
\label{subsec:5.3}
A standard technique
to overcome
the ill-posedness of an inverse problem 
(i.e., the nonexistence or instability of solutions)
is to add a regularization term to the objective function 
that penalizes the size of the involved parameters.
In the context of training problems of the type \eqref{eq:trainingpropprot2}, 
this approach
has the additional advantage that it 
allows to promote desirable sparsity properties 
of the vectors $\alpha \in D$ that are obtained from the optimization procedure, cf.\ 
\citep{Poerner2018,Hofmann2013,Pieper2020,Wen2016,Yoon2017} and the references therein.  
The aim of this subsection is to demonstrate that, as far as the existence of 
spurious local minima and the expressiveness of 
nonlinear approximation schemes are concerned, 
adding a regularization term to the loss function in \eqref{eq:trainingpropprot2}
can also have detrimental effects. 
The main idea in the following is
to exploit that 
many commonly used approximation instruments 
possess a ``linear'' lowest level 
in the sense that they depend on the product of 
the input variable $\xx$ and a matrix $A$ whose entries are 
part of the parameter vector $\alpha$. In the situation 
of our standing \cref{ass:standingassumpssec4},
this structural property can be expressed as follows: 

\begin{assumption}[Linearity of the Lowest Level]
\label[assumption]{ass:linearlowlev}~
\begin{itemize}
\item It holds $D = \R^m$, $\XX = \R^{d_\xx}$, 
and $(\YY, \|\cdot\|_\YY) = (\R^{d_\yy}, \|\cdot\|_2)$ with some $d_\xx, d_\yy \in \mathbb{N}$. 
\item There exists a function $\phi \colon \R^{p} \times \R^{q} \to \R^{d_\yy}$ such that, 
after reordering the vector $\alpha$ and reshaping it into a tuple $(\beta, A)  \in \R^p \times \R^{q \times d_\xx}$
with $ m = p + q d_\xx$, we can write $\psi(\alpha, \xx) = \phi(\beta, A\xx)$ for all
$\alpha \in \R^m$ and all $\xx \in \R^{d_\xx}$. 
\end{itemize}
\end{assumption}

Note that standard neural networks trivially satisfy the above conditions 
as they involve an affine-linear transformation on the lowest level, see \cref{subsec:6.2}.
A main feature of approximation schemes satisfying \cref{ass:linearlowlev} 
is that they behave polynomially when linearized at points $\alpha \in D$ that,
after reordering, yield the matrix $A = 0$. More precisely, 
we have the following result: 

\begin{lemma}{\hspace{-0.05cm}\bf(Polynomial First- and Second-Order Approximations)}\label[lemma]{lemma:polystruct}
Suppose that \cref{ass:standingassumpssec4,ass:linearlowlev} hold.
Consider further an arbitrary but fixed $\bar \alpha \in \R^m$
which, after the reshaping procedure in \cref{ass:linearlowlev}, 
takes the form $(\bar \beta, 0) \in \R^p \times \R^{q \times d_\xx}$.
Then, the following is true:
\begin{enumerate}[label=\roman*)]
\item\label{item:lowtang:i}
If the map $\phi \colon \R^{p} \times \R^{q} \to \R^{d_\yy}$ in \cref{ass:linearlowlev} 
is differentiable at $(\bar \beta, 0)$, then there exists a subspace $V_1 \subset Y$
of dimension at most $d_\yy(d_\xx + 1)$ satisfying 
\begin{equation*}
\Psi(\bar \alpha, x_d) + \partial_\alpha \Psi(\bar \alpha, x_d)\left \langle h \right \rangle \in V_1\quad \forall h \in \R^m.
\end{equation*}

\item\label{item:lowtang:ii}
If the map $\phi \colon \R^{p} \times \R^{q} \to \R^{d_\yy}$ in \cref{ass:linearlowlev} 
is continuously differentiable in an open neighborhood of
the point $(\bar \beta, 0)$ and twice differentiable at $(\bar \beta, 0)$, 
then there exist a subspace $V_1 \subset Y$
of dimension at most $d_\yy(d_\xx + 1)$ 
and a subspace $V_2 \subset Y$ of dimension
at most $\frac12 d_\yy(d_\xx + 2)(d_\xx + 1)$ satisfying
\begin{equation}
\label{eq:randomeq283672329-1}
V_1 \subset V_2,\qquad 
\Psi(\bar \alpha, x_d) + \partial_\alpha \Psi(\bar \alpha, x_d)\left \langle h \right \rangle \in V_1\quad \forall h \in \R^m,
\end{equation}
and
\begin{equation}
\label{eq:randomeq283672329-2}
\Psi(\bar \alpha, x_d) + \partial_\alpha \Psi(\bar \alpha, x_d)\left \langle h \right \rangle
+ \frac12 \partial_\alpha^2 \Psi(\bar \alpha, x_d)\left \langle h,h \right \rangle \in V_2\quad \forall h \in \R^m.
\end{equation}
\end{enumerate}
Here, $\partial_\alpha \Psi(\bar \alpha, x_d)\left \langle h \right \rangle$
and $\partial_\alpha^2 \Psi(\bar \alpha, x_d)\left \langle h,h \right \rangle$ denote the
first and the second derivative of the function $\Psi$ w.r.t.\  the variable $\alpha$ at $\bar \alpha$ evaluated 
at $h$ and $(h,h)$, respectively.  
\end{lemma}

\begin{proof}
Suppose that an $\bar \alpha \in \R^m$ satisfying the assumptions 
in point \ref{item:lowtang:i} of the lemma is given. 
Then, it holds
\begin{equation*}
\begin{aligned}
\psi(\bar \alpha + h, \xx)
&= \phi(\bar \beta + \tilde h, 0 + H\xx)
= \phi(\bar \beta, 0) + \phi'(\bar \beta, 0)\langle \tilde h, H\xx  \rangle 
+
\oo\big(\| (\tilde h, H\xx)\|_2 \big)
\\
&=
\phi(\bar \beta, 0) 
+
\phi'(\bar \beta, 0)\langle \tilde h, 0 \rangle 
+ \sum_{i=1}^{d_\xx}\phi'(\bar \beta, 0)  \langle 0, H e_i \rangle \xx_{\;i}
+\oo\big (\| (\tilde h, H\xx)\|_2 \big)
\end{aligned}
\end{equation*}
for all arbitrary but fixed $\xx \in \R^{d_\xx}$ and all $h \in \R^m$ which, after the reshaping
procedure in \cref{ass:linearlowlev},
take the form $(\tilde h, H) \in \R^p \times \R^{q \times d_\xx}$. 
Here, 
with $\phi'(\bar \beta, 0)\left \langle \cdot \right \rangle \colon \R^{p} \times \R^{q} \to \R^{d_\yy}$
we mean the first derivative of the function $\phi$ at $(\bar \beta, 0)$,
the Landau symbol refers to the limit $\|(\tilde h, H\xx)\|_2 \to 0$,
and $e_i$, $i=1,..., d_\xx$, are the unit vectors in $\R^{d_\xx}$.  
Note that the above implies in particular that 
\begin{equation*}
\begin{aligned}
&\psi(\bar \alpha, \xx) + \partial_\alpha \psi(\bar \alpha, \xx)\left \langle h \right \rangle
=
\phi(\bar \beta, 0) 
+
\phi'(\bar \beta, 0)\langle \tilde h, 0 \rangle 
+ \sum_{i=1}^{d_\xx}\phi'(\bar \beta, 0)  \langle 0, H e_i \rangle \xx_{\;i}
= P_{\bar \beta, h}(\xx)
\end{aligned}
\end{equation*}
holds for all $h \in \R^m$ and all arbitrary but fixed $\xx \in \R^{d_\xx}$, where 
$P_{\bar \beta, h}\colon \R^{d_\xx} \to \R^{d_\yy}$ is an affine map that depends only on $\bar \beta$ and $h$. 
The function $\xx \mapsto \psi(\bar \alpha, \xx) + \partial_\alpha \psi(\bar \alpha, \xx)\left \langle h \right \rangle$
is thus contained in a subspace of dimension $\smash{d_\yy(d_\xx + 1)}$ 
that is independent of $h$, namely the space of vector-valued polynomials of degree at most one
which map $\R^{d_\xx}$ to $\R^{d_\yy}$. Since the function $\Psi$ is defined
by $\Psi(\alpha, x_d) :=\left \{\psi(\alpha, \xx_{\;d}^k) \right \}_{k=1}^n$, 
i.e., by plugging in certain values of $\xx$, the first claim of the lemma now follows immediately. 

To establish the assertion in \ref{item:lowtang:ii}, we can proceed completely 
along the same lines as in the first part of the proof. By Taylor's formula, we obtain that
\begin{equation*}
\begin{aligned}
&\psi(\bar \alpha + h, \xx)
= \phi(\bar \beta + \tilde h, 0 + H\xx)
\\
&~=
\phi(\bar \beta, 0) + \phi'(\bar \beta, 0) \langle \tilde h, H\xx \rangle 
+
\frac12 \phi''(\bar \beta, 0) \langle (\tilde h, H\xx) , (\tilde h, H\xx) \rangle 
+
\oo\big (\| (\tilde h, H\xx)\|_2^2 \big)
\\
&~=
\phi(\bar \beta, 0) 
+ 
\phi'(\bar \beta, 0) \langle \tilde h, 0\rangle 
+
\frac12 \phi''(\bar \beta, 0) \langle (\tilde h, 0) , (\tilde h, 0) \rangle 
\\
&\qquad 
+  \sum_{i=1}^{d_\xx}
\left ( \phi'(\bar \beta, 0)  \langle 0, H e_i \rangle
+
\frac12 \phi''(\bar \beta, 0) \langle (0, H e_i) , (\tilde h, 0) \rangle 
+
\frac12 \phi''(\bar \beta, 0) \langle (\tilde h, 0) , (0, H e_i) \rangle 
\right ) \xx_{\;i}
\\
&\qquad 
+
\frac12 \sum_{i=1}^{d_\xx} \sum_{j=1}^{d_\xx} \phi''(\bar \beta, 0) \langle (0, H e_i) , (0, H e_j) \rangle 
\xx_{\;i} \xx_{\;j}
+\oo\big (\| (\tilde h, H\xx)\|_2^2 \big)
\end{aligned}
\end{equation*}
holds for all arbitrary but fixed $\xx \in \R^{d_\xx}$ and all vectors $h \in \R^m$ which, after reshaping,
take the form $(\tilde h, H) \in \R^p \times \R^{q \times d_\xx}$. Here,
$\phi'(\bar \beta, 0)\left \langle \cdot \right \rangle \colon \R^{p} \times \R^{q} \to \R^{d_\yy}$
again denotes the first and
$\phi''(\bar \beta, 0)\left \langle \cdot , \cdot \right \rangle \colon  (\R^{p} \times \R^{q})^2 \to \R^{d_\yy}$
the second derivative of $\phi$ at $(\bar \beta, 0)$. The above yields
\begin{equation*}
\begin{aligned}
&\psi(\bar \alpha, \xx) + \partial_\alpha \psi(\bar \alpha, \xx)\left \langle h \right \rangle
= P_{\bar \beta, h}(\xx)\quad \forall \xx \in \R^{d_\xx}
\end{aligned}
\end{equation*}
and
\begin{equation*}
\psi(\bar \alpha, \xx) + \partial_\alpha \psi(\bar \alpha, \xx)\left \langle h \right \rangle
+ \frac12 \partial_\alpha^2 \psi(\bar \alpha, \xx)\left \langle h,h \right \rangle
=
Q_{\bar \beta, h}(\xx)\quad \forall \xx \in \R^{d_\xx}
\end{equation*}
for all $h \in \R^m$, where 
$P_{\bar \beta, h}\colon \R^{d_\xx} \to \R^{d_\yy}$ and 
$Q_{\bar \beta, h}\colon \R^{d_\xx} \to \R^{d_\yy}$ 
are vector-valued polynomials of degree at most one and two, respectively, that depend only on $\bar \beta$ and $h$. 
The functions $\xx \mapsto \psi(\bar \alpha, \xx) + \partial_\alpha \psi(\bar \alpha, \xx)\left \langle h \right \rangle$
and $\xx \mapsto \psi(\bar \alpha, \xx) + \partial_\alpha \psi(\bar \alpha, \xx)\left \langle h \right \rangle
+ \frac12 \partial_\alpha^2 \psi(\bar \alpha, \xx)\left \langle h,h \right \rangle$
are thus contained in subspaces of dimension $d_\yy(d_\xx + 1)$ and $\frac12 d_\yy(d_\xx + 2)(d_\xx + 1)$, respectively, 
that are independent of $h$, namely the spaces of polynomials of degree at most one and two, respectively, 
which map $\R^{d_\xx}$ to $\R^{d_\yy}$. 
By again exploiting the definition of  $\Psi$, the assertion of \ref{item:lowtang:ii} now follows immediately. 
This completes the proof. 
\end{proof}

We would like to point out that \cref{lemma:polystruct} 
is not only interesting for the study of regularized training problems but also for the 
results on the existence of non-optimal critical points that we have derived in \cref{subsec:5.2}. Indeed, as a straightforward consequence 
of \cref{theorem:existencestatpts,lemma:polystruct}, we obtain: 

\begin{corollary}{\hspace{-0.05cm}\bf(Critical Points in the Presence of a Linear Lowest Level)}\label[corollary]{corollary:structstatpts}
Suppose that \cref{ass:standingassumpssec4,ass:linearlowlev} are satisfied and that  $d_\xx + 1 < n$ holds. 
Consider further an arbitrary but fixed $\bar \alpha \in \R^m$
which, after reshaping, 
takes the form $(\bar \beta, 0) \in \R^p \times \R^{q \times d_\xx}$,
and assume that the map $\phi \colon \R^{p} \times \R^{q} \to \R^{d_\yy}$ in \cref{ass:linearlowlev} 
is differentiable at $(\bar \beta, 0)$. 
Then, for every $\varepsilon > 0$, there exist uncountably many label vectors $y_d \in Y$ satisfying
\begin{equation*}
\left (\frac{\Theta(\Psi, x_d)}{1 - \Theta(\Psi, x_d)}\right)^{1/2} \|\Psi(\bar \alpha, x_d)\|_Y < 
\|\Psi(\bar \alpha, x_d) - y_d\|_Y < \left (\frac{\Theta(\Psi, x_d)}{1 - \Theta(\Psi, x_d)}\right)^{1/2} \|\Psi(\bar \alpha, x_d)\|_Y + \varepsilon
\end{equation*}
such that $\bar \alpha$ 
is a spurious local minimum or a saddle point of 
\eqref{eq:trainingpropprot2}. Further, for every arbitrary but fixed $C>0$,
there exist uncountably many $y_d \in Y$ such that $\bar \alpha$ 
is a spurious local minimum or a saddle point of 
\eqref{eq:trainingpropprot2} and such that \eqref{eq:errorestimates}
and \eqref{eq:lossestimate42} hold.
\end{corollary}

\begin{proof}
From \cref{lemma:polystruct}, we obtain that, in the considered situation, 
the linear hull of the vectors $\Psi(\bar \alpha, x_d)$ 
and $\partial_i \Psi(\bar \alpha, x_d)$, $i=1,...,m$, is contained in a 
subspace of dimension $d_\yy(d_\xx + 1) < d_\yy n = \dim(Y) = \dim(\YY^n)$. 
This shows that $\bar \alpha$ satisfies the assumptions of \cref{theorem:existencestatpts}.
By invoking this theorem, the claim of the corollary follows immediately. 
\end{proof}

Note that, in the case $d_\xx + 1 \geq n$, for every training set $(\xx_{\;d}^k, \yy_d^k)$, $k=1,...,n$,
with $\dim(\span(\xx_{\;d}^2 - \xx_{\;d}^1,...,\xx_{\;d}^n - \xx_{\;d}^1)) = n-1$, we can find an 
affine-linear $T\colon \R^{d_\xx} \to \R^{d_\yy}$ with $T(\xx_{\;d}^k) = \yy_d^k$ for all $k=1,...,n$. 
The condition $d_\xx + 1 < n$ 
in \cref{corollary:structstatpts} is thus directly related to the 
approximation properties of affine functions.
In \cref{subsec:6.2}, \cref{corollary:structstatpts} will in particular allow us to show
that, for neural networks with differentiable activations, there is always a 
subspace of parameters $\alpha$ in $\R^m$ that can be turned into 
arbitrarily bad saddle points or spurious local minima of \eqref{eq:trainingpropprot2}
by choosing appropriate label vectors $y_d \in Y$, see \cref{cor:saddleoverpar}.

To show that regularized versions of the training problem 
\eqref{eq:trainingpropprot2} can indeed possess spurious local minima, we will use that,
by adding a regularization term to the loss function of \eqref{eq:trainingpropprot2},
the stationary points in \cref{corollary:structstatpts} can be transformed into 
local minimizers. 
This leads to:
\begin{theorem}{\hspace{-0.05cm}\bf(Spurious Local Minima in the Presence of Regularization Terms)}\label{theorem:spuriousregprob}
Suppose that \cref{ass:standingassumpssec4,ass:linearlowlev} 
are satisfied and that $\frac12(d_\xx + 2)(d_\xx + 1) < n$ holds. 
Consider an arbitrary but fixed vector $\bar \alpha \in \R^m$
which, after the reshaping procedure in \cref{ass:linearlowlev},
takes the form $(\bar \beta, 0) \in \R^p \times \R^{q \times d_\xx}$,
and assume that the function $\phi$ in \cref{ass:linearlowlev} 
is continuously differentiable in an open neighborhood of $(\bar \beta, 0)$ and twice differentiable at $(\bar \beta, 0)$.
Assume further that a function $g\colon \R^m \to [0, \infty)$ with $g(0) = 0$
is given such that there exist a constant $c>0$ and an open neighborhood $U \subset \R^m$
of the origin with $g(z) \geq c \|z\|_2^2$ for all $z \in U$.
Then, for every arbitrary but fixed $C>0$, there exist uncountably many combinations 
of training label vectors $y_d \in Y$ and regularization parameters $\nu > 0$ such that 
$\bar \alpha$ is a spurious local minimum of the regularized training problem
\begin{equation}
\label{eq:trainingpropprot-reg}
\min_{\alpha \in D} \|\Psi( \alpha, x_d) - y_d\|_Y^2 + \nu g(\alpha - \bar \alpha)
\end{equation}
that satisfies a local quadratic growth condition of the form
\begin{equation}
\label{eq:quadgrowth253543}
\|\Psi( \alpha, x_d) - y_d\|_Y^2 + \nu g(\alpha - \bar \alpha)  
\geq \|\Psi(\bar \alpha, x_d) - y_d\|_Y^2 + \nu g(0) + \varepsilon \|\alpha - \bar \alpha\|_2^2 ~~\quad \forall \alpha \in B_r(\bar \alpha)
\end{equation}
for some $\varepsilon, r > 0$ and 
\begin{equation}
\label{eq:randomeq28h73}
\inf_{\alpha \in D} \|\Psi( \alpha, x_d) - y_d\|_Y^2 + \nu g(\alpha - \bar \alpha)
+ C 
\leq 
 \|\Psi( \bar \alpha, x_d) - y_d\|_Y^2 + \nu g(0).
\end{equation}
\end{theorem}

\begin{proof}
Suppose that a point $\bar \alpha$ satisfying the 
assumptions of the theorem is given and let $V_2$ denote 
the subspace from part \ref{item:lowtang:ii} of \cref{lemma:polystruct}.
Then, it follows from the inequality $\frac12(d_\xx + 2)(d_\xx + 1) < n$ 
that
$\dim(V_2) < \dim (Y)$ holds and that 
 there exists a $v \in Y$ that is orthogonal to $V_2$ and satisfies $\|v\|_Y = 1$.
Let us again define  $y_d^s := \Psi(\bar \alpha, x_d) + s v \in Y$, $s \in \R$,
and assume that $C>0$ is an arbitrary but fixed constant. 
Then, we obtain completely analogously to the proof of \cref{theorem:existencestatpts} that
there exists an $M > 0$ with 
\begin{equation*}
\inf_{\alpha \in D} \|\Psi( \alpha, x_d) - y_d^s\|_Y^2 + C + 2 \leq  \|\Psi( \bar \alpha, x_d) - y_d^s\|_Y^2
\end{equation*}
for all $s \in \R$ with $|s| > M$. Note that the above in particular implies that,
for all $s \in \R$ with $|s| > M$, we can find an $\tilde \alpha_s \in D$ with 
\begin{equation*}
\|\Psi( \tilde \alpha_s, x_d) - y_d^s\|_Y^2 + C + 1\leq  \|\Psi( \bar \alpha, x_d) - y_d^s\|_Y^2.
\end{equation*}
If we now choose a $\nu_s > 0$ for all $s \in \R$ with $|s| > M$ such that 
$\nu_s g(\tilde \alpha_s - \bar \alpha) < 1$ holds and exploit the identity $g(0) = 0$, 
then it readily follows that 
\begin{equation*}
\|\Psi( \tilde \alpha_s, x_d) - y_d^s\|_Y^2 + \nu_s g(\tilde \alpha_s - \bar \alpha) + C
\leq  \|\Psi( \bar \alpha, x_d) - y_d^s\|_Y^2 + \nu_s g(0).
\end{equation*}
This establishes \eqref{eq:randomeq28h73}. It remains to prove that,
for all of the above $y_d^s$, $\nu_s$, and $s$, the vector $\bar \alpha$ is indeed a 
spurious local minimum of \eqref{eq:trainingpropprot-reg} that satisfies a local quadratic growth condition
of the form \eqref{eq:quadgrowth253543}. 
To this end, we note that the binomial identities, the definition of $y_d^s$, the choice of $v$,
and our assumptions on $g$ and $\bar \alpha$ imply that 
\begin{equation}
\label{eq:randomeq726353}
\begin{aligned}
&\|\Psi( \alpha, x_d) - y_d^s\|_Y^2 + \nu_s g(\alpha - \bar \alpha)  
-
\|\Psi( \bar \alpha, x_d) - y_d^s\|_Y^2 - \nu_s g(0)  
\\
&\quad\geq
\|\Psi( \alpha, x_d) - \Psi(\bar \alpha, x_d) - s v  \|_Y^2  - \| s v \|_Y^2 
+ \nu_s c \|\alpha - \bar \alpha \|_2^2
\\
&\quad=
\|\Psi( \alpha, x_d) - \Psi(\bar \alpha, x_d) \|_Y^2  
- 2\left ( \Psi( \alpha, x_d) - \Psi(\bar \alpha, x_d), s v  \right )_Y
+ \nu_s c \|\alpha - \bar \alpha \|_2^2
\\
&\quad=
\|\Psi( \alpha, x_d) - \Psi(\bar \alpha, x_d) \|_Y^2  + \nu_s c \|\alpha - \bar \alpha \|_2^2
\\
&\quad\quad~~- 2\left ( \Psi( \alpha, x_d) 
- \Psi(\bar \alpha, x_d) 
- \partial_\alpha \Psi(\bar \alpha, x_d)\left \langle \alpha - \bar \alpha \right \rangle
- \frac12 \partial_\alpha^2 \Psi(\bar \alpha, x_d)\left \langle \alpha - \bar \alpha, \alpha - \bar \alpha \right \rangle, s v  \right )_Y
\\
&\quad= \|\Psi( \alpha, x_d) - \Psi(\bar \alpha, x_d) \|_Y^2  + \nu_s c \|\alpha - \bar \alpha \|_2^2 + \oo(\|\alpha - \bar \alpha\|_2^2)
\end{aligned}
\end{equation}
holds for all $\alpha \in D$ with $\alpha - \bar \alpha \in U$, 
where the Landau symbol refers to the limit $\|\alpha - \bar \alpha\|_2 \to 0$. 
By choosing sufficiently small $\varepsilon, r > 0$, it now follows immediately that 
$\bar \alpha$ is a spurious local minimum of \eqref{eq:trainingpropprot-reg} that satisfies \eqref{eq:quadgrowth253543}.
This completes the proof. 
\end{proof}
\begin{remark}~\label[remark]{rem:Oneighborhood}
\begin{itemize}
\item Analogously to the condition $d_\xx + 1 < n$ in \cref{corollary:structstatpts},
which is directly related to the approximation capabilities of affine functions $T\colon \R^{d_\xx} \to \R^{d_\yy}$
(cf.\ the comments after \cref{corollary:structstatpts}), the assumption 
 $\frac12(d_\xx + 2)(d_\xx + 1) < n$ in \cref{theorem:spuriousregprob}
is directly related to the approximation properties of second-order polynomials. 
Indeed, as seen in the proof of \cref{lemma:polystruct}, this inequality expresses 
that the dimension $n d_\yy$ of the overall output space $Y$ should be larger than 
the dimension of the space of 
vector-valued polynomials of degree at most two mapping $\XX = \R^{d_\xx}$ to $\YY = \R^{d_\yy}$.
The condition $\frac12(d_\xx + 2)(d_\xx + 1) < n$
thus---roughly speaking---ensures that there exist non-pathological choices of the training data that 
cannot be fit precisely by a quadratic vector-valued polynomial. Such an assumption 
is necessary in \cref{theorem:spuriousregprob} because its proof relies on the non-surjectivity of
a second-order Taylor approximation of the function $\Psi(\cdot, x_d)$, 
see the choice of the vector $v$ used in the construction 
of the spurious local minima in \eqref{eq:quadgrowth253543}. 
We remark that, 
although the condition $\frac12(d_\xx + 2)(d_\xx + 1) < n$ prevents \cref{theorem:spuriousregprob}
from being applicable in some applications, in the context of classical regression problems
(as arising, e.g., in the field of partial differential equations),
in which quadratic ansatz functions are typically unable to fit the training data precisely, 
this assumption is not overly restrictive. 
\item 
If the function $g$ is continuous in an open neighborhood of the origin, then 
the quadratic growth condition \eqref{eq:quadgrowth253543} and the fact that the function values of the objective 
in \eqref{eq:trainingpropprot-reg} depend continuously on $y_d$ and $\nu$ imply that,
in the situation of \cref{theorem:spuriousregprob}, there 
always exists an open nonempty subset $O$ of $Y \times (0, \infty)$ such that 
the regularized problem \eqref{eq:trainingpropprot-reg} possesses a spurious local minimum for all $(y_d, \nu) \in O$. 
To see this, let us suppose that 
$(\tilde y_d, \tilde \nu) \in Y \times (0, \infty)$ is a tuple 
such that the assertion of  \cref{theorem:spuriousregprob} holds for some $\bar \alpha \in \R^m$,
$C>0$, $\varepsilon > 0$, and $r>0$, and let $\delta \in (0, r)$ be chosen such that 
the functions $g(\cdot - \bar \alpha)\colon \R^m \to \R$ 
and $\Psi(\cdot, x_d)\colon D \to Y$ are continuous on the closed ball $B_{2\delta}(\bar \alpha)$.
(Such a $\delta$ can be found due to the regularity assumptions on $g$ and $\phi$.)
Then, it follows from \eqref{eq:quadgrowth253543} and \eqref{eq:randomeq28h73} that  
\begin{equation}
\label{eq:randomeq3674638}
\begin{aligned}
&\min_{\alpha \in  \R^m,\, \|\alpha - \bar \alpha\|_2 = \delta}
\big (  \|\Psi( \alpha, x_d) - \tilde y_d\|_Y^2 + \tilde \nu g(\alpha- \bar \alpha)  \big)
\\
&>  \min_{\alpha \in B_\delta(\bar \alpha)} \big ( \|\Psi(\alpha , x_d) - \tilde y_d\|_Y^2 + \tilde \nu g(\alpha-\bar \alpha) \big)
> \|\Psi(\tilde \alpha, x_d) -  \tilde y_d\|_Y^2 + \tilde \nu g(\tilde \alpha - \bar \alpha)
\end{aligned}
\end{equation}
holds for some $\tilde \alpha \in \R^m$. Since the function 
\[
B_{2\delta}(\bar \alpha) \times Y \times (0, \infty) \ni (\alpha, y_d, \nu) \mapsto \|\Psi( \alpha, x_d) -  y_d\|_Y^2 
+ \nu g( \alpha - \bar \alpha) \in \R
\]
is continuous, it is uniformly continuous on every compact subset of its domain of definition. 
In combination with the compactness of the sets ${\{\alpha \in  \R^m \mid\|\alpha - \bar \alpha\|_2 = \delta \}}$
and $B_\delta(\bar \alpha)$, this implies that all of the inequalities in \eqref{eq:randomeq3674638} 
remain valid for all $(y_d, \nu)$ in a sufficiently small open neighborhood $O \subset Y \times (0, \infty)$
of the tuple $(\tilde y_d, \tilde \nu)$. From the first of these inequalities, however, it follows 
that the loss function in \eqref{eq:trainingpropprot-reg} possesses a local minimum in the interior of the ball $B_\delta(\bar \alpha)$
(namely in that point at which the minimum of the loss on the compact set $B_\delta(\bar \alpha)$ is attained),
and from the second inequality that this local minimum cannot be global. Thus,
we indeed obtain that \eqref{eq:trainingpropprot-reg} possesses a spurious local minimum 
for all $(y_d, \nu) \in O$. Note that
the local quadratic growth condition \eqref{eq:quadgrowth253543} 
is typically lost when passing over from $(\tilde y_d, \tilde \nu)$
to the perturbed tuples $(y_d, \nu) \in O$ here since it is not stable w.r.t.\ uniform convergence. 
 \end{itemize}
\end{remark}

As \cref{theorem:spuriousregprob} demonstrates, the addition of 
a regularization term to the objective function of \eqref{eq:trainingpropprot2}
may create spurious local minima 
by introducing a bias towards certain values of the parameter vector $\alpha$.
The proof of \cref{theorem:spuriousregprob} further shows that these effects are a direct consequence of the 
approximation property \ref{fa:II} 
and will typically appear when the regularization parameter $\nu$ is too small
relative to the size of the training label vector $y_d$
(see the conditions $|s| > M$ and $\nu_s g(\tilde \alpha_s - \bar \alpha) < 1$).
Choosing $\nu$ too large, however, is also not a good idea as the following result demonstrates:

\begin{theorem}{\hspace{-0.05cm}\bf(Loss of Approximation Property \ref{fa:II} by Regularization)}\label{theorem:approxgone}
Suppose that \cref{ass:standingassumpssec4,ass:linearlowlev} hold,
that the function 
$\phi \colon \R^{p} \times \R^{q} \to \R^{d_\yy}$ in \cref{ass:linearlowlev} 
is continuously differentiable in an open neighborhood of the origin and
twice differentiable at the origin,
and that $\frac12(d_\xx + 2)(d_\xx + 1) < n$ and $\phi(0,0) = 0$ holds.
Assume further
that a function $g \colon\R^m \to [0, \infty)$ with $g(0) = 0$ is given
such that there exist constants $c_1, c_2>0$ and an open neighborhood $U \subset \R^m$ of the origin
satisfying $g(z) \geq c_1 \|z\|_2^2$ for all $z \in U$ and 
$g(z) \geq c_2$ for all $z \in \R^m \setminus U$. Then, for every arbitrary but fixed 
$\nu > 0$,  there exist 
uncountably many $y_d \in Y \setminus \{0\}$ such that $\bar \alpha = 0$ is the 
unique global solution of the problem 
\begin{equation}
\label{eq:trainingpropprot-reg2}
\min_{\alpha \in D} \|\Psi( \alpha, x_d) - y_d\|_Y^2 + \nu g(\alpha).
\end{equation}
\end{theorem}

\begin{proof}
The assumptions of the theorem imply that $\bar \alpha = 0 \cong (0, 0) \in \R^p \times \R^{q \times d_\xx}$ satisfies the 
conditions in part \ref{item:lowtang:ii} of \cref{lemma:polystruct} with $\frac12 d_\yy(d_\xx + 2)(d_\xx + 1) < \dim(Y)$. 
We can thus again find a proper subspace $V_2$ of $Y$ such that 
the first and second derivatives of the function $\Psi(\cdot, x_d)\colon D \to Y$ at $\bar  \alpha = 0$
are contained in $V_2$ in the sense of \eqref{eq:randomeq283672329-1} and \eqref{eq:randomeq283672329-2},
choose a vector $v \in Y$ with $\|v\|_Y = 1$ that is orthogonal to $V_2$, 
and define $y_d^s := \Psi(0, x_d) + s v \in Y$ for all $s \in \R$. 
Since $\phi$ is twice differentiable
at the origin and since $U$ is open, 
we further obtain that there exists an $r > 0$ with $\alpha \in U$ and
\begin{equation*}
\left \| 
\Psi( \alpha, x_d) 
- \Psi(0, x_d) 
- \partial_\alpha \Psi(0, x_d)\left \langle \alpha \right \rangle
- \frac12 \partial_\alpha^2 \Psi(0, x_d)\left \langle \alpha, \alpha\right \rangle
\right \|_Y \leq \frac12 \|\alpha\|_2^2 
\end{equation*}
for all $\alpha \in \R^m$ with $\|\alpha\|_2 \leq r$. Note that this estimate and exactly the same calculation as in 
\eqref{eq:randomeq726353} yield 
\begin{equation*}
\begin{aligned}
&\|\Psi( \alpha, x_d) - y_d^s\|_Y^2 + \nu g(\alpha)  
-
\|\Psi( 0, x_d) - y_d^s\|_Y^2 - \nu g(0)  
\\
&\geq
\nu c_1 \|\alpha \|_2^2
- 2\left ( \Psi( \alpha, x_d) 
- \Psi(0, x_d) 
- \partial_\alpha \Psi(0, x_d)\left \langle \alpha \right \rangle
- \frac12 \partial_\alpha^2 \Psi(0, x_d)\left \langle \alpha , \alpha\right \rangle, s v  \right )_Y
\\
&\geq 
(\nu c_1  - |s|)\|\alpha  \|_2^2
\end{aligned}
\end{equation*}
for all $\alpha$ with $\|\alpha\|_2 \leq r$ and all arbitrary but fixed $\nu > 0$. 
For all $\alpha \in \R^m$ with $\|\alpha\|_2 > r$,
on the other hand, we have 
\begin{equation*}
\|\Psi( \alpha, x_d) - y_d^s\|_Y^2 + \nu g(\alpha)  
-
\|\Psi( 0, x_d) - y_d^s\|_Y^2 - \nu g(0)  
\geq 
\nu \min(c_2, c_1 r^2) - s^2.
\end{equation*}
By combining the last two estimates, it follows that $\bar \alpha = 0$ is
the unique global solution of \eqref{eq:trainingpropprot-reg2} with label vector $y_d^s$ for all $s \in \R$
with $|s|< \min(\nu c_1, (\nu c_2)^{1/2}, (\nu c_1)^{1/2} r)$. 
\end{proof}

\Cref{theorem:approxgone} shows that,
although the function $\Psi(\cdot, x_d)\colon D \to Y$ is able to 
provide a best approximation for every $y_d \neq 0$
that is better than the origin by \ref{fa:II},
the regularized problem \eqref{eq:trainingpropprot-reg2} may very well 
possess the optimal solution $\bar \alpha = 0$ with the associated vector $\Psi(0, x_d) = 0$
for nonzero label vectors $y_d$. (Recall that $\phi(0,0) = 0$ implies $\psi(0, \xx) = 0$ for all $\xx \in \R^{d_\xx}$
so that we indeed have $\Psi(0, x_d) = 0$ here.) 
Adding a regularization term to the objective function of \eqref{eq:trainingpropprot2} thus impairs the approximation 
properties of the map $\Psi(\cdot, x_d)\colon D \to Y$ and compromises the property \ref{fa:II} 
that distinguishes the function $\psi$ from a linear approximation scheme in the first place, cf.\ the discussion in \cref{sec:4}.  
Note that the estimate $|s|< \min(\nu c_1, (\nu c_2)^{1/2}, (\nu c_1)^{1/2} r)$ in the proof of \cref{theorem:approxgone}
suggests that these effects get worse when $\nu$ increases as in this case
the solution $\bar \alpha = 0$ is obtained from  \eqref{eq:trainingpropprot-reg2}
for vectors $y_d$ with larger norms. We would like to point out that 
studying how the approximation properties of the global solutions 
of the problem \eqref{eq:trainingpropprot-reg2} are affected by the choice of the tuple $(\nu, g)$
is in general far from straightforward. The main reason for this is that,
in a nonlinear approximation scheme, there is typically no immediate connection between, e.g., 
the norm of the parameter vector $\alpha$ and the size of the output $\Psi(\alpha, x_d)$
so that it is a-priori often completely unclear which features of the elements 
of the image $\Psi(D, x_d)$ are penalized by a term of the form $\nu g(\alpha)$. 
 
We conclude this section with a result that shows that 
the addition of a regularization term to the objective function of \eqref{eq:trainingpropprot2} does not necessarily 
remove the instability and nonuniqueness of solutions, either:

\begin{theorem}{\hspace{-0.05cm}\bf(Instability and Nonuniqueness in the Regularized Case)}\label{theorem:nonuniquereg}
Suppose that \cref{ass:standingassumpssec4,ass:linearlowlev} 
are satisfied, 
that $\frac12(d_\xx + 2)(d_\xx + 1) < n$ holds,
and that $\bar \alpha$, $\bar \beta$, $\phi$, $g$, $c$, and $U$
are as in \cref{theorem:spuriousregprob}.
Then, there exist uncountably many combinations of regularization parameters $\nu>0$
and label vectors $y_d \in Y$ such that
there exist an $s_0 \geq 0$, a sequence $\{y_d^s\}_{s > s_0} \subset Y$, 
and an open neighborhood $\tilde U \subset \R^m$
of $\bar \alpha$ satisfying $y_d^s \to y_d$ for $s \to s_0$, 
\begin{equation}
\label{eq:randomeq28363636}
\tilde U  \cap   \argmin_{\alpha \in D} 
 \|\Psi( \alpha, x_d) - y_d^s\|_Y^2 + \nu g(\alpha - \bar \alpha)   = \emptyset\qquad \forall s > s_0,
\end{equation}
and
\begin{equation}
\label{eq:randomeq182736}
\bar \alpha \in \argmin_{\alpha \in D} 
\|\Psi( \alpha, x_d) - y_d \|_Y^2 + \nu g(\alpha - \bar \alpha).
\end{equation}
Further, there are uncountably many tuples $(y_d, \nu) \in Y \times (0, \infty)$ with the above properties such that 
there exists a family $\{\alpha_s\}_{s > s_0} \subset D \setminus \tilde U$ satisfying
\begin{equation*}
\lim_{s \to s_0}  \|\Psi(\alpha_s, x_d) - y_d\|_Y^2 
+ \nu g(\alpha_s - \bar \alpha)
 =  \|\Psi(\bar \alpha, x_d) - y_d\|_Y^2
+ \nu g(0).
\end{equation*}
The solutions of the regularized training problem \eqref{eq:trainingpropprot-reg} thus possess the same nonuniqueness 
and instability properties as the optimization problem in \cref{cor:instabilityoverpara}. 
\end{theorem}

\begin{proof}
Suppose that $\bar \alpha$ is an arbitrary but fixed point satisfying the assumptions of the theorem 
and let $v \in Y$ and $y_d^s$, $s \in \R$, be defined as in the proof of \cref{theorem:spuriousregprob}.
Then, from exactly the same construction as in the proof of \cref{theorem:spuriousregprob},
we obtain that there exist uncountably many tuples $(\nu, y_d^{\bar s})$, $\nu > 0$, $\bar s >0$,
with w.l.o.g.\ different $\nu$
such that 
\begin{equation}
\label{eq:randomeq273653}
\inf_{\alpha \in D} \|\Psi( \alpha, x_d) - y_d^{\bar s}\|_Y^2 + \nu g(\alpha - \bar \alpha)
<
 \|\Psi( \bar \alpha, x_d) - y_d^{\bar s}\|_Y^2 + \nu g(0)
\end{equation}
holds. 
Let us fix such a tuple $(\nu, y_d^{\bar s})$ and consider the auxiliary function
\begin{equation*}
F\colon  [0, \bar s] \to [0, \infty),\qquad s \mapsto \inf_{\alpha \in D} \|\Psi( \alpha, x_d) - y_d^{s}\|_Y^2 + \nu g(\alpha - \bar \alpha).
\end{equation*}
We claim that this function is Lipschitz continuous. 
Indeed, for all $s_1, s_2 \in [0, \bar s]$ and every sequence $\{\alpha_i\} \subset D$
with 
\begin{equation*}
\lim_{i \to \infty} \|\Psi( \alpha_i, x_d) - y_d^{s_1}\|_Y^2 + \nu g(\alpha_i - \bar \alpha) = F(s_1),
\end{equation*}
we obtain from the non-negativity of $g$ and the definitions of $F$, $y_d^{s_1}$, and $\{\alpha_i\}$ that
\begin{equation*}
\begin{aligned}
0 &\leq \limsup_{i \to \infty }\| \Psi( \alpha_i, x_d) -  \Psi(\bar \alpha, x_d)\|_Y
\\
&\leq 
\limsup_{i \to \infty } \left ( \| \Psi( \alpha_i, x_d) -  y_d^{s_1}\|_Y^2 + \nu g(\alpha_i - \bar \alpha) \right )^{1/2}+ s_1\|v\|_Y
\\
&=
 \left (  \inf_{\alpha \in D} \|\Psi( \alpha, x_d) - y_d^{s_1}\|_Y^2 + \nu g(\alpha - \bar \alpha) \right )^{1/2}+ s_1 
\\
&\leq
 \left (   \|\Psi( \bar \alpha, x_d) - y_d^{s_1}\|_Y^2 + \nu g(0) \right )^{1/2}+ s_1 
\\
&\leq 2\bar s,
\end{aligned}
\end{equation*}
and, as a consequence,
\begin{equation*}
\begin{aligned}
F(s_1)&=  \lim_{i \to \infty} \|\Psi( \alpha_i, x_d) - y_d^{s_1}\|_Y^2 + \nu g(\alpha_i - \bar \alpha)
\\
&= 
\lim_{i \to \infty} \|\Psi( \alpha_i, x_d) - \Psi(\bar \alpha, x_d) - s_2 v - (s_1 - s_2)v\|_Y^2 + \nu g(\alpha_i - \bar \alpha)
\\
&\geq 
\limsup_{i \to \infty} 
\|\Psi( \alpha_i, x_d) -  y_d^{s_2} \|_Y^2 + \nu g(\alpha_i - \bar \alpha)
- 2 \| \Psi( \alpha_i, x_d) -  \Psi(\bar \alpha, x_d) - s_2 v\|_Y  |s_1 - s_2| 
\\
&\geq 
F(s_2)
- 6 \bar s \, |s_1 - s_2|. 
\end{aligned}
\end{equation*}
After exchanging the roles of $s_1$ and $s_2$, we thus have 
$
 | F(s_1) - F(s_2) | \leq 6 \bar s \, |s_1 - s_2|
$
and $F$ is Lipschitz continuous on $[0, \bar s]$ as claimed. Consider now the value
\begin{equation}
\label{eq:defs0}
s_0 := \inf 
\left \{ 
\tilde s \in [0, \bar s]
\,
\big |
\,
F(s) < \|\Psi( \bar \alpha, x_d) - y_d^{s}\|_Y^2 + \nu g(0)
~~
\forall s \in (\tilde s, \bar s]
\right \}.
\end{equation}
Then, it follows from the continuity of $F$, the definition of $y_d^s$, \eqref{eq:randomeq273653},
and the trivial identity $F(0) = 0 = \|\Psi( \bar \alpha, x_d) - y_d^{0}\|_Y^2 + \nu g(0)$ that 
$s_0$ satisfies $0 \leq s_0 < \bar s$ and
\begin{equation*}
\inf_{\alpha \in D} \|\Psi( \alpha, x_d) - y_d^{s_0}\|_Y^2 + \nu g(\alpha - \bar \alpha) 
=  \|\Psi( \bar \alpha, x_d) - y_d^{s_0}\|_Y^2 + \nu g(0).
\end{equation*}
This shows that $\bar \alpha$ satisfies \eqref{eq:randomeq182736} 
for $y_d := y_d^{s_0}$. To see that the above construction also yields \eqref{eq:randomeq28363636},
we note that the same estimates as in \eqref{eq:randomeq726353} imply that
\begin{equation*}
\begin{aligned}
&\|\Psi( \alpha, x_d) - y_d^s\|_Y^2 + \nu  g(\alpha - \bar \alpha)  
-
\|\Psi( \bar \alpha, x_d) - y_d^s\|_Y^2 - \nu g(0)  
\\
&\qquad \geq
\|\Psi( \alpha, x_d) - \Psi(\bar \alpha, x_d) \|_Y^2  + \nu c \|\alpha - \bar \alpha \|_2^2
\\
&\qquad\qquad - 2\bar s \left \| \Psi( \alpha, x_d) 
- \Psi(\bar \alpha, x_d) 
- \partial_\alpha \Psi(\bar \alpha, x_d)\left \langle \alpha - \bar \alpha \right \rangle
- \frac12 \partial_\alpha^2 \Psi(\bar \alpha, x_d)\left \langle \alpha - \bar \alpha, \alpha - \bar \alpha \right \rangle \right \|_Y 
\\
&\qquad\geq 0
\end{aligned}
\end{equation*}
holds for all 
$0 \leq s < \bar s$ and all $\alpha \in \R^m$ in a sufficiently small open neighborhood $\tilde U \subset \R^m$ of $\bar \alpha$ 
that depends only on $\Psi$, $g$, $\nu$, and the bound $\bar s$. As $\bar \alpha$ is not a global minimizer of the problem \eqref{eq:trainingpropprot-reg}
for all label vectors $y_d^s$ with $s_0 < s < \bar s$ by the definition of $s_0$ in \eqref{eq:defs0},
this shows that $\bar \alpha$, $\tilde U$, and the vectors $y_d^s$, $s \in (s_0, \bar s)$, indeed satisfy \eqref{eq:randomeq28363636}.
Since the convergence $y_d^s \to y_d^{s_0}$ for $s \to s_0$ is trivial, this proves the first part of the theorem.
(Note that we indeed end up with uncountably many different tuples $(\nu, y_d)$
with the desired properties here since, although we have modified the label vector
during the course of the proof,
we have not altered the regularization parameter $\nu$.)
To establish the second assertion of the theorem,
it suffices to note that the above considerations and the triangle inequality imply that, for all $s \in (s_0, \bar s)$,
there has to exist an $\alpha_s \in D \setminus \tilde U$ with 
\begin{equation*}
\begin{aligned}
\|\Psi( \bar \alpha, x_d) - y_d^{s_0}\|_Y^2 + \nu g(0)
&= \|\Psi( \bar \alpha, x_d) - y_d^{s}\|_Y^2 + \nu g(0) + \oo(1)
\\
&>
\|\Psi(\alpha_s, x_d) - y_d^{s}\|_Y^2 + \nu g(\alpha_s - \bar \alpha) + \oo(1)
\\
&=
\|\Psi(\alpha_s, x_d) - y_d^{s_0}\|_Y^2 + \nu g(\alpha_s - \bar \alpha) + \oo(1)
\\
&\geq 
\inf_{\alpha \in D} \|\Psi( \alpha, x_d) - y_d^{s_0}\|_Y^2 + \nu g(\alpha - \bar \alpha)  + \oo(1)
\\
&= \|\Psi( \bar \alpha, x_d) - y_d^{s_0}\|_Y^2 + \nu g(0) + \oo(1),
\end{aligned}
\end{equation*}
where the Landau symbol refers to the limit $(s_0, \bar s) \ni s \to s_0$. This proves the claim.
\end{proof}

\section{Application to Tangible Examples}
\label{sec:6}
With the abstract results of \cref{sec:5} in place, we are in the position 
to turn our attention to tangible examples and applications. 
In what follows, we will first consider
a classical free-knot spline interpolation scheme that 
is closely related to neural networks with ReLU activation functions, see \cref{subsec:6.1}. 
After this, we turn our attention to training problems 
for neural networks with various architectures, see \cref{subsec:6.2}.

\subsection{Free-Knot Linear Spline Interpolation}
\label{subsec:6.1}
The first example that we consider in this section is a special instance of a
dictionary approximation approach that generalizes classical 
piecewise linear interpolation---the so-called free-knot linear spline interpolation method. 
This technique is based on the idea to not only adapt the function values 
at the nodes of a linear spline to the function that is to be approximated,
but also to vary the nodes of the underlying mesh.
For details on this topic and its background, we refer to \cite{DeVore1998} and \cite{Daubechies2019}.
The setting that we consider in this subsection is the following:

\begin{setting}[Free-Knot Linear Spline Interpolation]~\label[setting]{set:variableknot}
\begin{itemize}
\item $\XX = \YY = \R$ and the norm of $\YY$ is just the absolute value function, 
\item $n \in \mathbb{N}$ and $n \geq 2$,
\item $x_d := \{\xx_{\;d}^k\}_{k=1}^n \in \XX^n$ is an arbitrary but fixed vector 
satisfying  $\xx_{\;d}^1 < \xx_{\;d}^2 <...< \xx_{\;d}^n$, 
\item $m = 2p$, $p \in \mathbb{N}$, $p \geq 3$, and $D \subset \R^m \cong \R^p \times \R^p$ is defined by 
\begin{equation}
\label{eq:Ddefvarknot}
D := 
\left \{
\alpha = (\beta, \gamma) \in \R^p \times \R^p
\,
\Big |
\,
\gamma_1 < \gamma_2 < ... < \gamma_p
\right \},
\end{equation}
\item $\psi\colon D \times \XX \to \YY$ is defined by 
\begin{equation}
\label{eq:variableknotfunction}
\begin{aligned}
\psi((\beta, \gamma), \xx)
:=\begin{cases}
\beta_1 & \text{if } \xx \leq \gamma_1
\\
\displaystyle
\frac{\gamma_{j+1} - \xx}{\gamma_{j+1} - \gamma_j} \beta_{j }
+
\frac{\xx - \gamma_{j}}{\gamma_{j+1} - \gamma_j} \beta_{j + 1}
& \text{if } \gamma_j < \xx \leq \gamma_{j+1},\, j \in \{1,...,p - 1\}
\\
\beta_p & \text{if } \xx > \gamma_p.
\end{cases}
\end{aligned}
\end{equation}
\end{itemize}
\end{setting}

Note that the above situation is trivially covered by our general \cref{ass:notation}.
To simplify the notation, 
in what follows,
we will again use the abbreviations collected in \cref{def:basicdefnotation}.
For every arbitrary but fixed training label vector $y_d = \{ \yy_d^k\}_{k=1}^n \in Y = \YY^n$,
our squared-loss training problem \eqref{eq:trainingpropprot} thus takes the form 
\begin{equation}
\label{eq:trainingpropprot4}
\min_{\alpha = (\beta, \gamma) \in D} \|\Psi( \alpha, x_d) - y_d\|_Y^2 
= \frac{1}{2n}\sum_{k=1}^n  \big ( \psi(\alpha, \xx_{\;d}^k) - \yy_d^k\big )^2. 
\end{equation}
Similarly to the example from \cref{sec:4},
the problem \eqref{eq:trainingpropprot4} models the task 
of finding a vector of breakpoints $\gamma \in \R^p$ and a coefficient vector $\beta \in \R^p$
such that the map in \eqref{eq:variableknotfunction}
possesses function values at the points $\xx_{\;d}^k$, $k=1,...,n$,
that fit the given data vector $y_d := \{\yy_d^k\}_{k=1}^n$
optimally in the least-squares sense. 
Solving \eqref{eq:trainingpropprot4} for the function \eqref{eq:variableknotfunction}
is thus a problem of nonlinear regression. 
As already mentioned, the free-knot interpolation scheme \eqref{eq:variableknotfunction}
is closely related to neural networks 
involving the ReLU activation function. 
In fact, it has been shown by \cite{Daubechies2019} that the image of a ReLU-based network
with a real in- and output
is always contained in the image of the scheme \eqref{eq:variableknotfunction} for a sufficiently 
large $p$ and that the image of \eqref{eq:variableknotfunction} is always contained in the 
image of a ReLU-network of sufficient depth and width. 
See \cite[Sections 3 and 4]{Daubechies2019} for precise results on this topic.  
We will see below that, as far as the optimization landscape and the stability 
of the training problem \eqref{eq:trainingpropprot4} are concerned, ReLU-networks and 
the scheme \eqref{eq:variableknotfunction} share many common properties as well.

To be able to apply the abstract results of \cref{sec:5} to the scheme \eqref{eq:variableknotfunction} and 
the problem \eqref{eq:trainingpropprot4},
we have to check if the conditions in \cref{ass:standingassumpssec4} are satisfied. 
This, however, is an easy task:

\begin{lemma}
{\hspace{-0.05cm}\bf(Conicity and Improved Expressiveness)}\label[lemma]{lemma:varableknotassumptions}
In the situation of \cref{set:variableknot}, 
the map $\Psi(\cdot, x_d)\colon D \to Y$ associated with $x_d$ and 
the free-knot spline interpolation scheme $\psi$ possesses the properties \ref{fa:I} and \ref{fa:II}.
Moreover, the number $\Theta(\Psi, x_d)$ in \eqref{eq:defTheta} 
associated with $x_d$ and the function $\psi$ in \eqref{eq:variableknotfunction} satisfies $\Theta(\Psi, x_d) \leq 1 - 1/n$.
\end{lemma}

\begin{proof}
If we consider an arbitrary but fixed $y \in \Psi(D, x_d)$, then there exists a tuple $(\beta, \gamma) \in D$
with $\Psi((\beta, \gamma), x_d) = y$ and it follows immediately from \eqref{eq:Ddefvarknot} 
and \eqref{eq:variableknotfunction} that we also have $(s\beta, \gamma) \in D$
and $s y \in \Psi(D, x_d)$ for all $s \in \R$. This establishes \ref{fa:I}.
To prove \ref{fa:II}, let us assume that an arbitrary but fixed 
$y_d \in Y \setminus \{0\}$ is given. Then, there exists at least one $l \in \{1,...,n\}$
with $\yy_d^{l} \neq 0$. Consider now a vector $(\beta, \gamma) \in D$ with 
\begin{equation}
\label{eq:randomeq273636}
\begin{gathered}
\beta_1 = \beta_p = 0, \qquad \beta_j = \yy_d^{l}\quad\forall j = 2,...,p - 1,
\qquad \gamma_1 < ... < \gamma_p,
\\
\xx_{\;d}^l = \gamma_2,\qquad \xx_{\;d}^k \not\in [\gamma_1, \gamma_p]\quad\forall k \neq l.
\end{gathered}
\end{equation}
(Note that such a vector always exists by our assumptions on the entries of $x_d$.)
Then, \eqref{eq:variableknotfunction} and the definition of $\Psi$ yield that 
$\Psi((\beta, \gamma), x_d) = \yy_d^{l}\,e_l$ holds, where $e_l$ denotes the $l$-th unit vector 
of $Y = \R^n$, and we obtain from \eqref{eq:trainingpropprot4} that 
\begin{equation*}
\|\Psi((\beta, \gamma), x_d) - y_d\|_Y^2
=
\frac{1}{2n} \sum_{k \neq l} \big| \yy_d^k\big |^2 < \|y_d\|_Y^2.
\end{equation*}
This proves \ref{fa:II}. To finally establish the inequality $\Theta(\Psi, x_d) \leq 1 - 1/n$,
it suffices to note that every $y_d \in Y$ with $\|y_d\|_Y = 1$ possesses at least one entry that 
has an absolute value greater than or equal to $\sqrt{2}$. In combination with the construction in \eqref{eq:randomeq273636}, this yields
\begin{equation*}
\Theta(\Psi, x_d) =
\sup_{y_d \in Y,\, \|y_d\|_Y = 1} \left ( \inf_{y \in  \closure_Y\left (\Psi(D, x_d)\right )} \|y - y_d\|_Y^2 \right)
\\
\leq 
1 - \frac{1}{n} 
\end{equation*}
and completes the proof. 
\end{proof}

Note that the estimate $\Theta(\Psi, x_d) \leq 1 - 1/n$ in \cref{lemma:varableknotassumptions} is very pessimistic. 
(We have, after all, only used one node to establish it.) Deriving better estimates for this quantity 
not only for the scheme \eqref{eq:variableknotfunction} but also for the 
neural networks discussed in the next subsection is an interesting topic
and more precise results on the number $\Theta(\Psi, x_d)$ and its 
dependence on $x_d$ and $n$ would certainly improve the understanding 
of the expressiveness of nonlinear approximation schemes---in particular in 
view of the inequality \eqref{eq:randomeq2736352}. We leave this topic for future research. 

Next, we collect some results on the mapping properties 
of the function $\Psi(\cdot, x_d)\colon D \to Y$
that make it possible to decide which theorems of \cref{sec:5} are applicable to \eqref{eq:variableknotfunction}:
\begin{lemma}
{\hspace{-0.05cm}\bf(Mapping Properties of the Free-Knot Spline Interpolation Scheme)}\label[lemma]{lemma:varknotaux}
In the situation of \cref{set:variableknot}, the following is true:
\begin{enumerate}[label=\roman*)]
\item 
\label{item:lemmavarknot:i}
If $n \leq p$ holds, then we have  $\Psi(D, x_d) =  Y$.
\item 
\label{item:lemmavarknot:ii}
If $n > p$ holds, then we have $\closure_Y(\Psi(D, x_d)) \neq  Y$.
\item 
\label{item:lemmavarknot:iii}
Define 
\begin{equation}
\label{eq:VspaceDef23}
V := 
\left \{
\{\zz_k\}_{k=1}^n \in Y
\,
\big |
\,
\exists a, b \in \R \text{ such that }  \zz_k = a \xx_{\;d}^k + b~~\forall k=1,...,n
\right \}.
\end{equation}
Then, for every element $z$ of the subspace $V$, there exist a point $\bar \alpha \in D$ and an open set $U \subset D$
satisfying $\bar \alpha \in U$, $z = \Psi(\bar \alpha, x_d)$, and $\Psi(U, x_d) \subset V$. 

\item 
\label{item:lemmavarknot:iv}
If $n > 3p$ holds, then, for every $\bar \alpha \in D$, 
there exist an open set $U \subset D$ with $\bar \alpha \in U$
and a subspace $V$ of $Y$ with $V \neq Y$ such that $\Psi(U, x_d) \subset V$ holds. 
\end{enumerate}
\end{lemma}
\begin{proof}
If we suppose that $n \leq p$ holds
and that $y = \{ \yy_k\}_{k=1}^n$
is an arbitrary but fixed element of the space $Y = \R^n$,
then every $\alpha = (\beta, \gamma) \in D$
with $\gamma_k := \xx_{\;d}^k$ and $\beta_k := \yy_k$ for all $k=1,...,n$
satisfies $\Psi( \alpha, x_d) = y$. This establishes the equality in \ref{item:lemmavarknot:i}.

To prove \ref{item:lemmavarknot:ii},
let us assume that there is a situation with $p < n$ and $\closure_Y(\Psi(D, x_d)) =  Y$. 
Then, the density of the set $\Psi(D, x_d)$ in $Y$ implies that 
we can find an $\bar \alpha = (\bar \beta, \bar  \gamma) \in D$ 
which satisfies  $|\psi(\bar \alpha, \xx_{\;d}^k) - (-1)^k| < 0.1$ for all $k=1,...,n$. 
Consider now an interval of the form $(\xx_{\;d}^{2l}, \xx_{\;d}^{2l+2})$, 
$l=1,..., \left \lfloor{(n - 2)/2}\right \rfloor$, where $ \lfloor{\cdot \rfloor}$ denotes the floor function. 
Then,  by the properties of $\bar \alpha$, we have 
$\psi(\bar \alpha, \xx_{\;d}^{2l}) > 0.9$, 
$\psi(\bar \alpha, \xx_{\;d}^{2l + 1}) < -0.9$, and 
$\psi(\bar \alpha, \xx_{\;d}^{2l+2}) > 0.9$,
and it follows that the function 
$\R \ni \xx \mapsto \psi(\bar \alpha, \xx) \in \R$ attains its minimum 
on $[\xx_{\;d}^{2l}, \xx_{\;d}^{2l+2}]$ in the open interval $(\xx_{\;d}^{2l}, \xx_{\;d}^{2l+2})$. 
As the map $\xx \mapsto \psi(\bar \alpha, \xx)$ is piecewise linear, we also know that 
this minimum has to be attained at one of the breakpoints $\bar \gamma_j$, $j=1,...,p$.
Note that the cases $j = 1$ and $j = p$ are impossible here since
the map $\xx \mapsto \psi(\bar \alpha, \xx)$ is constant on the left of $\bar \gamma_1$
and on the right of $\bar \gamma_p$, since we know that 
the minimal function value in $[\xx_{\;d}^{2l}, \xx_{\;d}^{2l+2}]$ is smaller than $-0.9$,
and since $\psi(\bar \alpha, \xx_{\;d}^{2l}) > 0.9$ and $\psi(\bar \alpha, \xx_{\;d}^{2l+2}) > 0.9$. 
In summary, we may thus conclude that each of the open intervals $(\xx_{\;d}^{2l}, \xx_{\;d}^{2l+2})$,
$l=1,..., \left \lfloor{(n - 2)/2}\right \rfloor$, 
has to contain (at least) one $\bar \gamma_j$ with $j \in \{2,...,p-1\}$ and $\psi(\bar \alpha, \bar \gamma_j) = \bar \beta_j < - 0.9$.
Using exactly the same arguments (with maxima instead of minima), we also obtain that 
each of the intervals $(\xx_{\;d}^{2l - 1}, \xx_{\;d}^{2l+1})$, $l=1,..., \left \lfloor{(n - 1)/2}\right \rfloor$,
has to contain (at least) one $\bar \gamma_j$ with $j \in \{2,...,p-1\}$ and $\psi(\bar \alpha, \bar \gamma_j) = \bar \beta_j > 0.9$. 
Since the intervals in both of these groups are mutually disjoint and due to the different conditions on the function values,
it now follows immediately that there have to be at least 
$\left \lfloor{(n - 2)/2}\right \rfloor + \left \lfloor{(n - 1)/2}\right \rfloor + 2 = n$ breakpoints in \eqref{eq:variableknotfunction}.
Thus, $p \geq n$ which contradicts our assumption $p < n$. This establishes \ref{item:lemmavarknot:ii}.\pagebreak

Next, we prove \ref{item:lemmavarknot:iii}: Let $z = \{\zz_k\}_{k=1}^n$ be an arbitrary 
but fixed element of the space $V$ in \eqref{eq:VspaceDef23} with associated $a,b \in \R$, i.e., $\zz_k = a \xx_{\;d}^k + b$
for all $k=1,...,n$.
Then, we can clearly find a point $\bar \alpha = (\bar \beta, \bar \gamma) \in D$ with the properties 
\begin{equation}
\label{eq:randomeq273553}
\bar \gamma_1 <  \xx_{\;d}^1 < \xx_{\;d}^2 <...< \xx_{\;d}^n < \bar \gamma_2 < ... < \bar \gamma_p,
\qquad
\bar \beta_1 = a \bar \gamma_1 + b, \quad \text{and}\quad \bar \beta_2 = a \bar \gamma_2 + b.
\end{equation}
Due to \eqref{eq:variableknotfunction}, such an $\bar \alpha \in D$ trivially satisfies $z = \Psi(\bar \alpha, x_d)$,
and from the strictness of the inequalities in \eqref{eq:randomeq273553} 
we obtain that there exists an open neighborhood  $U \subset D$ of $\bar \alpha$
such that $\gamma_1 <  \xx_{\;d}^1 < \xx_{\;d}^2 <...< \xx_{\;d}^n < \gamma_2$ holds for all $\alpha = (\beta, \gamma) \in U$.
Since the latter property implies that the map $\xx \mapsto \psi(\alpha, \xx)$ is affine-linear
on the open interval $(\gamma_1, \gamma_2)$ and that $ \xx_{\;d}^k \in (\gamma_1, \gamma_2)$
holds for all $k=1,...,n$, it follows immediately that $\Psi(U, x_d) \subset V$. 
In summary, we thus have $\bar \alpha \in U$, $z = \Psi(\bar \alpha, x_d)$, and $\Psi(U, x_d) \subset V$
and the proof of \ref{item:lemmavarknot:iii} is complete. 

To finally obtain \ref{item:lemmavarknot:iv},
we note that, in the case $n > 3p$, every $\bar \alpha = (\bar \beta, \bar \gamma) \in D$ 
has to satisfy (at least) one of the following three conditions (as one may easily check by contradiction):
\begin{equation*}
\begin{gathered}
\xx_{\;d}^1 < \xx_{\;d}^2 < \bar \gamma_1,\qquad \bar \gamma_p <  \xx_{\;d}^{n-1} < \xx_{\;d}^{n},
\\
\exists j \in \{1,..., p-1\} \text{ and }l \in \{1,..., n-2\}\colon~\bar \gamma_j < \xx_{\;d}^{l} < \xx_{\;d}^{l+1} < \xx_{\;d}^{l+2} < \bar \gamma_{j+1}.
\end{gathered}
\end{equation*}
Due to the definition of $\psi$ in \eqref{eq:variableknotfunction}
and the strictness of the involved inequalities, the above implies that every 
 $\bar \alpha \in D$ admits an open neighborhood $U \subset D$ such that (at least) one of the following is true:
\begin{equation*}
\begin{gathered}
\Psi(U, x_d) \subset V_1 := \left \{ z \in Y = \R^n \, \big | \, z_1 = z_2 \right \},
\quad~~
\Psi(U, x_d) \subset V_2 := \left \{ z \in Y = \R^n \, \big | \, z_{n-1} = z_n \right \},
\\
\Psi(U, x_d) \subset V_3 := \left \{ z \in Y = \R^n \, \Bigg | 
\, \frac{\xx_{\;d}^{l+2} - \xx_{\;d}^{l+1}}{\xx_{\;d}^{l+2} - \xx_{\;d}^{l}} z_l + \frac{\xx_{\;d}^{l+1} - \xx_{\;d}^{l}}{\xx_{\;d}^{l+2} - \xx_{\;d}^{l}} z_{l + 2} = z_{l+1}   \right \}
\end{gathered}
\end{equation*}
for some $l \in \{1,..., n-2\}$. The assertion of \ref{item:lemmavarknot:iv} now follows immediately. 
This completes the proof of the lemma. 
\end{proof}~\\[-0.8cm]

We remark that, in application problems involving free-knot splines, the number of 
training data samples $n$ typically exceeds the number of nodes $p$ by far. The assumption $n > 3p$
in \cref{lemma:varknotaux}\ref{item:lemmavarknot:iv} is thus not very restrictive in practice.
Compare, e.g., with the comments on this topic and the numerical 
experiments in \cite[Sections 1 and 5]{Schwetlick1995}.
By invoking the abstract results of \cref{sec:5}, we now obtain (for example)
the following for the free-knot linear spline interpolation scheme in \eqref{eq:variableknotfunction}:

\begin{corollary}
{\hspace{-0.05cm}\bf(Properties of Squared-Loss Training Problems)}\label[corollary]{cor:knotapprox}
In the situation of \cref{set:variableknot}, the following is true: 
\begin{enumerate}[label=\roman*)]
\item\label{freebird:item:i}
 {\bf (Nonuniqueness and Instability of Best Approximations)} 
If $n > p$ holds, then the best approximation map 
\begin{equation*}
P_{\Psi}^{x_d}\colon Y \rightrightarrows Y,\qquad y_d \mapsto \argmin_{y \in  \closure_Y\left (\Psi(D, x_d)\right )} \|y - y_d\|_Y^2,
\end{equation*}
associated with the free-knot linear spline interpolation scheme \eqref{eq:variableknotfunction} 
is set-valued and there exist uncountably many training label vectors $y_d \in Y$ satisfying
$|P_{\Psi}^{x_d}(y_d)| > 1$. Moreover, the map $P_{\Psi}^{x_d}\colon Y \rightrightarrows Y$ 
is discontinuous in the sense that, for 
every arbitrary but fixed $C>0$, there exist uncountably many $y_d \in Y$ 
such that there are sequences $\{y_d^l\}, \{\tilde y_d^l\} \subset Y$ with the properties in \eqref{eq:randomeq2837}. 

\item\label{freebird:item:ii}
{\bf (Existence of Spurious Local Minima)}  
If $n > 2$ holds, then
there exists an open nonempty cone $K \subset Y$ such that
the training problem \eqref{eq:trainingpropprot4} possesses 
at least one spurious local minimum satisfying a growth condition of the form \eqref{eq:quadgrowthstates}
for all $y_d \in K$. These spurious local minima can be arbitrarily bad in the sense 
that, for every arbitrary but fixed $C>0$, there exist uncountably many $y_d \in K$
such that at least one of the spurious local minima of \eqref{eq:trainingpropprot4} satisfies 
\eqref{eq:errorestimates}, \eqref{eq:lossestimate42}, and \eqref{eq:quadgrowthstates}.
The size of the cone $K$ depends on the extent to which condition \ref{fa:II} is satisfied,
cf.\ \eqref{eq:Kspurdef}.
If $n \leq p$ holds, then the cone $K$ consists of all vectors in $Y$ that are not affine-linearly fittable,
i.e., it holds $K = Y \setminus V$ with the subspace $V$ in \eqref{eq:VspaceDef23}. 

\item\label{freebird:item:iii}
{\bf (Every Point is a Potential Spurious Local Minimum in the Case $\mathbf{n > 3p}$)}  
If $n > 3p$ holds, then, for every $\bar \alpha \in D$ and every arbitrary but fixed $C>0$,
there exist uncountably many label vectors $y_d$ such that $\bar \alpha$ is a 
spurious local minimum of the training problem \eqref{eq:trainingpropprot4}
that satisfies \eqref{eq:errorestimates}, \eqref{eq:lossestimate42}, and a quadratic 
growth condition of the form \eqref{eq:quadgrowthstates}. 

\item\label{freebird:item:iv}
{\bf (Instability of Solutions in the Case $\mathbf{2 < n \leq p}$)}  
If $2 < n \leq p$ holds, then the solution operator
\begin{equation*}
 Y \ni
y_d \mapsto \argmin_{\alpha \in D} \|\Psi( \alpha, x_d) - y_d\|_Y^2 \subset D
\end{equation*}
of the problem \eqref{eq:trainingpropprot4} is discontinuous
in the sense that there exist points $\bar \alpha \in D$ and vectors $y_d \in Y$
such that there are a family $\{y_d^s\}_{s > 0} \subset Y$ and an open neighborhood 
$U \subset D$ of $\bar \alpha$ satisfying $y_d^s \to y_d$ for $s \to 0$, 
\begin{equation*}
\bar \alpha \in  \argmin_{\alpha \in D} \|\Psi( \alpha, x_d) - y_d\|_Y^2,
\end{equation*}
and 
\begin{equation*}
U \, \cap\, \argmin_{\alpha \in D} \|\Psi( \alpha, x_d) - y_d^s\|_Y^2  = \emptyset\quad \forall s >0.
\end{equation*}

\item\label{freebird:item:v}
{\bf (Nonuniqueness of Solutions in the Case $\mathbf{2 < n \leq p}$)}  
If $2 < n \leq p$ holds, then 
there exist choices of the training label vector $y_d$ such that the problem 
\eqref{eq:trainingpropprot4} is not uniquely solvable in the sense
of minimizing sequences. More precisely, 
there exist vectors $y_d \in Y$ such that there are an $\bar \alpha \in D$, an open set
$U \subset D$, and 
a family $\{\alpha_s\}_{s > 0}$
satisfying $\bar \alpha \in U$, $\{\alpha_s\}_{s > 0} \subset D \setminus U$, and 
\begin{equation*}
\lim_{s \to 0}  \|\Psi(\alpha_s, x_d) - y_d\|_Y^2  =\|\Psi(\bar \alpha, x_d) - y_d\|_Y^2
= \inf_{\alpha \in D} \|\Psi(\alpha, x_d) - y_d\|_Y^2. 
\end{equation*}
 \end{enumerate}
\end{corollary}
\begin{proof}
To prove the various claims of the corollary, it suffices to combine 
\cref{lemma:varableknotassumptions,lemma:varknotaux}
with 
\cref{theorem:abstractinstability,prop:existencehotspurs,theorem:badcone,cor:instabilityoverpara}
in \cref{sec:5}. 
\end{proof}

Using the same ideas as in the proofs of 
\cref{lemma:varableknotassumptions}, \cref{lemma:varknotaux}, and 
\cref{cor:knotapprox},
one can also show that the results of \cref{sec:5} can be applied to other 
nonlinear approximation schemes that 
are based on the idea to not only 
optimize the coefficient vector w.r.t.\ a certain basis 
but also the choice of the basis itself. 
We omit a detailed discussion of this topic
to avoid overloading this paper and will focus on the consequences 
that our results have for neural networks instead. 

\subsection{Neural Networks}
\label{subsec:6.2}
Next, we apply our abstract results
to neural networks with vector-valued in- and outputs. 
The setting that we consider in this subsection is as follows: 
\begin{setting}[Setting for the Study of Neural Networks]~\label[setting]{set:NNN}
\begin{itemize}
\item $\XX := \R^{d_\xx}$, $\YY := \R^{d_\yy}$, $d_\xx, d_\yy \in \mathbb{N}$, and $\YY$ is endowed with 
the Euclidean norm $\|\cdot\|_2$,
\item $n \in \mathbb{N}$, $n \geq 2$, and $x_d := \{\xx_{\;d}^k\}_{k=1}^n \in \XX^n$ is an arbitrary but fixed training data vector 
satisfying $\xx_{\;d}^j \neq \xx_{\;d}^k$ for all $j \neq k$,  
\item $w_i \in \mathbb{N}$, $i=1,..., L$, $L \in \mathbb{N}$, are given numbers, $w_0 := d_\xx$, $w_{L+1} := d_\yy$, and
the set $D$ is defined by 
\begin{equation}
\label{eq:randomeq2635hdg36}
D := 
\left \{
\alpha = 
\{ (A_{i}, b_{i})\}_{i=1}^{L+1}
\, \Big | \,
A_{i} \in \R^{w_{i} \times w_{i-1}},\,
b_{i} \in \R^{w_{i}}~
\forall i=1,...,L+1
\right \},
\end{equation}
\item $\sigma_i\colon \R \to \R$, $i=1,...,L$, are given activation functions,
\item $\varphi_i^{A_i, b_i}\colon \R^{w_{i-1}} \to \R^{w_i}$, $i=1,...,L+1$, are the functions defined by 
\begin{equation}
\label{eq:randomeq27353628hd37wb}
\varphi_i^{A_i, b_i}(z) := \sigma_i\left (A_i z + b_i \right )~\forall i=1,...,L,
\qquad \varphi_{L+1}^{A_{L+1}, b_{L+1}}(z) := A_{L+1}z + b_{L+1},
\end{equation}
where $\sigma_i$ acts componentwise on the entries of the vectors $ A_{i}z + b_{i}$,
\item $\psi \colon D \times \XX \to \YY$ is defined by 
\begin{equation}
\label{eq:NNdef}
\psi(\alpha, \xx) 
:= \left ( \varphi_{L+1}^{A_{L+1}, b_{L+1}} \circ ... \circ \varphi_{1}^{A_{1}, b_{1}} \right )(\xx)
\end{equation}
for all $\xx \in \XX$ and all $\alpha = (A_{L+1}, b_{L+1},..., A_1, b_1)  \in D$. 
\end{itemize}
\end{setting}

Note that,
for the sake of brevity and readability, in the remainder of this section, we will not reorder
elements $ \alpha = \{ (A_{i}, b_{i})\}_{i=1}^{L+1} = (A_{L+1}, b_{L+1},..., A_1, b_1)$ 
of the parameter space $D$ in \eqref{eq:randomeq2635hdg36} as column vectors to conform
with the notation of \cref{sec:5}, i.e., we will not always explicitly state that we use the isomorphism 
\begin{equation}
\label{eq:randomisomorphism363}
\R^{w_{L+1} \times w_{L}} \times \R^{w_{L+1}} \times ... \times \R^{w_{1} \times w_{0}} \times \R^{w_{1}} 
\cong
\R^{m},\quad m := w_{L+1} (w_{L} + 1) + ... + w_{1}(w_{0} + 1),
\end{equation}
when referring to the results of the previous sections. 
We will further again employ the abbreviations introduced in \cref{def:basicdefnotation}
so that the squared-loss training problem for the neural network \eqref{eq:NNdef} reads as follows:
\begin{equation}
\label{eq:trainingpropprot5}
\min_{ \alpha = (A_{L+1}, b_{L+1},..., A_1, b_1) \in D} \|\Psi( \alpha, x_d) - y_d\|_Y^2
=
\frac{1}{2n}\sum_{k=1}^n  \|\psi(\alpha, \xx_{\;d}^k) - \yy_d^k\|_2^2. 
\end{equation}

We would like to point out that the situation in \cref{set:NNN} is a very general one 
as it not only allows for different widths of the $L$ hidden layers of the network 
but also for the use of different activation functions.
Compare, e.g., with the architectures considered by \cite{Daubechies2019} and \cite{Ding2020} in this context. 
As in \cref{subsec:6.1}, we begin our analysis of the 
approximation scheme \eqref{eq:NNdef} by checking whether 
the conditions \ref{fa:I} and \ref{fa:II} are satisfied.
For \ref{fa:I}, we obtain:

\begin{lemma}[Conicity of Neural Networks]
\label[lemma]{lemma:nnconic}
In the situation of \cref{set:NNN}, 
the function $\Psi(\cdot, x_d)\colon D \to Y$ associated with $x_d$ 
and the network $\psi$ in \eqref{eq:NNdef}
satisfies \ref{fa:I}. 
\end{lemma}

\begin{proof}
Since
$\varphi_{L+1}^{sA_{L+1}, sb_{L+1}} \circ ... \circ \varphi_{1}^{A_{1}, b_{1}}
=
s \varphi_{L+1}^{A_{L+1}, b_{L+1}} \circ ... \circ \varphi_{1}^{A_{1}, b_{1}}
$
holds for all $s \in \R$, the cone property in \ref{fa:I} follows immediately. 
\end{proof}

Verifying condition \ref{fa:II} is more delicate. 
In what follows, 
the main idea to establish this approximation property will be to 
first prove \ref{fa:II} for neural networks with Heaviside-type activations of the form
\begin{equation}
\label{eq:binary}
\sigma_i(s)=
\begin{cases}
0 & \text{ if } s < 0
\\
c_i  & \text{ if } s = 0
\\
1 & \text{ if } s > 0
\end{cases},\qquad c_i \in \R, 
\end{equation}
and to subsequently exploit that almost all
activation functions that are currently used in the literature can emulate step functions of the form \eqref{eq:binary}
by saturation. In combination with the observation in \cref{lemma:reformulatedcondition} 
that \ref{fa:II} is a property of the closure $\closure_Y\left (\Psi(D, x_d)\right )$
and completely independent of how the elements in this set are realized or approximated 
by the parameters in $D$, this then immediately yields the desired condition \ref{fa:II}.
We would like to point out that the above approach to the analysis of 
neural networks and, in view of the results of \cref{sec:5}, the existence of saddle points and spurious
local minima is conceptually very different from the techniques used, e.g., 
by \cite{Yun2019} and \cite{Goldblum2020Truth}
which primarily rely on the observation that many neural networks can locally 
imitate linear approximation schemes. 
In our analysis, the main step is not to exploit such a local linearity
but, on the contrary, to reduce the problem
to the situation where the activation functions are essentially binary and thus to the most nonlinear case possible. 
The starting point of our proof of \ref{fa:II} is the following lemma:

\begin{lemma}[A Separation Lemma]\label[lemma]{lemma:coneseparation}
Consider the situation in \cref{set:NNN}. 
Then, for every arbitrary but fixed $l \in \{1,...,n\}$, 
there exist a matrix $A \in \R^{2 \times d_\xx}$ and a vector $b \in \R^2$ satisfying 
\begin{equation*}
A\xx_{\;d}^k + b \in (-\infty, 0)^2 \cup (0, \infty)^2\quad \forall k \neq l
\qquad \text{and}\qquad  
A\xx_{\;d}^l + b\in (0, \infty) \times (-\infty, 0).
\end{equation*} 
\end{lemma}

\begin{proof}
To establish the assertion of the lemma, we first prove by induction w.r.t.\ $p \in \mathbb{N}$ that,
for every collection of vectors $\zz_1,...,\zz_p \in \R^{d_\xx} \setminus \{0\}$, 
there exists an $a \in \R^{d_\xx}$ with $a^T\zz_j \neq 0$ for all $j = 1,...,p$. 
For $p=1$, the existence of such an $a$ is trivial as we can simply choose $a := \zz_1$.
So let us assume that $p > 1$ holds. Then, 
the induction hypothesis yields that there exists an $a \in \R^{d_\xx}$
with $a^T\zz_j \neq 0$ for all $j = 1,...,p-1$. If this $a$ also satisfies 
$a^T\zz_p \neq 0$, then there is nothing left to show.
If, on the other hand, $a^T\zz_p = 0$ holds, then we can find a 
small $\varepsilon > 0$ with $(a + \varepsilon \zz_p)^T \zz_j \neq 0$
for all $j=1,...,p-1$, and it follows immediately that the 
vector $\tilde a := a + \varepsilon \zz_p$ has all of the desired properties.
This concludes the induction step. 

Consider now an arbitrary but fixed $l \in \{1,...,n\}$.
Then, the above result and our assumption $\xx_{\;d}^k \neq \xx_{\;d}^j $ for all $k\neq j$
imply that there exists an $a \in \R^{d_\xx}$ with 
$a^T(\xx_{\;d}^k - \xx_{\;d}^l) \neq 0$ for all $k \neq l$,
and we can find an $\varepsilon > 0$ with 
$a^T(\xx_{\;d}^k - \xx_{\;d}^l) \pm \varepsilon \neq 0$ and 
$\sgn(a^T(\xx_{\;d}^k - \xx_{\;d}^l) \pm \varepsilon) = \sgn(a^T(\xx_{\;d}^k - \xx_{\;d}^l))$
for all $k \neq l$. If we use this $\varepsilon$ to define 
\begin{equation*}
A := 
\begin{pmatrix}
a^T
\\
a^T
\end{pmatrix} \in \R^{2 \times d_\xx},
\qquad b := 
\begin{pmatrix}
\varepsilon - a^T\xx_{\;d}^l
\\
-\varepsilon - a^T\xx_{\;d}^l
\end{pmatrix} \in \R^2,
\end{equation*}
then it holds 
\begin{equation*}
A\xx_{\;d}^k + b = 
\begin{pmatrix}
a^T(\xx_{\;d}^k - \xx_{\;d}^l) + \varepsilon
\\
a^T(\xx_{\;d}^k - \xx_{\;d}^l) - \varepsilon
\end{pmatrix}
\in 
\begin{cases}
(-\infty, 0)^2 \cup (0, \infty)^2 & \text{ if } k \neq l 
\\
(0, \infty) \times (-\infty, 0) & \text{ if } k = l 
\end{cases}
\end{equation*}
as desired. This completes the proof. 
\end{proof}

Using \cref{lemma:coneseparation}, we can prove:
\begin{lemma}
{\hspace{-0.05cm}\bf(Approximation Property \ref{fa:II} for Heaviside-Type Activations)}\label[lemma]{lemma:heavisideapprox}
Consider the situation in \cref{set:NNN}. Suppose
that $w_1 \geq 2$ holds and that there exist constants $c_i \in \R$ with 
\begin{equation}
\label{eq:heavisideactivation}
\sigma_i(s)=
\begin{cases}
0 & \text{ if } s < 0
\\
c_i  & \text{ if } s = 0
\\
1 & \text{ if } s > 0
\end{cases}\qquad \forall i=1,...,L. 
\end{equation}
Then, the function  $\Psi(\cdot, x_d)\colon D \to Y$
associated with $x_d$ and the
neural network $\psi$ in \eqref{eq:NNdef} possesses the property \ref{fa:II}
and the number $\Theta(\Psi, x_d)$ in \eqref{eq:defTheta} satisfies $\Theta(\Psi, x_d) \leq 1 - 1/n$.
\end{lemma}

\begin{proof}
Suppose that a $y_d \in Y\setminus \{0\}$ is given. Then,
there exists at least one $l \in \{1,...,n\}$ with $\yy_d^l \in \R^{d_\yy} \setminus \{0\}$,
and it follows from our assumption $w_1 \geq 2$
and \cref{lemma:coneseparation} that we can find a matrix $A \in \R^{2 \times d_\xx}$
and a vector $b \in \R^2$ such that the parameters 
\begin{equation}
\label{eq:randomeq275353}
A_1 := 
\begin{pmatrix}
A
\\
0_{(w_1 - 2) \times w_0}
\end{pmatrix}
\in \R^{w_1 \times w_0}
\qquad \text{and}
\qquad
b_1 := 
\begin{pmatrix}
b
\\
0_{(w_1 - 2)}
\end{pmatrix}
\in \R^{w_1}
\end{equation}
satisfy 
\begin{equation*}
A_1 \xx_{\;d}^k + b_1
\in 
\begin{cases}
\big [ (-\infty, 0)^2 \times \{0_{(w_1 - 2)}\}\big ] \cup \big [ (0, \infty)^2 \times \{0_{(w_1 - 2)}\}\big ]  & \text{ if } k \neq l 
\\
(0, \infty) \times (-\infty, 0)  \times \{0_{(w_1 - 2)}\} & \text{ if } k = l.
\end{cases}
\end{equation*}
Here and in what follows, the symbols $0_{p \times q}$ and $0_p$ denote the zero matrix in $\R^{p \times q}$
and the zero (column) vector in $\R^p$, respectively, with the convention that these zero-blocks are ignored 
when $p$ or $q$ vanishes.
In combination with \eqref{eq:heavisideactivation}, the above implies in particular that 
\begin{equation*}
\underbrace{(1,-1, 0_{1 \times (w_1 - 2)})}_{\in \R^{1 \times w_1}}
\sigma_1\left (A_1 \xx_{\;d}^k + b_1 \right ) 
=
\begin{cases}
0  & \text{ if } k \neq l 
\\
1   & \text{ if } k = l.
\end{cases}
\end{equation*}
Let us now first consider the case $L=1$, i.e., the situation where the neural network \eqref{eq:NNdef}
possesses precisely one hidden layer. Then, the properties of $A_1$, $b_1$, and $\sigma_1$ and the definitions
\begin{equation*}
A_{L+1} = A_2 := 
\begin{pmatrix}
\yy_d^{l}, 0_{d_\yy \times (d_\yy-1)}
\end{pmatrix}
\begin{pmatrix}
1 , -1, 0_{1 \times (w_1 - 2)}
\\
0_{(w_2-1) \times w_1}
\end{pmatrix}
\in \R^{w_2 \times w_1}
\quad
\text{and}
\quad 
b_2 := 0_{w_2}
\end{equation*}
yield
\begin{equation*}
A_2 \sigma_1\left (A_1 \xx_{\;d}^k + b_1 \right ) + b_2
=
\begin{cases}
0_{d_\yy}  & \text{ if } k \neq l 
\\
\yy_d^l   & \text{ if } k = l.
\end{cases}
\end{equation*}
The parameter $\bar \alpha := (A_2, b_2, A_1, b_1) \in D$ thus satisfies 
\begin{equation}
\label{eq:randomeq162525}
\|\Psi(\bar \alpha, x_d) - y_d\|_Y^2 
=
\frac{1}{2n} \sum_{k \neq l} \| \yy_d^k \|_{2}^2
< \|y_d\|_Y^2
\end{equation}
which establishes \ref{fa:II} as desired. If, on the other hand, $L$ is bigger than one, 
then by defining $A_1$ and $b_1$ as in \eqref{eq:randomeq275353} and by setting
\begin{equation*}
\begin{aligned}
&A_2:=
\begin{pmatrix}
1 , -1, 0_{1 \times (w_1 - 2)}
\\
0_{(w_2-1) \times w_1}
\end{pmatrix},
&&b_2 :=
\begin{pmatrix}
-1/2
\\
...
\\
-1/2
\end{pmatrix} \in \R^{w_2},
\\
&A_i:= 
\begin{pmatrix}
1 , 0_{1 \times (w_{i-1} - 1)}
\\
0_{(w_i-1) \times w_{i-1}}
\end{pmatrix},
\quad
&&b_i :=
\begin{pmatrix}
-1/2
\\
...
\\
-1/2
\end{pmatrix} \in \R^{w_i}
\qquad i=3,..., L,
\\
&A_{L+1}:= 
\begin{pmatrix}
\yy_d^{l}, 0_{d_\yy \times (w_L - 1)}
\end{pmatrix},
\quad
&&b_{L+1} := 0_{d_\yy},
\end{aligned}
\end{equation*}
we obtain a parameter $\bar \alpha := (A_{L+1}, b_{L+1},..., A_1, b_1) \in D$ with
\begin{equation*}
\left ( \varphi_{i}^{A_{i}, b_{i}} \circ ... \circ \varphi_{1}^{A_{1}, b_{1}} \right )(\xx_{\;d}^k)
=
\begin{cases}
0_{w_i} & \text{ if } k \neq l
\\
\begin{pmatrix}
1
\\
0_{w_i - 1}
\end{pmatrix}
&\text{ if } k = l
\end{cases}
\qquad \forall i=2,...,L
\end{equation*}
and, analogously to the case $L=1$,
\begin{equation*}
\psi(\bar \alpha, \xx_{\;d}^k) 
= \left ( \varphi_{L+1}^{A_{L+1}, b_{L+1}} \circ ... \circ \varphi_{1}^{A_{1}, b_{1}} \right )(\xx_{\;d}^k)
=
\begin{cases}
0_{d_\yy}  & \text{ if } k \neq l 
\\
\yy_d^l   & \text{ if } k = l.
\end{cases}
\end{equation*}
Using the same calculation as in \eqref{eq:randomeq162525}, 
\ref{fa:II}  now follows immediately. To finally see that
$\Theta(\Psi, x_d) \leq 1 - 1/n$ holds, it suffices to note that
every $y_d \in Y$ with $\|y_d\|_Y = 1$ has to possess at least one component with $\|\yy_d^k\|_2 \geq \sqrt{2}$
and to use the same arguments as in \cref{lemma:varableknotassumptions}. This completes
the proof. 
\end{proof}

We would like to mention that the estimate $\Theta(\Psi, x_d) \leq 1 - 1/n$ established above
is again very pessimistic. 
We leave the derivation of better bounds 
for the number $\Theta(\Psi, x_d)$ in \eqref{eq:defTheta} for future research.  
Further, we would like to point out that, in the (essentially) binary case studied in \cref{lemma:heavisideapprox},
the task of solving a squared-loss problem of the form \eqref{eq:trainingpropprot5}
is closely related to classical mixed integer programming. For further details on this topic, 
we refer to \cite{Kurtz2020}.
By combining \cref{lemma:reformulatedcondition,lemma:heavisideapprox},
we arrive at:

\begin{theorem}
{\hspace{-0.05cm}\bf(Approximation Property \ref{fa:II} for General Neural Networks)}\label{theorem:generalactiv}
Consider the situation in \cref{set:NNN} and suppose
that the index set $\{1,...,L\}$ can be split into two (possibly empty) disjoint subsets 
$I$ and $J$ such that the following is true:
\begin{enumerate}[label=\roman*)]
\item
\label{item:generalactiv:i}
 For each $i \in I$, the activation function 
$\sigma_i\colon \R \to \R$ is continuous, the limits 
\begin{equation*}
\sigma_i(-\infty) := \lim_{s \to - \infty} \sigma_i(s)
\qquad
\text{and}
\qquad
\sigma_i(\infty) := \lim_{s \to \infty} \sigma_i(s)
\end{equation*}
exist in $\R$, and it holds $\sigma_i(-\infty) \neq \sigma_i(\infty)$. 

\item 
\label{item:generalactiv:ii}
For each $i \in J$, 
the function $\tilde \sigma_i(s) := \sigma_i(s) - \sigma_i(s - 1)$ satisfies the conditions in \ref{item:generalactiv:i}
and it holds $w_i \geq 2$ for all $i \in J$. 

\item The width $w_1 \in \mathbb{N}$ of the lowest hidden layer
satisfies $w_1 \geq 2$ in the case $1 \in I$ and $w_1 \geq 4$ in the case $1 \in J$.
\end{enumerate}
Then, the map $\Psi(\cdot, x_d)\colon D \to Y$ associated with 
$x_d$ and the neural network $\psi$ in \eqref{eq:NNdef} with the activation functions 
$\sigma_i$, $i=1,...,L$, the depth $L$, and the widths $w_i$, $i=1,...,L$, satisfies condition \ref{fa:II}
and the number $\Theta(\Psi, x_d)$ in \eqref{eq:defTheta} is at most $1 - 1/n$.
\end{theorem}

\begin{proof}
We first consider the case $I = \{1,...,L\}$ and $J = \emptyset$. 
In this situation, it follows from the properties of the 
activation functions $\sigma_i$, $i=1,...,L$,
that
\begin{equation}
\label{eq:randomeq1273645}
\frac{\sigma_i(\gamma s) - \sigma_i(-\infty)}{\sigma_i(\infty) - \sigma_i(-\infty)}
\to 
\bar \sigma_i (s)
:=
\begin{cases}
0& \text{if } s < 0
\\
\displaystyle
\frac{\sigma_i(0) - \sigma_i(-\infty)}{\sigma_i(\infty) - \sigma_i(-\infty)} & \text{if } s=0
\\
1 & \text{if } s > 0
\end{cases}
\end{equation}
holds for all $s \in \R$ and all $i=1,...,L$ for $0 < \gamma \to \infty$
and, as a consequence, that
\begin{equation*}
\frac{1}{\sigma_i(\infty) - \sigma_i(-\infty)}
\left [ \varphi_i^{\gamma A_i, \gamma b_i}(z) 
- \sigma_i(-\infty) 1_{w_i} \right ]
\to
\bar \sigma_i( A_i z + b_i) =: \bar \varphi_i^{ A_i, b_i}(z)
\end{equation*}
holds for all $z \in \R^{w_{i-1}}$ and all $i=1,...,L$ for $0 < \gamma \to \infty$,
where  $1_{w_i}$ denotes the (column) vector in $\R^{w_i}$ that contains the entry one in each component. 
Due to the definitions of the map $\Psi(\cdot, x_d)\colon D \to Y$ and $D$, we further know that,
for all $\alpha = (A_{L+1}, b_{L+1},..., A_1, b_1)\in D$
and all $\gamma > 0$, we have 
\begin{equation*}
\begin{aligned}
&\left \{
\psi
\left 
( A_{L+1}, b_{L+1},..., 
\frac{A_2}{\sigma_1(\infty) - \sigma_1(-\infty)} , b_2 - \frac{\sigma_1(-\infty) A_2 1_{w_1}}{\sigma_1(\infty) - \sigma_1(-\infty)} , \gamma A_1, \gamma b_1 , \xx_{\;d}^k
\right ) \right \}_{k=1}^n
\\
&=
\Bigg \{
\Big ( \varphi_{L+1}^{A_{L+1}, b_{L+1}} \circ ... \circ \varphi_{3}^{A_{3}, b_{3}} \circ \sigma_2\Big)
\\
&\qquad \Bigg (
A_2\left (
\frac{1}{\sigma_1(\infty) - \sigma_1(-\infty)}
\left [
\varphi_1^{\gamma A_1, \gamma b_1}(\xx_{\;d}^k) 
-  \sigma_1(-\infty) 1_{w_1}
\right ]
\right )
+ b_2
\Bigg )
 \Bigg \}_{k=1}^n
\in \closure_Y\left (\Psi(D, x_d)\right ). 
\end{aligned}
\end{equation*}
(Here, in the borderline case $L=1$, the ``empty''
composition  $\varphi_{L+1}^{A_{L+1}, b_{L+1}} \circ ... \circ \varphi_{3}^{A_{3}, b_{3}} \circ \sigma_2$
has to be interpreted as the identity map.)
Combining the last two results, exploiting that
the functions $\sigma_i$ are continuous, and 
using that the set $ \closure_Y\left (\Psi(D, x_d)\right )$
is closed yields that 
\begin{equation*}
\left \{
\left ( \varphi_{L+1}^{A_{L+1}, b_{L+1}} \circ ... \circ \varphi_2^{A_2, b_2} \circ \bar\varphi_{1}^{A_{1}, b_{1}} \right )(\xx_{\;d}^k)
 \right \}_{k=1}^n \in \closure_Y\left (\Psi(D, x_d)\right )
\end{equation*}
holds for all $\alpha = (A_{L+1}, b_{L+1},..., A_1, b_1)\in D$. 
Note that, here and in what follows, 
with $\Psi(\cdot, x_d)\colon D \to Y$ we still mean the function \eqref{eq:PsiDef} 
associated with the original neural network $\psi$
involving the activation functions $\sigma_i$, $i=1,...,L$. 
By proceeding along exactly the same lines as above for the layers
$i=2,...,L$ (in that order), 
we obtain that, for all $\alpha = (A_{L+1}, b_{L+1},..., A_1, b_1)\in D$,
we have 
\begin{equation*}
\left \{
\left ( \varphi_{L+1}^{A_{L+1}, b_{L+1}} \circ \bar \varphi_{L}^{A_{L}, b_{L}}  \circ ... \circ \bar \varphi_{1}^{A_{1}, b_{1}} \right )(\xx_{\;d}^k)
 \right \}_{k=1}^n \in \closure_Y\left (\Psi(D, x_d)\right ).
\end{equation*}
As  the map $\varphi_{L+1}^{A_{L+1}, b_{L+1}}$ does not depend on any activation function, 
this shows that the set $\closure_Y\left (\Psi(D, x_d)\right )$ associated with the 
neural network $\psi$ involving the activation functions $\sigma_i$, $i=1,...,L$,
is at least as big as the set $\closure_Y\left (\bar \Psi(D, x_d)\right )$
that is associated with the neural network $\bar \psi$
that has the same depth $L$ and widths $w_i$ as $\psi$ and involves the Heaviside-type activation 
functions on the right-hand side of \eqref{eq:randomeq1273645}. 
Since the latter set satisfies \eqref{eq:randomeq172636}
by \cref{lemma:reformulatedcondition,lemma:heavisideapprox} and due to our assumptions on $w_1$, 
it now follows immediately that 
\begin{equation*}
\min_{y \in  \closure_Y\left (\Psi(D, x_d)\right )} \|y - y_d\|_Y^2 
\leq 
\min_{y \in  \closure_Y\left (\bar \Psi(D, x_d)\right )} \|y - y_d\|_Y^2 
< \|y_d\|_Y^2\qquad \forall y_d \in Y\setminus \{0\},
\end{equation*}
and, due to \eqref{eq:defTheta} and again \cref{lemma:heavisideapprox}, 
that $\Theta(\Psi, x_d) \leq \Theta(\bar \Psi, x_d) \leq 1 - 1/n$.
This completes the proof for $I = \{1,...,L\}$ and $J = \emptyset$ (see \cref{lemma:reformulatedcondition}).

Let us now suppose that $J$ is not empty,
and let us first assume that 
an $i \in J$ with $i \geq 2$ and $w_i  = 2\tilde w_i$, $\tilde w_i \in \mathbb{N}$,  is given. Then, for all
parameters of the form 
\begin{equation}
\label{eq:randomeq263535}
A_{i+1}:=
\tilde A_{i+1}
\begin{pmatrix}
\mathrm{Id}_{\tilde w_i \times \tilde w_i}, - \mathrm{Id}_{\tilde w_i \times \tilde w_i}
\end{pmatrix},
\qquad 
A_i := 
\begin{pmatrix}
\tilde A_i
\\
\tilde A_i
\end{pmatrix},
\qquad
b_i := 
\begin{pmatrix}
\tilde b_i
\\
\tilde b_i - 1_{\tilde w_i}
\end{pmatrix},
\end{equation}
with $\tilde A_{i+1} \in \R^{w_{i+1} \times \tilde w_i}$, $\tilde A_i \in \R^{\tilde w_i \times w_{i-1}}$,
and $\tilde b_i \in \R^{\tilde w_i}$, we have
\begin{equation*}
A_{i+1}\sigma_i(A_i z + b_i) = \tilde A_{i+1} \tilde \sigma_i(\tilde A_i z + \tilde b_i)\qquad \forall z \in \R^{w_{i-1}}.
\end{equation*}
Here, $ \mathrm{Id}_{\tilde w_i \times \tilde w_i} \in \R^{\tilde w_i \times \tilde w_i}$ is the identity matrix,
$1_{\tilde w_i}$ again denotes the 
vector in $\R^{\tilde w_i}$ that contains the entry one in every component, 
and $\tilde \sigma_i$ is defined as in point \ref{item:generalactiv:ii} 
of the theorem, i.e., 
$\tilde \sigma_i(s) := \sigma_i(s) - \sigma_i(s - 1)$ for all $s \in \R$. 
In combination with \eqref{eq:NNdef}
and our assumptions on the activation functions $\sigma_i$, $i \in J$, the above shows that every layer 
of $\psi$, that is associated with an index 
$2 \leq i \in J$ and possesses an even width,
can emulate a neural network layer of a smaller width that involves an activation function
of the type studied in point \ref{item:generalactiv:i} of the theorem. 
Note that, in the case $i \in J$ with $i \geq 2$ and $w_i  = 2\tilde w_i + 1$, $\tilde w_i \in \mathbb{N}$,
and in the case $i = 1 \in J$, we can proceed completely analogously to the above 
by adding suitable rows/columns of zeros in \eqref{eq:randomeq263535},
and that, in the case $i = 1 \in J$, we can always achieve that $\tilde w_1 \geq 2$
holds by our assumptions on $w_1$.
In summary, we may thus conclude that, for arbitrary $I$ and $J$, we can always find a neural network
of the type \eqref{eq:NNdef}, 
$\tilde \psi$ lets say, 
which satisfies the assumptions of the first part of this proof and 
$\closure_Y  (\tilde \Psi(D, x_d) ) \subset \closure_Y (\Psi(D, x_d) )$,
where $\tilde \Psi$ and $\Psi$ are the functions in \eqref{eq:PsiDef} associated 
with $\tilde \psi$ and the original network $\psi$, respectively. 
The claim of the theorem now follows immediately from \cref{lemma:reformulatedcondition}
and the definition of the number $\Theta(\Psi, x_d)$ (cf.\ the first part of this proof). 
\end{proof}

Note that the above proof shows that a network $\psi$
with the properties in \cref{theorem:generalactiv} is always at least as expressive 
as a network $\bar \psi$ that involves the Heaviside-type activation
functions on the right-hand side of \eqref{eq:randomeq1273645}
and possesses the widths $w_i$ for all $i \in I$ and $\left \lfloor{w_i/2}\right \rfloor$ for all $i \in J$. 
(Here, with ``at least as expressive'', we mean that the inclusion
$\closure_Y (\bar \Psi(D, x_d)) \subset \closure_Y\left (\Psi(D, x_d)\right )$
holds for all data vectors $x_d$ that satisfy the conditions in \cref{set:NNN}.) 
As an immediate consequence of \cref{theorem:generalactiv}, we obtain \citep[cf.][Chapter 2]{Calin2020}:
\begin{corollary}
{\hspace{-0.05cm}\bf(Approximation Property \ref{fa:II} for Popular Activation Functions)}\label[corollary]{cor:approxpropacti}
Consider the situation in \cref{set:NNN} and suppose
that the set $\{1,...,L\}$ can be split into two (possibly empty) disjoint subsets 
$I$ and $J$ such that the following is true:
\begin{enumerate}
\item For each $i \in I$, the activation function $\sigma_i$ is of one of the following types:
\begin{enumerate}[label=\roman*)]
\item $\sigma_{\textup{sig}}(s) := 1/ (1 + \mathrm{e}^{-s})$ (sigmoid/logistic/soft-step activation),
\item  $\sigma_{\tanh}(s) := \tanh(s)$ (tanh-activation),
\item  $\sigma_{\arctan}(s) := \arctan(s)$ (arctan-activation),
\item $\sigma_{\textup{es}}(s) := s / (1 + |s|)$ (soft-sign/Elliot-sig activation),
\item $\sigma_{\textup{isru}}(s) := s/(1 + c s^2)^{1/2}$ with some $c > 0$ (inverse square root unit),
\item $\sigma_{\textup{sc}}(s) := c^{-1} \log \left ( (1+\mathrm{e}^{cs})/(1 + \mathrm{e}^{c(s - 1)}) \right ) $ with some $c>0$ 
(soft-clip activation),
\Item \begin{equation*}
\sigma_{\textup{sqnl}}(s) :=
\begin{cases}
-1 & \text{ if } s \leq -2
\\
s + s^2 /4& \text{ if } -2 < s \leq 0
\\
s - s^2 /4& \text{ if } 0 < s \leq 2
\\
1 & \text{ if } s > 2
\end{cases}
\qquad \text{(SQNL-activation)}.
\end{equation*}
\end{enumerate}
\item\label{item:activationex} 
It holds $w_i \geq 2$ for all $i \in J$, and, for every $i \in J$, $\sigma_i$ is of one of the following types:
\begin{enumerate}[label=\roman*)]
\item $\sigma_{\textup{relu}}(s) := \max(0, s)$ (rectified linear unit), 
\item $\sigma_{\textup{prelu}}(s) := \max(0, s) + \min(0, cs)$, $|c| \neq 1$ (leaky/parametric ReLU), 
\item $\sigma_{\textup{soft+}}(s) :=\ln(1 + \mathrm{e}^s)$  (soft-plus activation), 
\item $\sigma_{\textup{bentid}}(s) := \frac12 (s^2 + 1)^{1/2} - \frac12 + s$  (bent-identity activation), 
\item $\sigma_{\textup{silu}}(s) := s/(1 + \mathrm{e}^{-s})$  (sigmoid linear unit a.k.a.\ swish-1), 
\Item \begin{equation*}
\sigma_{\textup{isrlu}}(s) :=
\begin{cases}
s/ (1 + c s^2)^{1/2}& \text{ if } s < 0
\\
s & \text{ if } s \geq 0
\end{cases},\quad c > 0
\qquad \text{(ISRL-unit)},
\end{equation*}
\Item   
\begin{equation*}
\sigma_{\textup{elu}}(s) :=
\begin{cases}
c \left (\mathrm{e}^{s} - 1 \right )& \text{ if } s < 0
\\
s & \text{ if } s \geq 0
\end{cases},\quad c \in \R
\qquad \text{(exponential linear unit)}.
\end{equation*}
\end{enumerate}
\item It holds $w_1\geq 2$ in the case $1 \in I$ and $w_1 \geq 4$ in the case $1 \in J$. 
\end{enumerate}
Then, the function $\Psi(\cdot, x_d)\colon D \to Y$ associated with 
$x_d$ and the neural network $\psi$ in \eqref{eq:NNdef} 
possesses the property \ref{fa:II}
and the number $\Theta(\Psi, x_d)$ in \eqref{eq:defTheta} is at most $1 - 1/n$.
\end{corollary}

\begin{proof}
This follows immediately by checking the assumptions of \cref{theorem:generalactiv}.
\end{proof}

We remark that the list of activation functions in \cref{cor:approxpropacti} 
is far from exhaustive and that, even when the assumptions 
of \cref{theorem:generalactiv} are not satisfied, 
it is still often possible to establish \ref{fa:II} by hand.
Compare, e.g., with the calculation in \cref{sec:4} in this context, where 
we have done precisely that. 
Having checked that the properties \ref{fa:I} and \ref{fa:II} hold
under reasonable assumptions on the activation functions and widths in \eqref{eq:NNdef}, 
we can now again apply the abstract results of \cref{sec:5}. 
Before we collect the numerous corollaries that we obtain in this way, 
we prove two lemmas that simplify the application of \cref{theorem:badcone}.

\begin{lemma}
\label[lemma]{lemma:affinesubspace}
Consider the situation in \cref{set:NNN}. 
Suppose further that 
$\min(d_\xx, d_\yy)\leq \min(w_1,...,w_L)$ holds and that, for each $i \in \{1,...,L\}$, 
there exists an open nonempty interval $I_i \subset \R$ such that 
$\sigma_i$ is affine with a non-vanishing derivative on $I_i$. 
Define%
\begin{equation*}
V := 
\left \{
\{\zz_k\}_{k=1}^n \in Y
\,
\Big |
\,
\exists A \in \R^{d_\yy \times d_\xx}, b \in \R^{d_\yy} \text{ such that }  \zz_k = A \xx_{\;d}^k + b~~\forall k=1,...,n
\right \}.
\end{equation*}
Then, for every element $z$ of the subspace $V$, there exist an $\bar \alpha \in D$ and an open set $U \subset D$
satisfying $\bar \alpha \in U$, $z = \Psi(\bar \alpha, x_d)$, and $\Psi(U, x_d) \subset V$. 
\end{lemma}

\begin{proof}
Suppose that an arbitrary but fixed element $z$ of the subspace $V$ with 
associated $A \in \R^{d_\yy \times d_\xx}$ and $b \in \R^{d_\yy}$ is given
and
let $a_i \in \R$, $\varepsilon_i > 0$, $\beta_i \in \R \setminus \{0\}$, and $\gamma_i \in \R$
satisfy $I_i = (a_i - \varepsilon_i, a_i + \varepsilon_i)$
and $\sigma_i(s) = \beta_i s + \gamma_i$ for all $s \in I_i$
and all $i=1,...,L$. 
Let us further again use the symbols $0_{p \times q}$ and $0_p$ to denote the zero matrix in $\R^{p \times q}$
and the zero (column) vector in $\R^p$, respectively, with the convention that these blocks are ignored in 
the cases $p=0$ and $q=0$, let $\mathrm{Id}_{p \times p}$ and $1_p$ denote the identity matrix in $\R^{p \times p}$
and the vector in $\R^p$ whose entries are all equal to one, respectively, 
and let $\|\cdot\|_\infty$ be the maximum norm on the Euclidean space. 
Then, in the case $d_\xx \leq \min(w_1,...,w_L)$,
 it is easy to check  that the matrices and vectors
\begin{equation*}
\begin{gathered}
A_1 := 
\frac{\varepsilon_1}{2\max_{k=1,...,n} \|\xx_{\;d}^k\|_\infty}
\begin{pmatrix}
\mathrm{Id}_{d_\xx \times d_\xx} 
\\
0_{(w_1 - d_\xx) \times d_\xx}
\end{pmatrix}
\in \R^{w_1 \times w_{0}},
\qquad 
b_1 := 
a_1 1_{w_1}
\in \R^{w_1},
\\
A_i := 
\frac{\varepsilon_i}{\beta_{i-1} \varepsilon_{i-1}}
\begin{pmatrix}
\mathrm{Id}_{d_\xx \times d_\xx}, 0_{d_\xx \times (w_{i - 1} - d_\xx)}
\\
0_{(w_i - d_\xx) \times w_{i - 1}}
\end{pmatrix}
\in \R^{w_i \times w_{i - 1}},\qquad i=2,...,L,
\\
b_i := 
a_i 1_{w_i}
 - (a_{i - 1}\beta_{i - 1} + \gamma_{i - 1})  A_i 1_{w_{i - 1}}
\in \R^{w_i},\qquad i=2,...,L,
\\
A_{L+1}:= 
\frac{2\max_{k=1,...,n} \|\xx_{\;d}^k\|_\infty}{\beta_{L} \varepsilon_{L}}
\begin{pmatrix}
A, 0_{d_\yy \times (w_{L} - d_\xx)}
\end{pmatrix}
\in \R^{w_{L+1} \times w_{L}},
\\
b_{L+1} := 
b - (a_{L}\beta_L + \gamma_{L})  A_{L+1} 1_{w_{L}}
\in \R^{w_{L+1}}
\end{gathered}
\end{equation*}
satisfy 
\begin{equation}
\label{eq:randomeq273535}
\begin{gathered}
A_1\xx_{\;d}^j + b_1 
=
\frac{\varepsilon_1}{2\max_{k=1,...,n} \|\xx_{\;d}^k\|_\infty}
\begin{pmatrix}
\xx_{\;d}^j
\\
0_{w_{1} - d_\xx}
\end{pmatrix}
+
a_1 1_{w_1}
\in \left (a_1 -  \varepsilon_1,  a_1 +  \varepsilon_1  \right )^{w_1},
\\
A_{i}\left (
\varphi_{i-1}^{A_{i-1}, b_{i-1}} \circ ... \circ \varphi_{1}^{A_{1}, b_{1}}(\xx_{\;d}^j)
\right )
+
b_i
=
\frac{\varepsilon_i}{2\max_{k=1,...,n} \|\xx_{\;d}^k\|_\infty}
\begin{pmatrix}
\xx_{\;d}^j
\\
0_{w_{i} - d_\xx}
\end{pmatrix}
+
a_i 1_{w_i}
\\
\qquad\qquad\qquad\qquad\qquad\qquad\qquad\qquad\qquad 
\in \left (a_i - \ \varepsilon_i,  a_i + \varepsilon_i  \right )^{w_i}
\quad \forall i=2,...,L,
\end{gathered}
\end{equation}
and
\begin{equation*}
\left (\varphi_{L+1}^{A_{L+1}, b_{L+1}} \circ ... \circ \varphi_{1}^{A_{1}, b_{1}}\right )(\xx_{\;d}^j)
=
A \xx_{\;d}^j + b
\end{equation*}
for all $j=1,...,n$. The parameter $\bar \alpha := (A_{L+1}, b_{L+1}, ..., A_1, b_1)$
thus satisfies $\Psi(\bar \alpha, x_d) = z$ as desired. Since 
the inclusions in \eqref{eq:randomeq273535} are stable w.r.t.\ small 
perturbations in the matrices $A_i$ and the vectors $b_i$, we further obtain that
the function $\psi(\alpha, \cdot)\colon \XX \to \YY$ also behaves affine-linearly on the training
data $\xx_{\;d}^k$, $k=1,...,n$, for all $\alpha \in D$ in a small open neighborhood $U$ of $\bar \alpha$.
This shows that there exists an open set $U \subset D$ satisfying 
$\bar \alpha \in U$, $\Psi(\bar \alpha, x_d) = z$, and $\Psi(U, x_d) \subset V$
and proves the claim of the lemma in the case  $d_\xx \leq \min(w_1,...,w_L)$.

In the situation  $d_\yy\leq \min(w_1,...,w_L)$, we can proceed along similar lines as above. 
Given an arbitrary but fixed  $z \in V$ with 
associated $A \in \R^{d_\yy \times d_\xx}$ and $b \in \R^{d_\yy}$, we define
\begin{equation*}
\begin{gathered}
A_1 := 
\frac{\varepsilon_1}{2\max_{k=1,...,n} \|A\xx_{\;d}^k\|_\infty + 1}
\begin{pmatrix}
A
\\
0_{(w_1 - d_\yy) \times d_\xx}
\end{pmatrix}
\in \R^{w_1 \times w_{0}},
\\
b_1 := 
a_1 1_{w_1}
\in \R^{w_1},
\\
A_i := 
\frac{\varepsilon_i}{\beta_{i-1} \varepsilon_{i-1}}
\begin{pmatrix}
\mathrm{Id}_{d_\yy \times d_\yy}, 0_{d_\yy \times (w_{i - 1} - d_\yy)}
\\
0_{(w_i - d_\yy) \times w_{i - 1}}
\end{pmatrix}
\in \R^{w_i \times w_{i - 1}},\qquad i=2,...,L,
\\
b_i := 
a_i 1_{w_i}
 - (a_{i - 1}\beta_{i - 1} + \gamma_{i - 1}) A_i 1_{w_{i - 1}}
\in \R^{w_i},\qquad i=2,...,L,
\\
A_{L+1}:= 
\frac{2\max_{k=1,...,n} \|A\xx_{\;d}^k\|_\infty + 1}{\beta_{L} \varepsilon_{L}}
\begin{pmatrix}
\mathrm{Id}_{d_\yy \times d_\yy}, 0_{d_\yy \times (w_{L} - d_\yy)}
\end{pmatrix}
\in \R^{w_{L+1} \times w_{L}},
\\
b_{L+1} := 
b - (a_{L}\beta_L + \gamma_{L}) A_{L+1} 1_{w_{L}}
\in \R^{w_{L+1}}.
\end{gathered}
\end{equation*}
Then, it is easy to check that it holds
\begin{equation*}
\begin{gathered}
A_1\xx_{\;d}^j + b_1 
=
\frac{\varepsilon_1}{2\max_{k=1,...,n} \|A\xx_{\;d}^k\|_\infty + 1}
\begin{pmatrix}
A\xx_{\;d}^j
\\
0_{w_{1} - d_\yy}
\end{pmatrix}
+
a_1 1_{w_1}
\in \left (a_1 -  \varepsilon_1,  a_1 +  \varepsilon_1  \right )^{w_1},
\\
A_{i}\left (
\varphi_{i-1}^{A_{i-1}, b_{i-1}} \circ ... \circ \varphi_{1}^{A_{1}, b_{1}}(\xx_{\;d}^j)
\right )
+
b_i
=
\frac{\varepsilon_i}{2\max_{k=1,...,n} \|A\xx_{\;d}^k\|_\infty + 1}
\begin{pmatrix}
A\xx_{\;d}^j
\\
0_{w_{i} - d_\yy}
\end{pmatrix}
+
a_i 1_{w_i}
\\
\qquad\qquad\qquad\qquad\qquad\qquad\qquad\qquad\qquad 
\in \left (a_i - \ \varepsilon_i,  a_i + \varepsilon_i  \right )^{w_i}
\quad \forall i=2,...,L,
\end{gathered}
\end{equation*}
and
\begin{equation*}
\left (\varphi_{L+1}^{A_{L+1}, b_{L+1}} \circ ... \circ \varphi_{1}^{A_{1}, b_{1}}\right )(\xx_{\;d}^j)
=
A \xx_{\;d}^j + b
\end{equation*}
for all $j=1,...,n$. The existence of a parameter $\bar \alpha \in D$ and an open set $U \subset D$
with the desired properties now follows completely analogously to the first part of the proof. 
\end{proof}
\begin{lemma}\label[lemma]{lemma:constantsubspace}
Consider the situation in \cref{set:NNN}, 
assume that the functions $\sigma_i$, $i=1,...,L$, are bounded on bounded sets, and
suppose that there exists an index $j \in \{1,...,L\}$
such that the function $\sigma_j$ is constant on an open nonempty interval $I_j \subset \R$.
Define
\begin{equation*}
V := 
\left \{
\{\zz_k\}_{k=1}^n \in Y
\,
\Big |
\,
\zz_k = \zz_l~~\forall k, l \in \{1,...,n\}
\right \}.
\end{equation*}
Then, for every element $z$ of the subspace $V$, there exist an $\bar \alpha \in D$ and an open set $U \subset D$
satisfying $\bar \alpha \in U$, $z = \Psi(\bar \alpha, x_d)$, and $\Psi(U, x_d) \subset V$. 
\end{lemma}

\begin{proof}
Consider an arbitrary but fixed $z = \{\zz_k\}_{k=1}^n \in V$, let 
$b \in \R^{d_\yy}$ denote the unique vector with $\zz_k = b$ for all $k=1,...,n$, 
and let $a_j \in \R$ and $\varepsilon_j > 0$ satisfy $I_j = (a_j - \varepsilon_j, a_j + \varepsilon_j)$.
Then, by defining
\begin{equation*}
\begin{gathered}
A_i := 0 \in \R^{w_i \times w_{i - 1}}\quad \forall i \in \{1,...,L+1\},
\qquad 
b_i := 0 \in \R^{w_i}\quad \forall i \in \{1,...,L\} \setminus \{j\},
\\
b_j := a_j 1_{w_j} \in \R^{w_j},\quad\text{and}\quad b_{L+1} := b \in \R^{w_{L+1}},
\end{gathered}
\end{equation*}
where $1_{w_j}$ again denotes the vector which contains the number one in each of its entries, we obtain a 
parameter $\bar \alpha := (A_{L+1}, b_{L+1}, ..., A_1, b_1)$ which trivially satisfies 
$\psi(\bar \alpha, \xx_{\;d}^k) = b = \zz_k$ for all $k=1,...,n$ and 
\begin{equation}
\label{eq:randomeq163535}
A_{j}\left (\varphi_{j-1}^{A_{j-1}, b_{j-1}} \circ ... \circ \varphi_{1}^{A_{1}, b_{1}}(\xx_{\;d}^k)
\right )
+
b_j
=
b_j
\in \left (a_j - \ \varepsilon_j,  a_j + \varepsilon_j  \right )^{w_j}
\quad \forall k =1,...,n.
\end{equation}
(Here, the expression $\varphi_{j-1}^{A_{j-1}, b_{j-1}} \circ ... \circ \varphi_{1}^{A_{1}, b_{1}}$ 
again has to be interpreted as the identity map in the borderline case $j = 1$.) 
As the inclusion in \eqref{eq:randomeq163535} remains true for all parameters $\alpha$
in a small neighborhood $U \subset D$ of $\bar \alpha$ and since the activation function $\sigma_j$ is constant 
on $I_j$, it now follows immediately that there exist an $\bar \alpha \in D$ and an open set $U \subset D$
with the desired properties $\bar \alpha \in U$, $z = \Psi(\bar \alpha, x_d)$, and $\Psi(U, x_d) \subset V$. This 
completes the proof. 
\end{proof}

Note that the condition $\min(d_\xx, d_\yy)\leq \min(w_1,...,w_L)$ in \cref{lemma:affinesubspace} is vacuous 
in the case $d_\yy = 1$, i.e., in the situation where the neural network has a scalar output. 
With \cref{lemma:affinesubspace,lemma:constantsubspace} in place,
we are in the position to study the consequences that the analysis of \cref{sec:5}
has for the optimization landscape and the stability properties 
of training problems of the form \eqref{eq:trainingpropprot5}. 
Following the structure of \cref{sec:5}, we begin with a result on the 
set-valuedness and the stability
of the best approximation map $P_\Psi^{x_d}$ associated with \eqref{eq:trainingpropprot5}
in the case where there exist unrealizable vectors: 

\begin{corollary}
{\hspace{-0.05cm}\bf(Nonuniqueness and Instability of Best Approximations)}\label[corollary]{cor:NNbestapproxinstabil}
Consider the situation in \cref{set:NNN} and suppose that the 
widths $w_i$ and the activation functions $\sigma_i$
are such that
\cref{lemma:heavisideapprox}, \cref{theorem:generalactiv}, or \cref{cor:approxpropacti}
can be applied to~$\psi$. 
Assume further that there exist unrealizable vectors $y_d \in Y$, i.e., that $\closure_Y(\Psi(D, x_d)) \neq Y$ holds. 
Then, the best approximation map 
\begin{equation}
\label{eq:bestapproxmaprevNN}
P_{\Psi}^{x_d}\colon Y \rightrightarrows Y,\qquad y_d \mapsto \argmin_{y \in  \closure_Y\left (\Psi(D, x_d)\right )} \|y - y_d\|_Y^2,
\end{equation}
associated with the training problem \eqref{eq:trainingpropprot5}
is set-valued and there exist uncountably many training label vectors $y_d \in Y$ such that
$|P_{\Psi}^{x_d}(y_d)| > 1$ holds. Moreover, the map $P_{\Psi}^{x_d}\colon Y \rightrightarrows Y$ 
is discontinuous in the sense that, for every arbitrary but fixed $C > 0$, there exists an uncountable set $\MM_C \subset Y$
such that, for every $y_d \in \MM_C$, there are sequences $\{y_d^l\}, \{\tilde y_d^l\} \subset Y$ satisfying 
\begin{equation*}
\begin{gathered}
y_d^l \to y_d  \text{ for }l \to \infty,
\qquad
\tilde y_d^l \to y_d \text{ for }l \to \infty,
\\
|P_\Psi^{x_d}(y_d^l)| =  |P_\Psi^{x_d}(\tilde y_d^l)|  = 1 \quad \forall l,
\quad \text{and}\quad
\|P_\Psi^{x_d}(y_d^l) - P_\Psi^{x_d}(\tilde y_d^l)\|_Y \geq C\quad \forall l.
\end{gathered}
\end{equation*}
Further, for every $C > 0$, there exists at least one $y_d \in \MM_C$ satisfying \eqref{eq:thetabound61253}
(with the number $\Theta(\Psi, x_d) \in (0, 1)$ associated with $\psi$ and $x_d$ defined in Equation \ref{eq:defTheta}).
\end{corollary}
\begin{proof}
As \cref{lemma:heavisideapprox}, \cref{theorem:generalactiv}, or \cref{cor:approxpropacti} can be applied to $\psi$
by assumption and due to \cref{lemma:nnconic}, we know that the 
map $\Psi(\cdot, x_d)\colon D \to Y$ associated with $x_d$ 
and  $\psi$ satisfies \ref{fa:I} and \ref{fa:II}. By invoking \cref{theorem:abstractinstability},
the claim now follows immediately.
\end{proof}

Note that, in the situation of \cref{cor:NNbestapproxinstabil}, the comments made in 
\cref{rem:stability} (e.g., on the potential consequences for the convergence behavior of 
descent methods) apply to the training problem \eqref{eq:trainingpropprot5} as well. 
As a corollary of \cref{th:stevaluedspurious}, we next obtain: 

\begin{corollary}
{\hspace{-0.05cm}\bf(Excessive Nonuniqueness or Spurious Minima/Basins)}\label[corollary]{cor:NNinfiniteNonUniqueness}
Consider the situation in \cref{set:NNN} and suppose that the 
widths $w_i$ and the activation functions $\sigma_i$
satisfy the conditions in \cref{theorem:generalactiv} (or \cref{cor:approxpropacti}, respectively). 
Assume further that the functions $\sigma_i$ are continuous and
that there exist unrealizable label vectors $y_d \in Y$, i.e., that $\closure_Y(\Psi(D, x_d)) \neq Y$ holds. 
Then, there exist uncountably many $y_d \in Y$ such that the best approximation map $P_{\Psi}^{x_d}\colon Y \rightrightarrows Y$
defined in \eqref{eq:bestapproxmaprevNN} satisfies $|P_\Psi^{x_d}(y_d)| = \infty$ or there exist
an open nonempty cone $K \subset Y$ and a number $M \in \mathbb{N}$ with $M \geq 2$
such that, for every $y_d \in K$, there are nonempty disjoint closed subsets $D_1,..., D_M$ 
of the parameter space $D$ in \eqref{eq:randomeq2635hdg36} satisfying 
\begin{equation}
\label{eq:randomeq283636}
\inf_{\alpha \in D_1} \|\Psi( \alpha, x_d) - y_d\|_Y^2
< \inf_{\alpha \in D_i} \|\Psi( \alpha, x_d) - y_d\|_Y^2\quad \forall i=2,...,M
\end{equation}
and
\begin{equation}
\label{eq:randomeq283636-2}
\sup_{\alpha \in D_1 \cup ... \cup D_M} \|\Psi( \alpha, x_d) - y_d\|_Y^2 < \|\Psi( \tilde \alpha, x_d) - y_d\|_Y^2 
\quad \forall \tilde \alpha \in D \setminus (D_1 \cup ... \cup D_M).
\end{equation}
Further, if the second of the above cases applies, then the spurious local minima/basins in \eqref{eq:randomeq283636} 
can be arbitrarily bad in the sense that, for every arbitrary but fixed $C>0$,
there exist uncountably many $y_d \in K$ which not only satisfy \eqref{eq:randomeq283636}  but even
\begin{equation*}
\inf_{\alpha \in D_1} \|\Psi( \alpha, x_d) - y_d\|_Y^2 + C
< \inf_{\alpha \in D_i} \|\Psi( \alpha, x_d) - y_d\|_Y^2\quad \forall i=2,...,M.
\end{equation*}
\end{corollary}

\begin{proof}
The first half of the corollary follows immediately from
\cref{lemma:nnconic}, \cref{theorem:generalactiv}, \cref{cor:approxpropacti}, and \cref{th:stevaluedspurious}.
To see that the spurious minima/basins in \eqref{eq:randomeq283636} can get arbitrarily bad, it suffices to note that,
for every $\gamma > 0$ and every $y_d \in K$ that satisfies \eqref{eq:randomeq283636} and \eqref{eq:randomeq283636-2}
for some $M \in \mathbb{N}$ and $D_1,...,D_M$, 
the vector $\widehat y_d := \gamma y_d$ and the (trivially nonempty, disjoint, and closed) sets 
\begin{equation*}
\widehat D_i := 
\left \{\widehat \alpha = (\gamma A_{L+1}, \gamma b_{L+1}, A_L, b_L, ..., A_1, b_1)\, \Big | \, 
 \left ( A_{L+1}, b_{L+1}, ..., A_1, b_1\right ) \in D_i \right \},~i=1,...,M,
\end{equation*}
satisfy 
\begin{equation*}
\begin{aligned}
&\inf_{\widehat \alpha \in \widehat D_1} \|\Psi( \widehat \alpha, x_d) - \widehat y_d\|_Y^2
- \inf_{\widehat \alpha \in \widehat  D_i} \|\Psi( \widehat \alpha, x_d) - \widehat y_d\|_Y^2
\\
&\qquad= \gamma^2 \left ( 
\inf_{\alpha \in D_1} \|\Psi( \alpha, x_d) -  y_d\|_Y^2
- \inf_{\alpha \in  D_i} \|\Psi( \alpha, x_d) -   y_d\|_Y^2
\right ) 
\end{aligned}
\end{equation*}
for all $i=1,...,M$ and 
\begin{equation*}
\sup_{\widehat \alpha \in \widehat D_1 \cup ... \cup \widehat D_M} \|\Psi(\widehat \alpha, x_d) - \widehat y_d\|_Y^2 < 
\|\Psi( \tilde \alpha, x_d) - \widehat y_d\|_Y^2 
\quad \forall \tilde \alpha \in D \setminus (\widehat D_1 \cup ... \cup \widehat D_M)
\end{equation*}
by the architecture of $\psi$. This completes the proof.
\end{proof}\pagebreak

As already mentioned in \cref{sec:5}, the last result complements the findings of \cite{Venturi2019}
on the existence 
of spurious valleys in training problems for one-hidden-layer neural networks with 
non-polynomial non-negative activation functions
in the sense that it shows that the existence of such valleys 
also is to be expected in the multi-layer case provided that there are unrealizable label vectors, 
cf.\ the remarks after \cref{th:stevaluedspurious}.
For a geometric interpretation of the conditions in \eqref{eq:randomeq283636} 
and \eqref{eq:randomeq283636-2}, we refer to \cref{fig:spuriousvalleyillustration}.

Before we proceed, we briefly comment in more detail on
the unrealizability assumption $\closure_Y(\Psi(D, x_d)) \neq Y$ in \cref{cor:NNbestapproxinstabil,cor:NNinfiniteNonUniqueness}.
Checking for which choices of $n$ and $x_d$ this condition is satisfied 
when a neural network $\psi \colon D \times \XX \to \YY$ with an arbitrary but fixed 
architecture is considered is often hard as it requires in-depth knowledge about 
the approximation capabilities of the network. Typically, it is only possible to 
say that $\closure_Y(\Psi(D, x_d)) \neq Y$ holds for all training problems \eqref{eq:trainingpropprot5} 
that involve a sufficiently high number of training samples $n$
and that $\closure_Y(\Psi(D, x_d)) = Y$ holds for all problems \eqref{eq:trainingpropprot5}  
in which $n$ is sufficiently low. Indeed, if we define 
\begin{equation*}
\NN_\psi := \left \{
n \in \mathbb{N}
\,\Big|\, \closure_Y(\Psi(D, x_d)) = Y
 \text{ holds for all } x_d = \{\xx_{\;d}^k\}_{k=1}^n \text{ with } \xx_{\;d}^j \neq \xx_{\;d}^k
\text{ for } j \neq k\right \},
\end{equation*}
then this set clearly contains the number $n=1$ since, for 
every $x_d = \{\xx_{\;d}^1\}$ and every $y_d = \{\yy_d^1\}$,
we can trivially find a parameter $\alpha \in D$ such that $\psi(\alpha, \xx_{\;d}^1) = \yy_d^1$ holds
by choosing the bias $b_{L+1}$ of the last network layer appropriately. 
(Although we have excluded the degenerate case $n=1$ in \cref{set:NNN},
we temporarily allow it here to simplify the discussion.)
On the other hand, it is also obvious that $n-1 \in \NN_\psi$ holds for all $2 \leq n \in \NN_\psi$.
By combining these two observations, we obtain that there exists a number $\bar n_\psi \in \mathbb{N} \cup \{\infty\}$,
which depends non-trivially on the choice of activation functions $\sigma_i$, the widths $w_i$, and the depth $L$ of $\psi$,
such that $\NN_\psi  =\{n \in \mathbb{N} \mid n \leq \bar n_\psi\}$ holds. 
For every network $\psi$, there thus is an architecture-dependent threshold $\bar n_\psi$
such that, for all training problems \eqref{eq:trainingpropprot5} with $n \leq \bar n_\psi$ training samples,
\cref{cor:NNbestapproxinstabil,cor:NNinfiniteNonUniqueness} are guaranteed to be inapplicable and such that,
for all training problems \eqref{eq:trainingpropprot5} with $n > \bar n_\psi$ training samples, 
there exist cases in which \cref{cor:NNbestapproxinstabil,cor:NNinfiniteNonUniqueness} 
apply (cf.\ \cref{lemma:varknotaux}). Note that, as the properties in \cref{cor:NNbestapproxinstabil,cor:NNinfiniteNonUniqueness}  are highly undesirable, 
these considerations indicate that it is very beneficial to train neural networks in a regime in which 
$n \leq \bar n_\psi$ holds, i.e., in which the 
realizability condition $\closure_Y(\Psi(D, x_d)) = Y$ is guaranteed to hold. We remark that 
this observation is also in very good accordance with the findings on the absence of spurious valleys 
of \cite{Nguyen2018,LiDawei2021}. In fact, in both of these papers, 
the used assumptions on the network architecture immediately yield
that the number of training samples satisfies $n \leq \bar n_\psi$, see 
\cite[Theorem 3.4]{Nguyen2018} and \cite[proof of Theorem~4]{LiDawei2021}. 
The analysis of \cite{LiDawei2021} further provides a lower bound for the number $\bar n_\psi$
for networks with continuous activation functions, namely, $\bar n_\psi \geq w_{L}$. 
Our results complement the findings of \cite{Nguyen2018} and \cite{LiDawei2021} in the sense that they show that,
in the training regime $n > \bar n_\psi$, the absence of spurious valleys 
in the optimization landscape cannot be expected.

Next, we turn our attention to results that do not require the unrealizability condition
$\closure_Y(\Psi(D, x_d)) \neq Y$ but exploit the 
approximation property \ref{fa:II} directly, cf.\ \cref{subsec:5.2}. 
Note that the following corollaries offer the advantage that they 
can be invoked without checking whether there exist unrealizable label vectors.
We begin with an observation on non-optimal stationary points:\pagebreak

\begin{corollary}{\hspace{-0.05cm}\bf(Non-Optimal Stationary Points in the Case $\mathbf{m < nd_\yy}$)}\label[corollary]{cor:saddlenonover}
Consider the situation in \cref{set:NNN} and suppose that the 
widths $w_i$ and the activation functions $\sigma_i$
are such that \cref{lemma:heavisideapprox}, \cref{theorem:generalactiv}, or \cref{cor:approxpropacti}
can be applied to~$\psi$. 
Assume further that the number of parameters $m := w_{L+1} (w_{L} + 1) + ... + w_{1}(w_{0} + 1)$
in the neural network $\psi$ is smaller than the product $n d_\yy$. 
Then, for every point $\bar \alpha \in D$ at which the map 
$\Psi(\cdot, x_d)\colon D \to Y$ is differentiable and every arbitrary but fixed $\varepsilon > 0$, 
there exist uncountably many label vectors $y_d \in Y$ satisfying 
\begin{equation}
\label{eq:randomeq1263535-ddd}
\left (\frac{\Theta(\Psi, x_d)}{1 - \Theta(\Psi, x_d)}\right)^{1/2} \|\Psi(\bar \alpha, x_d)\|_Y < 
\|\Psi(\bar \alpha, x_d) - y_d\|_Y < \left (\frac{\Theta(\Psi, x_d)}{1 - \Theta(\Psi, x_d)}\right)^{1/2} \|\Psi(\bar \alpha, x_d)\|_Y + \varepsilon
\end{equation}
such that $\bar \alpha$ is a spurious local minimum or a saddle point of the training problem \eqref{eq:trainingpropprot5}. 
Here, $\Theta(\Psi, x_d) \in [0, 1)$ again denotes the number in \eqref{eq:defTheta}
associated with $\psi$ and $x_d$ that measures the extent to which \ref{fa:II} is satisfied and
the worst-case approximation error in \eqref{eq:randomeq2736352}.
Further,
for every $C>0$,
there exist uncountably many $y_d \in Y$ such that $\bar \alpha$ 
is a spurious local minimum or a saddle point of 
\eqref{eq:trainingpropprot5} and such that \eqref{eq:errorestimates}
and \eqref{eq:lossestimate42} hold.
As a saddle point or spurious local minimum, $\bar \alpha$ can thus be arbitrarily bad. 
\end{corollary}
\begin{proof}
The assertion follows straightforwardly from 
 \cref{lemma:nnconic}, \cref{lemma:heavisideapprox}, \cref{theorem:generalactiv}, \cref{cor:approxpropacti} 
and \cref{theorem:existencestatpts} and by noting that the affine-linearity of the topmost layer of $\psi$ implies that 
$\Psi(\bar \alpha, x_d)$ is an element of the linear hull 
$\span(\partial_1 \Psi(\bar \alpha, x_d),...,\partial_m \Psi(\bar \alpha, x_d))$
for all points of differentiability $\bar \alpha$ of the map $\Psi(\cdot, x_d)\colon D \to Y$. 
\end{proof}

Note that \cref{cor:saddlenonover} provides a strong argument
for the overparameterization of training problems 
of the form \eqref{eq:trainingpropprot5} or, more precisely, for training in the regime $m \geq nd_\yy$. 
This observation also accords well with the observations in 
\cref{cor:NNbestapproxinstabil,cor:NNinfiniteNonUniqueness}.

However, as we have already seen in \cref{sec:5}, 
overparameterization cannot resolve all of the difficulties 
that arise when training problems of the type \eqref{eq:trainingpropprot5} are considered.
This is also illustrated by the following corollary on the existence of 
non-optimal stationary points that also covers the case $n d_\yy \leq m$. 
\begin{corollary}{\hspace{-0.05cm}\bf(Non-Optimal Stationary Points in the Case $\mathbf{d_\xx + 1 < n}$)}\label[corollary]{cor:saddleoverpar}
Consider the situation in \cref{set:NNN} and suppose that the 
widths $w_i$ and the activation functions $\sigma_i$
satisfy the conditions in \cref{theorem:generalactiv} (or \cref{cor:approxpropacti}, respectively). 
Assume further that $d_\xx + 1 < n$ holds and that the activation functions
$\sigma_i$ are differentiable. Then, 
for every $\bar \alpha \in D$ of the form
$\bar \alpha = (A_{L+1}, b_{L+1},..., A_2, b_2, 0,b_1)$
(and thus for all elements of an $(m -  d_\xx w_1)$-dimensional subspace of $D$)
and every arbitrary but fixed $\varepsilon > 0$, 
there exist uncountably many label vectors $y_d \in Y$ satisfying \eqref{eq:randomeq1263535-ddd}
such that $\bar \alpha$ is a spurious local minimum or a saddle point of \eqref{eq:trainingpropprot5}. 
Further, 
for every $\bar \alpha$ of the above type and every arbitrary but fixed $C>0$,
there exist uncountably many $y_d \in Y$ such that $\bar \alpha$ 
is a spurious local minimum or a saddle point of 
\eqref{eq:trainingpropprot5} and such that \eqref{eq:errorestimates}
and \eqref{eq:lossestimate42} hold.
\end{corollary}

\begin{proof}
This follows from 
 \cref{lemma:nnconic}, \cref{theorem:generalactiv}, \cref{cor:approxpropacti},
and \cref{corollary:structstatpts}. 
Note that \cref{ass:linearlowlev} is trivially satisfied in the situation of \cref{set:NNN}
(up to the isomorphism in Equation \ref{eq:randomisomorphism363}) so that 
\cref{corollary:structstatpts} is indeed applicable here. 
\end{proof}\\[-0.65cm]

Next, we consider neural networks 
with activation functions that are affine-linear on some 
open nonempty subset of their domain of definition. 
We begin with two results on the existence of spurious local minima: 
\begin{corollary}{\hspace{-0.05cm}\bf(Spurious Minima for Activations with an Affine Segment)}\label[corollary]{cor:spurminaffineNN}
Consider the situation in \cref{set:NNN} and suppose that the 
widths $w_i$ and the activation functions $\sigma_i$
are such that  \cref{theorem:generalactiv} (or \cref{cor:approxpropacti}) 
can be applied to $\psi$. Assume further that, for every $i \in \{1,...,L\}$, 
there exists an open nonempty interval $I_i \subset \R$ such that 
$\sigma_i$ is affine-linear with a non-vanishing derivative on $I_i$ and that 
the inequalities  $\min(d_\xx, d_\yy)\leq \min(w_1,...,w_L)$ and $n > d_\xx + 1$ hold.
Define
\begin{equation*}
V := 
\left \{
\{\zz_k\}_{k=1}^n \in Y
\,
\Big |
\,
\exists A \in \R^{d_\yy \times d_\xx}, b \in \R^{d_\yy} \text{ such that }  \zz_k = A \xx_{\;d}^k + b~~\forall k=1,...,n
\right \}
\end{equation*}
and denote the $(\cdot, \cdot)_Y$-orthogonal complement of the space $V$ in $Y$ with $V^\perp$.
Then, the training problem \eqref{eq:trainingpropprot5} possesses at least one spurious local minimum
satisfying a growth condition of the form \eqref{eq:quadgrowthstates}
for all label vectors $y_d \in Y$ that are elements of the open cone
\begin{equation}
\label{eq:Kspurdef-2625sf35vs}
K :=
\left \{
y_d^1 + y_d^2 \in Y
\,
\Bigg |
\,
y_d^1 \in V,\, y_d^2 \in V^\perp,\, \|y_d^2\|_Y > \left (\frac{\Theta(\Psi, x_d)}{1 - \Theta(\Psi, x_d)}\right)^{1/2} \|y_d^1\|_Y
\right \}.
\end{equation}
Here,  $\Theta(\Psi, x_d) \in [0, 1)$ again denotes the number in \eqref{eq:defTheta}
associated with $\psi$ and $x_d$ that measures the extent to which \ref{fa:II} is satisfied and
the worst-case approximation error in \eqref{eq:randomeq2736352}.
Further, for every arbitrary but fixed $C>0$, there exist uncountably many $y_d \in K$
such that at least one of the spurious local minima of \eqref{eq:trainingpropprot5} satisfies 
\eqref{eq:errorestimates}, \eqref{eq:lossestimate42}, and \eqref{eq:quadgrowthstates},
and if $\closure_Y(\Psi(D, x_d)) = Y$ holds, then the cone $K$ in \eqref{eq:Kspurdef-2625sf35vs}
is equal to $Y \setminus V$ and \eqref{eq:trainingpropprot5}
possesses spurious local minima for all $y_d$ that are not affine-linearly fittable. 
\end{corollary}
\begin{proof}
To establish the assertions of the corollary, it suffices to combine 
\cref{lemma:nnconic}, \cref{theorem:generalactiv}, and \cref{cor:approxpropacti} 
with \cref{lemma:affinesubspace} and \cref{theorem:badcone}. 
\end{proof}~\\[-1cm]

\begin{corollary}{\hspace{-0.05cm}\bf(Spurious Minima for Activations with a Constant Segment)}\label[corollary]{cor:spurminconstantNN}
Consider the situation in \cref{set:NNN} and suppose that the 
widths $w_i$ and the activation functions $\sigma_i$
are such that \cref{lemma:heavisideapprox}, \cref{theorem:generalactiv}, or \cref{cor:approxpropacti}
can be applied to~$\psi$. Assume that 
the functions $\sigma_i$, $i=1,...,L$, are bounded on bounded sets and that there exists a  $j \in \{1,...,L\}$
such that $\sigma_j$ is constant on an open nonempty interval $I_j \subset \R$. Define
\begin{equation*}
V := 
\left \{
\{\zz_k\}_{k=1}^n \in Y
\,
\Big |
\,
\zz_k = \zz_l~~\forall k, l \in \{1,...,n\}
\right \}
\end{equation*}
and let $K$ be defined as in \eqref{eq:Kspurdef-2625sf35vs} (with the above $V$). 
Then, \eqref{eq:trainingpropprot5} possesses at least one spurious local minimum
satisfying a growth condition of the form \eqref{eq:quadgrowthstates}
for all $y_d \in K$ and, for every $C>0$, there exist uncountably many $y_d \in K$
such that at least one of the spurious local minima of \eqref{eq:trainingpropprot5} satisfies 
\eqref{eq:errorestimates}, \eqref{eq:lossestimate42}, and \eqref{eq:quadgrowthstates}.
In particular, in the case $\closure_Y(\Psi(D, x_d)) = Y$, the cone $K$ 
is equal to $Y \setminus V$ and \eqref{eq:trainingpropprot5}
possesses spurious local minima for all $y_d$ that cannot be fitted precisely with a constant function.
\end{corollary}

\begin{proof}
This follows completely analogously to the proof of \cref{cor:spurminaffineNN} 
with \cref{lemma:affinesubspace} replaced by \cref{lemma:constantsubspace}. 
\end{proof}

Some remarks regarding the last two results are in order: 
\begin{remark}~ 
\begin{itemize}

\item \Cref{cor:spurminaffineNN} covers
in particular neural networks with {ReLU-,} leaky \mbox{ReLU-,} ISRL-, and ELU-activation functions. 
\Cref{cor:spurminconstantNN} 
applies, for instance, to networks that involve a binary, ReLU-, or SQNL-layer.

\item 
As the proofs of \cref{cor:spurminaffineNN,cor:spurminconstantNN}
(or \cref{prop:existencehotspurs,theorem:badcone}, respectively) are constructive, 
they can also be used to find explicit examples of data sets 
that give rise to spurious local minima in \eqref{eq:trainingpropprot5}.
We do not pursue this approach here to avoid overloading the paper.

\item In the case $\closure_Y(\Psi(D, x_d)) = Y$, i.e., 
in the situation where all vectors are realizable, \cref{cor:spurminaffineNN} 
yields the same result as \cite[Corollary 1]{Ding2020} (albeit under weaker assumptions 
on the network widths $w_i$, $i=1,...,L$). \Cref{cor:spurminaffineNN,cor:spurminconstantNN}
are further similar in nature to \cite[Theorem 1]{Yun2019} where 
the existence of spurious local minima in squared-loss training problems 
for one-hidden-layer neural networks with leaky ReLU activation functions 
is proved for all label vectors that are not affine-linearly fittable. Note that,
in \cref{cor:spurminaffineNN,cor:spurminconstantNN}, we only obtain 
a result of comparable strength
in the case $\closure_Y(\Psi(D, x_d)) = Y$.  If the assumption of realizability is violated,  
then our analysis only yields that there exists an open nonempty cone $K \subset Y$
of label vectors for which the problem \eqref{eq:trainingpropprot5} possesses 
spurious local minima. However,
in contrast to \cite[Theorem 1]{Yun2019},  
\cref{cor:spurminaffineNN,cor:spurminconstantNN}
also cover neural networks with output dimension $d_\yy > 1$,
depth $L>1$, and activation functions $\sigma_i$ that are not positively homogeneous
and additionally also show that the spurious local minima of \eqref{eq:trainingpropprot5} can 
be arbitrarily bad in relative and absolute terms and in terms of loss. 
The statements on the size of the set of label vectors with spurious local minima in
\Cref{cor:spurminaffineNN,cor:spurminconstantNN} are 
thus weaker than that of \cite[Theorem 1]{Yun2019} but our results
are also far more general. In particular, they also cover the 
analytically very challenging and in practice due to mild 
overparameterization frequently appearing situation where the considered network is deep and 
the assumption of realizability is violated (or, alternatively, simply not verifiable).
At least to the best of the author's knowledge, 
results on the existence of spurious local minima of a similar strength and generality 
can currently not be found in the literature.

\item Recall that the number $\Theta(\Psi, x_d)$ is a measure for the worst-case 
approximation error in the situation of \eqref{eq:trainingpropprot5}
and the extent to which the
approximation property \ref{fa:II} is satisfied, see \eqref{eq:randomeq2736352} and \cref{def:thetadef}. 
\Cref{cor:spurminaffineNN,cor:spurminconstantNN} thus imply that,
for neural networks with activation functions that are affine 
on some open nonempty subset of their domain of definition, 
an improved expressiveness of the map $\Psi(\cdot, x_d)\colon D \to Y$ 
necessarily comes at the price of a larger cone $K$ of label vectors $y_d$
that give rise to spurious local minima in \eqref{eq:trainingpropprot5}. 
This shows that there is indeed ``no free lunch'' in the situation of 
\cref{cor:spurminaffineNN,cor:spurminconstantNN}. 

\item If it can be shown that a neural network parameterizes multiple subspaces 
in the sense of \cref{theorem:badcone}, then one can, of course, also invoke this result multiple times.
This then allows to prove that certain $y_d$ give rise to training problems of the form \eqref{eq:trainingpropprot5}
that possess several different spurious local minima. 
Using \cref{prop:existencehotspurs}, it is further easy to also establish results on the existence of spurious local minima 
for training problems of the form \eqref{eq:trainingpropprot5} that involve neural networks whose 
activation functions are polynomial on an open subset of their domain of definition. 
We omit discussing these extensions of our analysis in detail in this paper. 
\end{itemize}
\end{remark}

For networks satisfying the assumptions of \cref{cor:spurminaffineNN,cor:spurminconstantNN}, 
we also have: 
\begin{corollary}{\hspace{-0.05cm}\bf(Nonuniqueness and Instability in the Presence of Realizability)}\label[corollary]{cor:affineNNinstable}
Consider the situation in \cref{set:NNN} and suppose that the 
widths $w_i$ and the activation functions $\sigma_i$
are such that \cref{lemma:heavisideapprox}, \cref{theorem:generalactiv}, or \cref{cor:approxpropacti}
can be applied to~$\psi$. Assume further that $\closure_Y(\Psi(D, x_d)) = Y$ holds and 
that one of the following is true:
\begin{enumerate}[label=\roman*)]
\item  For every $i \in \{1,...,L\}$, 
there exists an open nonempty interval $I_i$ such that 
$\sigma_i$ is affine and non-constant on $I_i$ and  it holds
 $\min(d_\xx, d_\yy)\leq \min(w_1,...,w_L)$ and $n > d_\xx + 1$.
\item The functions $\sigma_i$, $i=1,...,L$, are bounded on bounded sets and there exists an index  $j \in \{1,...,L\}$
such that $\sigma_j$ is constant on an open nonempty interval $I_j$.
\end{enumerate}
Then, the solution map 
\begin{equation*}
 Y \ni
y_d \mapsto \argmin_{\alpha \in D} \|\Psi( \alpha, x_d) - y_d\|_Y^2 \subset D
\end{equation*}
of the training problem \eqref{eq:trainingpropprot5} is discontinuous 
in the sense that there exist uncountably many $y_d \in Y$ 
such that there are an open set $U \subset D$, an  $\bar \alpha \in D$, and a family $\{y_d^s\}_{s > 0} \subset Y$
satisfying $\bar \alpha \in U$, $y_d^s\to y_d$ for $s \to 0$, 
\begin{equation*}
\bar \alpha \in  \argmin_{\alpha \in D} \|\Psi( \alpha, x_d) -  y_d\|_Y^2,
\end{equation*}
and
\begin{equation*}
U \, \cap\, \argmin_{\alpha \in D} \|\Psi( \alpha, x_d) - y_d^s\|_Y^2  = \emptyset\qquad\forall s >0.
\end{equation*}
Further, in the above situation, there exist uncountably many $y_d \in Y$
such that \eqref{eq:trainingpropprot5} is not uniquely solvable 
in the sense that there are an $\bar \alpha \in D$, an open set
$U \subset D$, and 
a family $\{\alpha_s\}_{s > 0}$
satisfying $\bar \alpha \in U$, $\{\alpha_s\}_{s > 0} \subset D \setminus U$, and 
\begin{equation*}
\lim_{s \to 0}  \|\Psi(\alpha_s, x_d) - y_d\|_Y^2  =\|\Psi(\bar \alpha, x_d) - y_d\|_Y^2
= \inf_{\alpha \in D} \|\Psi(\alpha, x_d) - y_d\|_Y^2. 
\end{equation*}
\end{corollary}

\begin{proof}
To establish this corollary, it suffices to combine 
\cref{cor:approxpropacti,lemma:nnconic,lemma:heavisideapprox,theorem:generalactiv,lemma:affinesubspace,lemma:constantsubspace,cor:instabilityoverpara}. 
\end{proof}

It remains to study the consequences that the abstract results on regularized training problems in 
\cref{subsec:5.3} have for the neural networks in \cref{set:NNN}. 
For the sake of simplicity, in what follows, we will restrict our attention to 
regularization terms of the form $\nu \|\cdot\|_p^p$, $p \in [1,2]$, $\nu > 0$,
where $\|\cdot\|_p$ denotes the usual $p$-norm on the Euclidean space 
$\R^m \cong \R^{w_{L+1} \times w_{L}} \times \R^{w_{L+1}} \times ... \times \R^{w_{1} \times w_{0}} \times \R^{w_{1}}$. 
Other regularizers can be treated completely analogously, 
cf.\ the more general setting considered in \cref{theorem:spuriousregprob,theorem:approxgone,theorem:nonuniquereg}.

\begin{corollary}{\hspace{-0.05cm}\bf(Regularized Training Problems for Neural Networks)}\label[corollary]{cor:regNNcrap}
Consider the situation in \cref{set:NNN} and suppose that the 
widths $w_i$ and the activation functions $\sigma_i$
satisfy the conditions in \cref{theorem:generalactiv} (or \cref{cor:approxpropacti}, respectively). 
Assume further that $\frac12(d_\xx + 2)(d_\xx + 1) < n$ holds
and that the functions $\sigma_i$, $i=1,...,L$, are twice differentiable, and  
consider for an arbitrary but fixed $p \in [1,2]$ the regularized squared-loss training problem given by 
\begin{equation}
\label{eq:regNNproblem}
\min_{\alpha \in D} \|\Psi( \alpha, x_d) - y_d\|_Y^2 + \nu \|\alpha\|_p^p
=
\frac{1}{2n}\sum_{k=1}^n  \|\psi(\alpha, \xx_{\;d}^k) - \yy_d^k\|_2^2 + \nu \|\alpha\|_p^p. 
\end{equation}
Then, the following is true:
\begin{enumerate}[label=\roman*)]
\item
For every arbitrary but fixed $C>0$, there exist uncountably many combinations 
of training label vectors $y_d \in Y$ and regularization parameters $\nu > 0$ such that 
the origin
$\bar \alpha = 0 \in \R^m \cong \R^{w_{L+1} \times w_{L}} \times \R^{w_{L+1}} \times ... \times \R^{w_{1} \times w_{0}} \times \R^{w_{1}}$ is a spurious local minimum of \eqref{eq:regNNproblem}
that satisfies a local quadratic growth condition of the form
\begin{equation*}
\|\Psi( \alpha, x_d) - y_d\|_Y^2 + \nu \|\alpha\|_p^p
\geq \|\Psi(\bar \alpha, x_d) - y_d\|_Y^2 + \nu \|\bar \alpha\|_p^p + \varepsilon \|\alpha - \bar \alpha\|_2^2\quad \forall \alpha \in B_r(\bar \alpha)
\end{equation*}
for some $\varepsilon, r > 0$ and 
\begin{equation*}
\inf_{\alpha \in D} \|\Psi( \alpha, x_d) - y_d\|_Y^2 + \nu \|\alpha\|_p^p
+ C 
\leq 
 \|\Psi( \bar \alpha, x_d) - y_d\|_Y^2 + \nu \|\bar \alpha\|_p^p.
\end{equation*}
The problem \eqref{eq:regNNproblem} can thus possess arbitrarily bad spurious local minima. 
Moreover, there exists a nonempty open set $O \subset Y \times (0, \infty)$ 
such that \eqref{eq:regNNproblem} possesses at least one spurious local minimum 
for all tuples $(y_d, \nu) \in O$. 
(These minima do not necessarily satisfy a local quadratic growth condition.)
\item 
For every arbitrary but fixed 
regularization parameter $\nu > 0$,  there exist 
uncountably many label vectors $y_d \in Y \setminus \{0\}$ such that $\bar \alpha = 0$ is the 
unique global solution of the problem \eqref{eq:regNNproblem}. Adding the regularization term 
$\nu \|\alpha\|_p^p$ to the objective function of the problem \eqref{eq:trainingpropprot5} thus 
necessarily compromises the approximation property \ref{fa:II}.

\item 
There exist uncountably many combinations of regularization parameters $\nu>0$
and label vectors $y_d \in Y$ such that
there exist an $s_0 \geq 0$, a family $\{y_d^s\}_{s > s_0} \subset Y$, 
and an open neighborhood $U \subset D$ of the origin $\bar \alpha = 0$
satisfying $y_d^s \to y_d$ for $s \to s_0$, 
\begin{equation*}
U  \cap   \argmin_{\alpha \in D} 
 \|\Psi( \alpha, x_d) - y_d^s\|_Y^2 + \nu \|\alpha\|_p^p   = \emptyset\qquad \forall s > s_0,
\end{equation*}
and
\begin{equation*}
\bar \alpha \in \argmin_{\alpha \in D} 
\|\Psi( \alpha, x_d) - y_d \|_Y^2 + \nu \|\alpha\|_p^p.
\end{equation*}
Further, there exist uncountably many tuples $(y_d, \nu) \in Y \times (0, \infty)$ 
such that the set of solutions 
\begin{equation}
\label{eq:randomeq263535-23d3}
\argmin_{\alpha \in D} 
\|\Psi( \alpha, x_d) - y_d \|_Y^2 + \nu \|\alpha\|_p^p
\end{equation}
of the problem \eqref{eq:regNNproblem} contains more than one element. 
The regularized training problem \eqref{eq:regNNproblem} thus possesses the same nonuniqueness 
and instability properties as the optimization problem in \cref{cor:affineNNinstable}.
\end{enumerate}
\end{corollary}
\begin{proof}
From \cref{lemma:nnconic} and \cref{theorem:generalactiv} (or \cref{cor:approxpropacti}, respectively), 
we obtain that \ref{fa:I} and \ref{fa:II} hold. 
Further, the conditions in \cref{ass:notation,ass:linearlowlev} are trivially satisfied in the 
considered situation (up to the isomorphism $\R^m \cong \R^{w_{L+1} \times w_{L}} \times \R^{w_{L+1}} \times ... \times \R^{w_{1} \times w_{0}} \times \R^{w_{1}}$)
with a twice differentiable function $\phi$. 
The various claims of the corollary thus follow immediately from 
\cref{theorem:spuriousregprob,theorem:approxgone,theorem:nonuniquereg}
and \cref{rem:Oneighborhood}.
Note that, as the regularization term in \eqref{eq:regNNproblem} is coercive, 
given a sequence $\{\alpha_s\}_{s > s_0}$
with the properties in the second part of \cref{theorem:nonuniquereg},
we can pass over to a convergent subsequence to obtain that the solution set in \eqref{eq:randomeq263535-23d3} 
indeed contains more than one element.
This completes the proof. 
\end{proof}

Note that \cref{cor:regNNcrap} does not require any assumptions on the existence of unrealizable vectors
or the relationship between $m$ and $n$. 
We conclude this section with a result that illustrates that our analysis 
can also be applied to neural networks which possess an architecture different from 
that in \cref{set:NNN}: 

\begin{corollary}{\hspace{-0.05cm}\bf(Properties \ref{fa:I} and \ref{fa:II} for Residual Neural Networks)}\label[corollary]{cor:resNets}
Suppose that $\XX = \R^{d_\xx}$, $\YY = \R^{d_\yy}$, $d_\xx, d_\yy \in \mathbb{N}$,
$n \in \mathbb{N}$, $x_d = \{\xx_{\;d}^k\}_{k=1}^n$, $L \in \mathbb{N}$, the numbers $w_i \in \mathbb{N}$,
the set $D$, and the functions $\sigma_i\colon \R \to \R$ 
are as in \cref{set:NNN}. 
Suppose further that arbitrary but fixed matrices $E_i \in \R^{w_{i} \times w_{i-1}}$, $i=1,...,L$,
are given, let $\xi_i^{A_i, b_i}\colon  \R^{w_{i-1}} \to \R^{w_i}$ 
be the functions defined by 
\begin{equation*}
\xi_i^{A_i, b_i}(z) := E_i z + \sigma_i\left (A_i z + b_i \right )~\forall i=1,...,L,
\qquad \xi_{L+1}^{A_{L+1}, b_{L+1}}(z) := A_{L+1}z + b_{L+1},
\end{equation*}
where $\sigma_i$ again acts componentwise on the entries of the vectors $ A_{i}z + b_{i}$,
and consider the residual neural network $\psi \colon D \times \XX \to \YY$ defined by 
\begin{equation}
\label{eq:residualNNdef}
\psi(\alpha, \xx) 
:= \left ( \xi_{L+1}^{A_{L+1}, b_{L+1}} \circ ... \circ \xi_{1}^{A_{1}, b_{1}} \right )(\xx)
\end{equation}
for all $\xx \in \XX$ and all $\alpha = (A_{L+1}, b_{L+1},..., A_1, b_1)  \in D$.
Assume that the activation functions $\sigma_i$ satisfy
\begin{equation}
\label{eq:randomeq16352527-37gee}
\lim_{0 < \gamma \to \infty} \frac{1}{\gamma} \sigma_i(\gamma s) =\sigma_i^- \min(0, s) + \sigma_i^+ \max(0, s)
\qquad \forall i=1,...,L\qquad \forall s \in \R
\end{equation}
for some numbers $\sigma_i^-, \sigma_i^+ \in \R$, $i=1,...,L$, with $\sigma_i^- \neq \sigma_i^+$,
and that it holds $w_i \geq 2$ for all $i=2,...,L$ and $w_1 \geq 4$. 
Then, the function
$\Psi(\cdot, x_d)\colon D \to Y$, $\alpha \mapsto \{\psi(\alpha, \xx_{\;d}^k) \}_{k=1}^n$,
associated with 
$x_d$ and the neural network $\psi$ in \eqref{eq:residualNNdef} possesses the properties \ref{fa:I} and \ref{fa:II}.
\end{corollary}

\begin{proof}
The proof of \ref{fa:I} is trivial. To establish \ref{fa:II}, we can proceed 
along similar lines as in the first half of the proof of \cref{theorem:generalactiv}:
From the definitions of the set $D$ and the function $\Psi(\cdot, x_d)\colon D \to Y$,
it follows straightforwardly that, for every arbitrary but fixed parameter vector 
$\alpha = (A_{L+1}, b_{L+1},..., A_1, b_1)\in D$
and all $\gamma > 0$, we have 
\begin{equation*}
\begin{aligned}
&\left \{
\psi
\left 
( \frac{1}{\gamma} A_{L+1}, b_{L+1}, \gamma A_L, \gamma b_{L}, A_{L-1}, b_{L-1},..., A_1, b_1 , \xx_{\;d}^k
\right ) \right \}_{k=1}^n
\\
&=
\Bigg \{
\frac{1}{\gamma} A_{L+1}
\Big [
E_L \big( \xi_{L-1}^{A_{L-1}, b_{L-1}} \circ ... \circ \xi_{1}^{A_{1}, b_{1}} \big)(\xx_{\;d}^k)
\\
&\qquad\qquad
+
\sigma_{L}
\left (
\gamma A_{L} \big(  \xi_{L-1}^{A_{L-1}, b_{L-1}} \circ ... \circ \xi_{1}^{A_{1}, b_{1}} \big)(\xx_{\;d}^k)
+
\gamma b_{L}
\right )
\Big ]
+ b_{L+1}\Bigg \}_{k=1}^n
\in \closure_Y\left (\Psi(D, x_d)\right ). 
\end{aligned}
\end{equation*}
Here, in the borderline case $L=1$, the ``empty'' composition 
$\xi_{L-1}^{A_{L-1}, b_{L-1}} \circ ... \circ \xi_{1}^{A_{1}, b_{1}}$ again has to be interpreted 
as the identity map.
By passing to the limit $0 < \gamma \to \infty$ in the above
and by exploiting \eqref{eq:randomeq16352527-37gee}, we obtain that 
\begin{equation*}
\begin{aligned}
&\Bigg \{
A_{L+1}
\Big [
\tilde \sigma_{L}
\left (
A_{L} \big ( \xi_{L-1}^{A_{L-1}, b_{L-1}} \circ ... \circ \xi_{1}^{A_{1}, b_{1}} \big )(\xx_{\;d}^k)
+
b_{L}
\right )
\Big ]
+ b_{L+1}\Bigg \}_{k=1}^n
\\
&= \Bigg \{ \left (\tilde \varphi_{L+1}^{A_{L+1}, b_{L+1}} \circ \tilde \varphi_{L}^{A_{L}, b_{L}} 
\circ  \xi_{L-1}^{A_{L-1}, b_{L-1}} \circ ... \circ \xi_{1}^{A_{1}, b_{1}} \right )(\xx_{\;d}^k)
\Bigg \}_{k=1}^n
\in \closure_Y\left (\Psi(D, x_d)\right ) 
\end{aligned}
\end{equation*}
holds for all $\alpha = (A_{L+1}, b_{L+1},..., A_1, b_1)\in D$,
where $\tilde \sigma_L$ denotes the ReLU-type 
activation function on the right-hand side of \eqref{eq:randomeq16352527-37gee} for $i=L$, 
i.e., $\tilde \sigma_L(s) := \sigma_L^- \min(0, s) + \sigma_L^+ \max(0, s)$,
and where $\smash{\tilde \varphi_{L}^{A_{L}, b_{L}}}$ and  $\smash{\tilde \varphi_{L+1}^{A_{L+1}, b_{L+1}}}$
are defined as in \eqref{eq:randomeq27353628hd37wb}, i.e., 
\begin{equation*}
\tilde \varphi_L^{A_L, b_L}(z) := \tilde \sigma_L\left (A_L z + b_L \right ),
\qquad \tilde \varphi_{L+1}^{A_{L+1}, b_{L+1}}(z) := A_{L+1}z + b_{L+1}. 
\end{equation*}
Using exactly the same saturation argument as above for the remaining layers 
of the network (starting with the topmost 
unsaturated layer and then proceeding downwards 
and exploiting the continuity of the functions $\tilde \sigma_i(s) := \sigma_i^- \min(0, s) + \sigma_i^+ \max(0, s)$,
$i=1,...,L$) yields that 
\begin{equation*}
\begin{aligned}
\Bigg \{ \left (\tilde \varphi_{L+1}^{A_{L+1}, b_{L+1}} \circ \tilde \varphi_{L}^{A_{L}, b_{L}} 
\circ  ... \circ \tilde \varphi_{1}^{A_{1}, b_{1}} \right )(\xx_{\;d}^k)
\Bigg \}_{k=1}^n
\in \closure_Y\left (\Psi(D, x_d)\right ) 
\end{aligned}
\end{equation*}
holds for all $\alpha = (A_{L+1}, b_{L+1},..., A_1, b_1)\in D$, where $\smash{\tilde \varphi_{i}^{A_{i}, b_{i}}}$, $i=1,...,L+1$,
are the functions in \eqref{eq:randomeq27353628hd37wb} associated with the activations $\tilde \sigma_i$
and where the set $\closure_Y\left (\Psi(D, x_d)\right )$ still refers 
to the original network $\psi$ in \eqref{eq:residualNNdef}.
The above shows that the closure $\closure_Y  (\Psi(D, x_d) )$
has to be at least as big as the set $\closure_Y  (\tilde\Psi(D, x_d)  )$
that is obtained from the function $\smash{\tilde \Psi(\cdot, x_d)\colon D \to Y}$, 
\smash{$\alpha \mapsto \{\tilde \psi(\alpha, \xx_{\;d}^k) \}_{k=1}^n$},
associated with a neural network $\tilde \psi$ 
that possesses the architecture in \cref{set:NNN}
and involves the ReLU-type activation functions $\tilde \sigma_i$, $i=1,...,L$.
Since this network $\tilde \psi$ satisfies \ref{fa:II} by 
our assumptions on the widths $w_i$, $i=1,...,L$, and 
\cref{theorem:generalactiv}, it now follows immediately that 
\begin{equation*}
\min_{y \in  \closure_Y\left (\Psi(D, x_d)\right )} \|y - y_d\|_Y^2 
\leq 
\min_{y \in  \closure_Y\left (\tilde \Psi(D, x_d)\right )} \|y - y_d\|_Y^2 
< \|y_d\|_Y^2\qquad \forall y_d \in Y\setminus \{0\},
\end{equation*}
and, by \cref{lemma:reformulatedcondition},
that the map $\Psi(\cdot, x_d)\colon D \to Y$ 
indeed possesses the property \ref{fa:II}.
This completes the proof. 
\end{proof}

It is easy to check that the last result applies in particular to residual neural networks 
of the type \eqref{eq:residualNNdef} that involve an arbitrary mixture of the 
activation functions in point \ref{item:activationex}  of \cref{cor:approxpropacti}.
For more details on ResNets, 
see  \cite{He2016}. 
We remark that, with \cref{cor:resNets} at hand,
one can again use the abstract analysis of \cref{sec:5}
to obtain results analogous to 
\cref{cor:NNbestapproxinstabil,cor:NNinfiniteNonUniqueness,cor:saddlenonover,cor:saddleoverpar,cor:spurminaffineNN,cor:spurminconstantNN,cor:affineNNinstable,cor:regNNcrap}
for the networks in \eqref{eq:residualNNdef}. We do not state these here for the sake of brevity. 
Note further that the technique used in the proof of \cref{cor:resNets} (i.e., the idea to establish \ref{fa:II} by 
saturating the activation functions and by subsequently invoking \cref{lemma:heavisideapprox} or \cref{theorem:generalactiv})
also works for other architectures. Once the properties \ref{fa:I} and \ref{fa:II} are established,
one can then again apply the theoretical machinery developed in \cref{sec:5} to the network under consideration. 
This flexibility is the main advantage of the general, axiomatic approach 
that we have taken in \cref{sec:5}. 

\section{Concluding Remarks}
\label{sec:7}
~\\[-0.575cm]
We conclude this paper with some additional remarks: 

First, we would like to stress that, although the results proved in the previous sections 
paint a somewhat bleak picture of the 
optimization landscape and the stability properties 
of squared-loss training problems for neural networks and general
nonlinear conic approximation schemes,
one should keep in mind that, even when 
applying an optimization algorithm to a problem of the form 
\eqref{eq:trainingpropprot} only allows to determine a spurious
local minimum or a saddle point (which may very well happen as
we have seen, for instance, in \cref{cor:spurminaffineNN}),
this resulting point may still perform far better, e.g., in terms of loss 
than anything that is obtainable with a classical approximation approach. 
The fact that driving the value of the objective function of \eqref{eq:trainingpropprot} 
to the global optimum may, in practice, not be possible due to spurious local minima 
or the instability effects discussed in \cref{sec:5,sec:6} thus does not mean 
that trying to solve problems of the type \eqref{eq:trainingpropprot} is not sensible 
(in particular as the results obtained, for instance, with stochastic gradient descent methods 
often turn out to be remarkably good in applications). The main issue that 
arises from the observations made in \cref{sec:5,sec:6}
is more one of reliability and robustness. As solving problems of the type \eqref{eq:trainingpropprot}
numerically may only provide good or locally optimal choices of the parameter $\alpha \in D$
but not globally optimal ones and since points with similar optimal or nearly optimal loss values may 
perform very differently even on the training data (see point \ref{item:stabth:ii} of \cref{theorem:abstractinstability}),
theoretical guarantees for, e.g., the generalization behavior 
or approximation properties of global minimizers of problems of the form \eqref{eq:trainingpropprot} 
may simply not apply to the points that are determined with 
optimization algorithms in reality. 
Further, due to the instability and nonuniqueness effects documented, e.g.,
in \cref{theorem:abstractinstability,cor:instabilityoverpara,theorem:nonuniquereg},
small perturbations of the training data or the hyper-parameters of the considered numerical solution method 
and/or a different behavior of stochastic components of the used optimization algorithm 
may affect the performance of the obtained solutions significantly, cf.\
the 
experiments of \cite{Cunningham2000}. 
Note that this implies in particular that general deterministic guarantees for 
the convergence of optimization algorithms etc.\ are simply unobtainable
and that only probabilistic approaches have the potential to explain in a satisfying way why neural networks perform the way they do.
We leave the study of the latter and their connections to the results of this paper for future research. 

We would like to point out that the observation that undesirable properties of the optimization landscape 
may prevent a proper identification of those parameters $\alpha$ for which, e.g., 
a neural network provides the best approximation properties in a particular situation 
also suggests that one should be careful with claims that nonlinear approximation instruments are able 
to break the curse of dimensionality. 
The main point here is that this curse may not only 
manifest itself in the fact that the number of operations or degrees of freedom
in an approximation scheme has to grow exponentially with, for instance, 
the spatial dimension of an underlying PDE to achieve a certain prescribed precision, but also in the 
loss surface of the minimization problems 
that have to be solved in order to adapt an approximation instrument to a given function. 
Compare, e.g., with \cref{cor:spurminaffineNN,cor:spurminconstantNN}
in this context which demonstrate that improved approximation properties are necessarily paid for in the 
form of a larger set of label vectors for which \eqref{eq:trainingpropprot} possesses spurious local minima
when ReLU-type neural networks are considered. 
To see the essential problem, one can also consider the extreme 
case of a continuous function $\Psi\colon \R \to Y$ from the real line into a (not necessarily finite-dimensional)
Hilbert space $(Y, \|\cdot\|_Y)$ whose image $\Psi(\R)$ is dense in $Y$ (i.e., a space-filling curve). 
Such an approximation scheme only requires one parameter to approximate 
arbitrary elements of $Y$ to an arbitrary precision and thus clearly does not suffer from 
the scaling behavior that classically characterizes the curse of dimensionality. 
However, this construction certainly does not break this curse, either,
simply because, for a high- or infinite-dimensional space $Y$, the
optimization landscape of the problem $\min_{\alpha \in \R} \|\Psi(\alpha) - y_d\|_Y^2$ for a given $y_d \in Y$
typically contains countably many spurious local minima 
and can thus not be effectively navigated with classical optimization algorithms
so that identifying parameters $\alpha \in \R$ for which the error 
$ \|\Psi(\alpha) - y_d\|_Y$ becomes small is in practice impossible. 
The results proved in \cref{sec:5,sec:6} suggest that 
it makes sense to interpret nonlinear approximation instruments 
like neural networks as elements of a spectrum which, at the
one end, has linear approximation schemes (which suffer 
from the usual scaling problems related to the curse of dimensionality 
but also give rise to, e.g., squared-loss problems with the best possible optimization landscape)
and, at the other end, has space-filling curves (which only require a single parameter 
to achieve an arbitrary approximation accuracy but also give 
rise to optimization problems which typically have the worst properties possible). 
Considering only the scaling behavior of the degrees of freedom 
w.r.t.\ an underlying dimension without taking into account the effort necessary 
to determine best approximating elements does not seem to be 
sensible when studying how neural networks and nonlinear approximation schemes
in general behave in view of the curse of dimensionality. 
We remark that this impression is also 
confirmed by the results on the computational complexity  
of training problems available in the literature. 
Compare, e.g., with \cite{Blum1992},
which, for a 2-layer-3-node neural network, show that, for any 
polynomial-time training algorithm, there exist choices of the training data
such that the network is not trained correctly,
and that
it is NP-complete to decide whether there exist network parameters 
such that the training data are fit precisely. 
For further contributions on this topic, 
see also \cite{Gupta1995,Bartlett2002,Sima2002} and \cite{Abrahamsen2021}. 
The findings of these papers indicate that, for the improved approximation properties of 
neural networks, one necessarily pays in terms of NP-hardness or 
ER-completeness of the associated training problems, too.\\[-0.6cm]

Regarding the optimization landscape of the
squared-loss training problems in  \eqref{eq:trainingpropprot}, we finally would like to  point out 
that, if $y_d$ is not a label vector that gives rise to a
problem with a spurious local minimum or a saddle point,
but close to a vector that does,
then the objective function of \eqref{eq:trainingpropprot} 
will still possess points which are almost stationary since the
gradients (or subgradients, respectively) of the function $\alpha \mapsto \|\Psi( \alpha, x_d) - y_d\|_Y^2$
depend continuously on $y_d$. The presence of flat regions in 
the optimization landscape that slow down gradient descent or may falsely indicate convergence
is thus also to be expected for label vectors 
that are not directly covered by \cref{th:stevaluedspurious,theorem:existencestatpts,prop:existencehotspurs,theorem:badcone}.
We remark that these predictions of our analysis again agree well with what is observed 
in the numerical practice \citep[cf.\,][]{Dauphin2014}. 
Similarly, points at which the best approximation map of \eqref{eq:trainingpropprot}
is unstable may also already affect the convergence behavior of numerical solution algorithms when the iterates 
of the algorithm come close to them. 
Compare, for instance, with \cite[Section 3]{Wolfe1975} for an example which illustrates that,
even if a function only possesses a statistically negligible set of ``bad'' points, 
these points may still prevent the convergence of gradient descent algorithms on a global level. 

\acks{This research was conducted within the International Research Training Group IGDK 1754,
funded by the German Science Foundation (DFG) and the Austrian Science Fund (FWF)
under project number 188264188/GRK1754.}

\vskip 0.2in
\bibliography{references}

\end{document}